%% file: ms.tex
\documentclass[]{siamonline1116}

\usepackage{microtype} 
\usepackage{lmodern} 

\usepackage{amsmath,amsfonts,amssymb}
\usepackage{graphicx}
\usepackage{epstopdf} 
\usepackage{flafter} 

\usepackage{subcaption} 
\usepackage{placeins} 
\usepackage{siunitx} 
\usepackage{algorithm, algpseudocode} 
\usepackage{xcolor} 

\usepackage[margin=1.3in, top=0.95in, bottom=0.95in]{geometry} 

\usepackage{tikz, pgfplots} 
\usetikzlibrary{backgrounds}

\newcommand{\bs}[1]{\boldsymbol{#1}} 
\newcommand{\hl}[1]{\textit{#1}} 

\renewcommand{\d}[1]{\ensuremath{\operatorname{d}\!{#1}}}
\newcommand{\de}[1]{\ensuremath{\operatorname{\delta}\!{#1}}}
\newcommand{\diff}[2]{\ensuremath{\frac{\d{#1}}{\d{#2}}}}
\newcommand{\diffp}[2]{\ensuremath{\frac{\partial{#1}}{\partial{#2}}}}
\newcommand{\boundary}[1]{\partial\!{#1}}
\newcommand{\opt}[1]{\ensuremath{\bar{{#1}}}}

\newcommand{\rJ}{{\tilde{J}}} 
\newcommand{\hess}[1]{ \mbox{Hess } #1}
\newcommand{\und}{.} 

\newcommand{\mean}[1]{\ensuremath{\mathbb{E}[{#1}]}}
\newcommand{\meanBig}[1]{\ensuremath{\mathbb{E}\Big[{#1}\Big]}}
\newcommand{\var}[1]{\ensuremath{\mathbb{V}[{#1}]}}
\newcommand{\std}[1]{\ensuremath{\mathbb{S}[{#1}]}}
\newcommand{\cov}[2]{\ensuremath{\text{Cov}[{#1},{#2}]}}

\newcommand{\com}[1]{\ignorespaces}
\renewcommand{\com}[1]{{\color{blue}{<<#1>>}}} 

\newcommand{\nes}[1]{\ignorespaces}

\newcommand{\old}[1]{\ignorespaces}

\newtheorem*{remark}{Remark}

\newlength\figureheight
\newlength\figurewidth

\newcommand{\sethw}[2]{
	\setlength\figureheight{#1\textwidth}
	\setlength\figurewidth{#2\textwidth}	
}

\usepackage{lipsum}
\usepackage{amsfonts}
\usepackage{graphicx}
\usepackage{epstopdf}
\usepackage{amsopn}
\ifpdf
\DeclareGraphicsExtensions{.eps,.pdf,.png,.jpg}
\else
\DeclareGraphicsExtensions{.eps}
\fi

\numberwithin{theorem}{section}

\newcommand{\TheTitle}{Robust Optimization of PDEs with Random Coefficients Using a Multilevel Monte Carlo Method} 
\newcommand{\TheShortTitle}{Robust Optimization of PDEs Using Multilevel Monte Carlo} 
\newcommand{\TheAuthors}{A. Van Barel \and S. Vandewalle}

\headers{\TheShortTitle}{\TheAuthors}

\title{{\TheTitle}\thanks{Submitted to the editors \today.
	\funding{This research was funded by project IWT/SBO EUFORIA: ``Efficient Uncertainty quantification For Optimization in Robust desgn of Industrial Applications" (IWT-140068) of the Agency for Innovation by Science and Technology, Flanders, Belgium. Andreas Van Barel is funded by a PhD fellowship of the Research Foundation - Flanders.
}}}

\author{
	Andreas Van Barel\thanks{Department of Computer Science, KU Leuven, Celestijnenlaan 200A, 3001 Heverlee, Belgium
		(\email{andreas.vanbarel@cs.kuleuven.be}, \email{stefan.vandewalle@cs.kuleuven.be}).}
	\and
	Stefan Vandewalle\footnotemark[2]
	}

\ifpdf
\hypersetup{
  pdftitle={\TheTitle},
  pdfauthor={\TheAuthors}
}
\fi

\begin{document}

\maketitle

\begin{abstract}
  This paper addresses optimization problems constrained by partial differential equations with uncertain coefficients. In particular, the robust control problem and the average control problem are considered for a tracking type cost functional with an additional penalty on the variance of the state.
  The expressions for the gradient and Hessian corresponding to either problem contain expected value operators. Due to the large number of uncertainties considered in our model, we suggest to evaluate these expectations using a multilevel Monte Carlo (MLMC) method. Under mild assumptions, it is shown that this results in the gradient and Hessian corresponding to the MLMC estimator of the original cost functional. Furthermore, we show that the use of certain correlated samples yields a reduction in the total number of samples required.
  Two optimization methods are investigated: the nonlinear conjugate gradient method and the Newton method. For both, a specific algorithm is provided that dynamically decides which and how many samples should be taken in each iteration. The cost of the optimization up to some specified tolerance $\tau$ is shown to be proportional to the cost of a gradient evaluation with requested root mean square error $\tau$.
  The algorithms are tested on a model elliptic diffusion problem with lognormal diffusion coefficient. An additional nonlinear term is also considered.
\end{abstract}

\begin{keywords}
	Robust optimization, stochastic PDEs, multilevel Monte Carlo, optimal control, uncertainty, gradient, Hessian
\end{keywords}

\begin{AMS}
	35Q93, 65C05, 65K10, 49M05, 49M15
\end{AMS}

\section{Introduction}
We consider the optimization of a tracking type cost functional constrained by a partial differential equation (PDE) containing uncertain coefficients. The goal is to find an optimum that is satisfactory in a broad parameter range, and that is as insensitive as possible to parameter uncertainties. To that end we solve the so-called \hl{robust control problem}, in which the expected value of the cost functional is optimized. 
Other problem formulations that take into account the uncertainties can be found in \cite{Borzi2010, Borzi2012, Ali2016, kouri2016risk, lee2017comparison}. They differ mainly in computational cost and in the robustness of the obtained optimum.
Several techniques to solve the robust control problem have been described previously, in particular, stochastic collocation methods \cite{Rosseel2012, borzi2009VWmultigrid, tiesler2012stochastic, chen2014weighted, borzi2011pod} and stochastic Galerkin schemes \cite{Rosseel2012, lee2013}. These are based on earlier methods for simulation problems \cite{babuvska2010stochastic, babuska2004galerkin, xiu2003modeling, xiu2005high, nobile2008sparse}.
These methods are mainly used for relatively small stochastic dimensions, because the amount of collocation points increases rapidly with the dimension. Furthermore, Galerkin schemes may run into memory problems.
Many techniques sample the problem in some way and use a multigrid solver on the resulting equations. This effectively comes down to taking the same number of samples on all levels in the multigrid hierarchy. Fundamentally different is the method proposed by Kouri \cite{Kouri2014}, in which the multigrid optimization (MG/OPT) framework \cite{nash2000, lewis2005} is applied to a hierarchy of stochastic discretizations. `Finer' levels correspond to taking a larger number of sample points in the stochastic space. Finally, Newton methods have also been applied successfully to stochastic problems, see, e.g., \cite{martin2012stochastic}.

The computation of the gradient and the Hessian vector product corresponding to the robust optimization problem entail the solution of a system of PDEs with uncertain coefficients that contain expected value operators. Due to the large number of uncertainties considered in our work, we propose to evaluate these expected values using a multilevel Monte Carlo (MLMC) method. This is motivated by the recent developments in MLMC methods for the simulation of (elliptic) PDEs with uncertain coefficients \cite{cliffe2011, teckentrup2013further, giles2015}. The MLMC method reduces the computational cost by taking most samples on coarse grids, and refining the resulting estimate using fewer samples on finer grids. This idea is mainly responsible for the substantial performance increase of our method w.r.t. methods that implicitly take the same number of samples on every grid.
Recently, a MLMC method was proposed to solve the pathwise control problem \cite{Ali2016}, which consists of calculating the average of many optimal control solutions for different realizations of the PDE constraints. However, the resulting control is not guaranteed to be robust.

The method described in this paper solves the robust control problem. It retains the positive aspects of some of the previously described methods while avoiding some of the drawbacks. In particular, our method uses a different number of samples on different spatial discretization levels, it limits memory use by only storing a few samples of the state at any given time, and it reduces cost by adapting the precision (and thus the amount of samples) to the current stage of the optimization process, see also \cite{kouri2013trust}. Furthermore, the method dynamically choses the number of samples such that a solution satisfying a requested tolerance on the gradient norm of the original (unsampled) problem can be obtained. The method can also deal with an additional cost functional term for the variance on the state, as in \cite{Rosseel2012}. The method is especially suited for a large number of stochastic dimensions. 
If the samples are carefully taken, the resulting calculated gradient and Hessian are shown to be exact for some cost functional. Under mild conditions, this cost functional is equal to the one calculated using MLMC. Furthermore, we demonstrate that it is possible to have cheaper samples if correlated samples are allowed. This requires a slight extension of the classic MLMC theory. For the problems in this paper, the effect of the correlations is such that less samples are required.  


The paper is structured as follows. 
In Section 2, it is shown that the robust control formulation and the average control formulation are essentially equivalent for the tracking type cost functional with additional variance term. 
Section 3 introduces the model PDE, describes the properties of the stochastic variables and explains how they are sampled. Expressions for the gradient and Hessian are derived in Section 4. The proposed optimization methods follow the so-called reduced approach, i.e., the state is eliminated. Because the state is stochastic, the alternative would imply storing all realizations of the state in memory, which we want to avoid. 
Section 5 summarizes the existing MLMC theory and provides details on how function valued quantities of interest can be dealt with. Section 6 applies the MLMC method on the equations derived in Section 4. 
Section 7 investigates two optimization methods: the gradient based nonlinear conjugate gradient (NCG) method and the Newton method. For both, a specific algorithm is provided that dynamically decides which and how many samples should be used in each iteration. The cost of the optimization up to some specified tolerance $\tau$ on the gradient norm is shown to be proportional to the cost of a gradient evaluation with requested root mean square error (RMSE) $\tau$. 
The algorithms are tested on a model elliptic diffusion problem with lognormal diffusion coefficient in Section 8. An additional nonlinear term is also considered. 
Finally, we end with some concluding remarks in Section 9.



\section{Cost functional}
\newcommand{\normd}[1]{\ensuremath{\|{#1}\|}}
\newcommand{\Jrob}{\ensuremath{J_{\text{rob}}}}
\newcommand{\Jav}{\ensuremath{J_{\text{av}}}}
\newcommand{\stvar}{\ensuremath{k}} 
Let $(\Omega, \mathcal{A}, \mu)$ denote a probability space. The sample space $\Omega$ contains all possible realizations $\omega$ of the random influence. Its dimension is the stochastic dimension of the problem and may be infinite. \nes{as will be the case in this paper.} $\mathcal{A}$ is the set of all events (subsets of $\Omega$) and $\mu$ is a measure that maps events in $\mathcal{A}$ to probabilities in $[0,1]$. 
The expected value operator, the variance operator and the standard deviation operator of a stochastic variable $k$ are denoted as follows
\begin{align*}\mean{\stvar} = \int_\Omega \stvar \d{\mu(\omega)}, \quad\; \var{\stvar} = \mean{(\stvar - \mean{\stvar})^2} = \mean{\stvar^2} - \mean{\stvar}^2, \quad\; \std{\stvar} = \sqrt{\var{\stvar}}.\end{align*}

Assume a spatial domain $D \subset \mathbb{R}^d$ on which the state $y$, some target state $y_D$ and the control $u$ are defined.
In this paper, we consider the stochastic equivalent to the following classical deterministic goal function of tracking type
\begin{equation}
J_{\text{det}}(y,u) = \normd{y-y_D}^2 + \alpha\normd{u}^2. \label{eq:Jyu}
\end{equation}
The norm $\normd{\und}$ denotes the $L^2$-norm in $D$ induced by the classical inner product $(.,.)$ in $L^2(D)$. 
Consider now the case where, due to uncertainties in the state equations, $y$ is stochastic. The cost functional can then be made deterministic again in several ways \cite{Ali2016, Borzi2012}. The robust control problem attempts to minimize the mean of the cost functional, yielding
\begin{equation}\Jrob(y,u) = \mean{\normd{y-y_D}^2}  + \gamma\normd{\std{y}}^2+ \alpha\normd{u}^2\label{eq:Jyu_2}.\end{equation}
The term $\normd{\std{y}}^2 = \int_D \var{y} \d{x}$ was added because it is desirable to have a control for which the state is more accurately known, leading to a risk averse optimum. Note that the first term minimizes the expected distance to the target function $y_D$, which is not the same as minimizing the distance of the expected state to the target function. The latter is called the \hl{average control} cost functional
\begin{equation}\Jav(y,u) = \normd{\mean{y} - y_D}^2  + \gamma'\normd{\std{y}}^2 + \alpha\normd{u}^2\label{eq:Jyu_3}.\end{equation}
Both cost functionals can easily be shown to be convex. Moreover, we can prove that both are essentially equivalent.
\begin{theorem}[Equivalence of robust and average control] 
	\ \\
	Assume $\normd{\std{y}} \neq 0$, then $\Jrob = \Jav$ if and only if $\gamma' = 1+\gamma$. 
	\label{theorem:equivalence}
\end{theorem}
\begin{proof}
	By switching the order of integration, we have
	\begin{equation*}
	\normd{\std{y}}^2 = \int_D \sqrt{\mean{(y - \mean{y})^2}}^2\d{x} = \int_D \int_{\Omega} (y - \mean{y})^2 \d{\mu(\omega)}\d{x} = \mean{\normd{y-\mean{y}}^2}.
	\end{equation*}
	We can now write
	\begin{align*}
	\mean{\normd{y-y_D}^2} &= \mean{\normd{\mean{y} - y_D + y - \mean{y}}^2} \\
	&= \mean{\normd{\mean{y} - y_D}^2} + \mean{\normd{y - \mean{y}}^2} + \mean{2(\mean{y} - y_D, y - \mean{y})} \\
	&= \normd{\mean{y} - y_D}^2 + \normd{\std{y}}^2.
	\end{align*}
	The quantity $\mean{y} - y_D$ is deterministic. Hence, the last term drops out because $\mean{y-\mean{y}} = 0$.
	It is now clear that
	\begin{equation*}
	\Jrob(y,u) = \normd{\mean{y} - y_D}^2 + (1+\gamma)\normd{\std{y}}^2 + \alpha\normd{u}^2  = \Jav(y,u)
	\end{equation*}
	if and only if $\gamma' = 1+\gamma$. 
\end{proof}
In \cite{Rosseel2012} both robust and average control cost functionals are considered. Theorem \ref{theorem:equivalence} explains why two seemingly different problems produced the same result\footnote{\cite{Rosseel2012}: pp.~18, table 1, first and fourth problem under `unknown mean control'.}. The robust control cost functional (\ref{eq:Jyu_2}) will be denoted simply as $J$ in the remainder of this paper. 

\section{Model problem PDE constraint}
The method that we shall propose does not assume any specific PDE. However, to make matters more concrete and to simplify some expressions, we will focus our exposition on one particular model. Consider an object occupying the spatial domain $D = [0,1]^d \subset \mathbb{R}^d$ and denote its boundary by $\boundary{D}$. 
The temperature distribution on $D$ constitutes the state $y$. The control $u$ is a heat source (or sink) on $D$ which we assume to be constant in time. The heat conduction coefficient is the stochastic field $k: D\times \Omega \rightarrow \mathbb{R}:(x,\omega) \mapsto k(x,\omega)$. With Dirichlet boundary conditions, the system equations are now described by the following PDE with random coefficients:
\begin{align}
\begin{aligned}
-\nabla\cdot(k(x,\omega)\nabla y(x,\omega))& = \beta(x)u(x) && \mbox{on } D \\
y(x,\omega) &= 0 && \mbox{on } \boundary{D}.
\end{aligned}
\label{model_SPDE}
\end{align} 
The coefficient $\beta(x)$ allows to constrain the control input to a subset of $D$, by setting it to 1 if $x$ is in the subset and 0 otherwise, see e.g., \cite{troltzsch2010}.
The variables belong to the function spaces:
\vspace{-0.3cm}
\begin{equation*}
u \in L^2(D),\; y \in H^1_0(D) \otimes L^2(\Omega),\; k \in L^\infty_+(D) \otimes L^2(\Omega),\; \beta \in L^\infty(D). \label{eq:function_spaces}
\end{equation*}
The symbol $\otimes$ denotes the tensor product of Hilbert spaces.
The subscript $+$ indicates the subset of functions that are positive almost everywhere.

\subsection{Stochastic field $k$}
\newcommand{\sfield}{z}
\newcommand{\nkl}{ {n_{\text{KL}} } }
We assume a \hl{lognormal field} $k(x,\omega) \triangleq \exp(z(x,\omega))$, with $z$ a Gaussian field.  
We take $\mean{z} = 0$ and use the common assumption of an exponential covariance \cite{cliffe2011, graham2011}
\begin{equation}
C_z(x,y) = \cov{\sfield(x,\omega)}{\sfield(y,\omega)} = \sigma^2 \exp\Big(-\frac{\normd{x-y}_1}{\lambda}\Big),
\label{eq:KL_covariance_Gaussian}
\end{equation} 
with $\sigma^2$ the variance of the field and $\lambda$ the correlation length. 
Samples of $z$ can be generated starting from the Karhunen-Lo\`eve (KL) expansion \cite{loeve1946fonctions, karhunen1947ueber} of $z$:
\begin{equation}\sfield(x,\omega) = \mean{\sfield(x,\omega)} + \sum\limits_{n=1}^{\infty} \sqrt{\theta_n} \xi_n(\omega) f_n(x), \label{eq:KLexpansion}\end{equation}
see, e.g., \cite{borzi2009VWmultigrid, borzi2011pod, cliffe2011, graham2011, chen2014weighted}. The KL-expansion is the unique expansion of the above form that minimizes the total mean squared error if the expansion is truncated to a fixed finite number of terms 
\cite{ghanem2003}.
In this paper we confine ourselves to the choice $\lambda = 0.3$ and choose $\nkl = 500$ terms, capturing $94\%$ of the variance for a 2D problem. A typical realization for two values of $\sigma$ is found in Figure \ref{fig:2Dsamples}.
\begin{figure}[h]
	\centering
	\begin{subfigure}{.45\textwidth}
		\centering
		\sethw{0.5}{0.5}
		\input{fig/sample2D_var01.tex}
	\end{subfigure}
	\begin{subfigure}{.45\textwidth}
		\centering
		\sethw{0.5}{0.5}
		\input{fig/sample2D_var05.tex}
	\end{subfigure}
	\caption{Realizations of the lognormal field $k$ for the 2D case with $\lambda = 0.3$, $\nkl = 500$. Left: $\sigma^2 = 0.1$. Right: $\sigma^2 = 0.5$.}
	\label{fig:2Dsamples}
\end{figure}
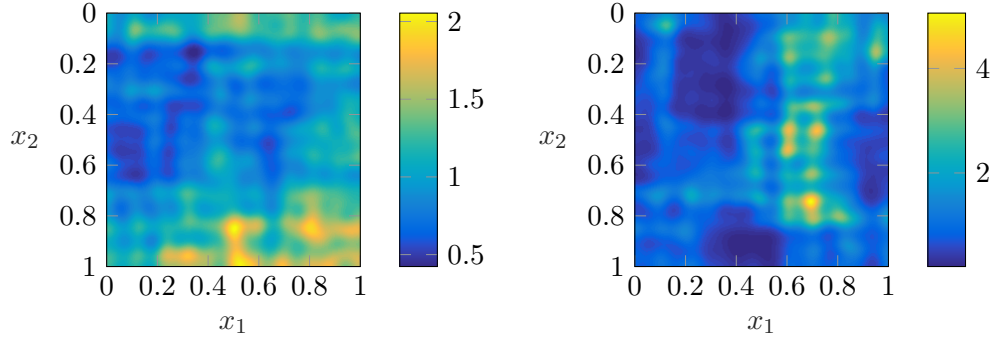

The accurate and efficient generation of samples is not the main topic of this paper. Other sampling techniques such as circulant embedding \cite{dietrich1997,	graham2011} may provide some computational advantages over the KL-expansion. 

\section{Optimality Conditions}
\label{sec:lagr/stochastic case}
\newcommand{\lm}{p}
\newcommand{\ip}[1]{(#1)_{D,\Omega}}
\newcommand{\norm}[1]{\|#1\|_{D,\Omega}}
This section derives the optimality conditions for the model problem. The constraint is denoted here by $c(y,u) = \nabla \cdot (k \nabla y)+\beta u = 0$, without explicit dependence on $k$ or $\omega$. It provides a relation between $y\in H^1_0(D) \otimes L^2(\Omega)$ and $u\in L^2(D)$ for all realizations of $\omega$. All inputs $u$ are assumed to be admissible, i.e., for every $u$, $c(y,u)=0$ can be uniquely solved for $y$.

\subsection{General expressions}
The optimality conditions can be derived starting from the Lagrangian
$$\mathcal{L}(y,u,\lm) = J(y,u) + \ip{\lm,c(y,u)},$$
with $\lm \in H^1(D) \otimes L^2(\Omega)$ a Lagrange multiplier and $\ip{.,.}$ the standard inner product in $L^2(D) \otimes L^2(\Omega)$.
The necessary first order conditions for optimality are then found by setting the partial derivatives to $p$, $y$ and $u$ to zero:
\begin{equation}
\left\{
\begin{aligned}
0 &= \nabla_\lm \mathcal{L} = c(y,u) \\
0 &= \nabla_y \mathcal{L} = \nabla_y J + \Big(\diffp{c}{y}\Big)^*[\lm] \\
0 &= \nabla_u \mathcal{L} = \nabla_u J + \Big(\diffp{c}{u}\Big)^*[\lm]
\end{aligned}
\right.
\label{eq:optcond_lagr_general}
\end{equation}
The superscript $*$ denotes the adjoint of a bounded linear operator.
The expression $(\diffp{c}{y})^*[\lm]$, for example, follows through the Riesz representation theorem \cite{peypouquet2015} from
$$\diffp{}{y}\ip{\lm,c}[h]  = \ip{\lm, \diffp{c}{y} [h]} = \ip{\Big(\diffp{c}{y}\Big)^*[\lm], h}.$$

\subsection{Reduced gradient for the model problem}
The robust optimization objective can be reformulated in a more compact manner in terms of the norm $\norm{.}$ induced by $\ip{.,.}$.
\begin{align*}
J(y,u) &= \mean{\|y-y_D\|_D^2} + \gamma\|\std{y}\|_D^2 + \alpha\|u\|_D^2 \\
&= \int_{\Omega} \int_D (y-y_D)^2 \d{x} \d{\mu(\omega)} + \gamma \int_D \int_{\Omega} (y - \mean{y})^2 \d{\mu(\omega)}\d{x} + \alpha \int_D u^2 \d{x} \\
&= \norm{y-y_D}^2 + \gamma \norm{y-\mean{y}}^2 + \alpha \|u\|_D^2.
\end{align*}
\nes{It is clear that in general $\|\sqrt{\mean{(\und)^2}}\|^2_D = \norm{\und}^2 = \mean{\|\und\|^2_D}$.} 
The terms $\nabla_y J$ and $\nabla_u J$ in (\ref{eq:optcond_lagr_general}) can be evaluated by noting that
\begin{align}
\begin{aligned}
\diff{}{y} \norm{y-\mean{y}}^2 [h] & = \ip{2(y-\mean{y}), \diff{(y-\mean{y})}{y}[h]} \\
&=\ip{2(y-\mean{y}), h-\mean{h}} \\
&= \ip{2(y-\mean{y}), h}.
\end{aligned}
\label{eq:lagr_deriv_5}
\end{align}
In the last step we used that $\ip{y-\mean{y}, \mean{h}} = (\mean{y} - \mean{y}, \mean{h})_D = 0.$ 
Hence we find $\nabla_y J = 2(y-y_D) + 2\gamma(y-\mean{y})$. Similarly,  $\nabla_u J = 2\alpha u$.
Since the operator $\nabla \cdot (k \nabla .)$ is linear and self-adjoint, $\nabla_y (\lm,c) = \nabla \cdot(k \nabla \lm)$. 
Finally, for the third equation in (\ref{eq:optcond_lagr_general}) we find
\begin{equation}
\diffp{}{u} \ip{\lm,c(y,u)}[h] =\ip{\lm,\beta h} =\ip{\beta\lm,h}.\label{eq:lagr_deriv_4} 
\end{equation}
Since $u \in L^2(D)$, this equality must hold for all $h \in L^2(D)$ (as opposed to $L^2(D) \otimes L^2(\Omega)$). Since $\ip{\beta\lm,h} = (\beta \mean{p}, h)_D$, we find $\nabla_u \ip{\lm,c} = \beta \mean{p}$.
Combining the results, the system of equations (\ref{eq:optcond_lagr_general}) reduces to
\begin{equation}
\left\{\begin{array}{rcll}
-\nabla \cdot(k\nabla y)& = &\beta u &\quad \mbox{on } D\\
-\nabla \cdot(k\nabla p)& = &2(y -y_D) + 2\gamma(y-\mean{y}) &\quad \mbox{on } D\\
\nabla \rJ(u) & = & 2\alpha u + \beta \mean{p} = 0 
\end{array}\right.\label{eq:optconds_lagr}
\end{equation}
The Dirichlet boundary conditions are omitted for brevity. The last equation provides the so-called \hl{reduced gradient}, i.e., the gradient of the \hl{reduced cost functional} $\rJ(u) = J(Su,u)$ where $S$ solves the constraint $c$. 

\begin{remark}
	The $\mean{y}$ term in the second equation of (\ref{eq:optconds_lagr}) causes a more intricate connection between the values of $y$ and $p$ for the different instances governed by $\omega$. This essentially bars one from deriving the conditions for each $\omega$ separately and joining them in the third equation through an expected value, as is often done in the case $\gamma = 0$.
\end{remark}

\subsection{Reduced Hessian for the model problem} \label{sec:opt/hessian}
Consider the second order derivative of a functional $f$ and apply some calculus to obtain
\begin{align}\diff{^2f(u)}{u^2}[h_1,h_2] = \diff{}{u}\Big(\diff{f(u)}{u}[h_2]\Big)[h_1] = \diff{}{u}\Big((\nabla f(u),h_2)\Big)[h_1] = (\diff{\nabla f(u)}{u}[h_1],h_2).\label{eq:Hess}\end{align}
\old{\begin{align*}\diff{^2f(u)}{u^2}[h_1,h_2] &= \diff{}{u}\Big(\diff{f(u)}{u}[h_2]\Big)[h_1] .\end{align*}
	Some calculus yields the following results:
	\begin{align}\diff{^2f(u)}{u^2}[h_1,h_2] &= \diff{}{u}\Big((\nabla f(u),h_2)\Big)[h_1] \nonumber \\
	&= (\diff{\nabla f(u)}{u}[h_1],h_2).
	\end{align}}
The mapping $\diff{\nabla f(u)}{u}[\und], H \rightarrow H$ is the Hessian of $f(u)$, denoted as $\hess{f(u)}[\und]$ \cite{troltzsch2010}.
In the finite dimensional setting, the Hessian can be represented by an ordinary matrix $M$.  For any vector $h_1$ and $h_2$, we have $h_1^TMh_2 = (Mh_1,h_2)_{\mathbb{R}^n}$. The similarity with (\ref{eq:Hess}) should be clear. 
Due to the linearity of the equations in (\ref{eq:optconds_lagr}), working out $\diff{\nabla \rJ(u)}{u}[\de{u}]$ immediately leads to
\begin{equation}
\left\{\begin{array}{rcll}
-\nabla \cdot(k\nabla \de{y})& = &\beta \de{u} &\quad \mbox{on } D\\
-\nabla \cdot(k\nabla \de{p})& = &2\de{y} + 2\gamma(\de{y}-\mean{\de{y}}) &\quad \mbox{on } D\\
\hess{\tilde{J}(u)}[\de{u}] & = & 2\alpha \de{u} + \beta \mean{\de{p}}
\end{array}\right.\label{eq:Hess_vec_cont}
\end{equation}
The model problem is quadratic since the Hessian is independent of $u$.

\subsection{Discretization}
We assume that the discretization of the equations (\ref{eq:optconds_lagr}) leads to a system of the form
\begin{equation}
\left\{\begin{array}{rcll}
A \bs{y}& = &\beta \bs{u} \\
A' \bs{p}& = &2(\bs{y} -\bs{y_D}) + 2\gamma(\bs{y}-\mean{\bs{y}}) \\
\nabla \rJ(\bs{u}) & = & 2\alpha \bs{u} + \beta \mean{\bs{p}}
\end{array}\right. \label{eq:optconds_lagr_d}
\end{equation}
with $A, A' \in \mathbb{R}^{m^d \times m^d}$ dependent on $\omega$. This is the case if, e.g., the finite volume discretization with $m^d$ volumes is used. We use boldface to denote the finite dimensional approximations.

Consider the discretized cost functional 
\begin{equation}J(\bs{y},\bs{u}) = \mean{\|\bs{y}-\bs{y_D}\|^2}  + \gamma\|\std{\bs{y}}\|^2+ \alpha\|\bs{u}\|^2 \approx J(y,u) \label{eq:Jyu2_d}\end{equation}
where the norms and inner products over $\mathbb{R}^{m^d}$ are defined as the approximation of their continuous counterparts, i.e.,
\begin{equation}
\normd{\bs{v}}^2 \triangleq \frac{\bs{v}^T \bs{v}}{m^d} \approx \normd{v}^2 \quad \text{and} \quad 
(\bs{u}, \bs{v}) \triangleq  \frac{\bs{u}^T\bs{v}}{m^d} \approx (u,v).
\label{eq:norm_ip_d}
\end{equation}
This definition ensures that the discretized cost functional gives comparable results regardless of the number of discretization points. Using standard differentiation techniques, one can show that the discretized gradient (\ref{eq:optconds_lagr_d}) is also the exact gradient of the reduced cost functional $\rJ(\bs{u}) = J(A^{-1}\beta \bs{u}, \bs{u})$ w.r.t.\ the inner product given in (\ref{eq:norm_ip_d}). 
Discretizing the Hessian (\ref{eq:Hess_vec_cont}) in the same way is identical to taking the derivative of the discretized gradient, i.e., $\smash{\hess{\rJ(\bs{u})}[\bs{\de{u}}] = \diff{\nabla \rJ(\bs{u})}{\bs{u}}[\bs{\de{u}}]}$.


\section{Multilevel Monte Carlo}
The evaluation of the reduced gradient (\ref{eq:optconds_lagr}) or Hessian (\ref{eq:Hess_vec_cont}) requires an approximation for $\mean{p}$ and, if $\gamma \neq 0$, also for $\mean{y}$. 
Because of the PDE setting, it makes sense to consider a multilevel Monte Carlo (MLMC) estimator, which is briefly recalled following the exposition in \cite{cliffe2011}. Section \ref{section:MLMC/function} discusses in detail how we handle function valued quantities of interest. Section \ref{section:samples} analyzes the application of the method to the estimation of $\mean{p}$ in particular.
\newcommand{\QMC}{\ensuremath{\hat{Q}{}^{\text{MC}}}}
\newcommand{\Qt}{\ensuremath{Q_m}}

\subsection{Scalar-valued quantities of interest}
Assume one wishes to estimate the expected value of some \hl{quantity of interest} (QoI) $Q : \Omega \rightarrow \mathbb{R}$.
Because of approximation or discretization errors, one can often not generate exact samples $Q(\omega)$ of $Q$. Instead, one can generate samples $Q_m(\omega)$ of an approximation $Q_m$ to $Q$, where $m$ is a measure for the accuracy of the approximation. Here, we let $m$ correspond to the number of discretization points in one dimension. The numerical scheme is assumed to have a weak order of convergence equal to $\rho$, i.e.,
\begin{equation}
|\mean{Q_m - Q}| \lesssim m^{-\rho}. \label{eq:MC_disc_error}
\end{equation}
We define $a \lesssim b \Leftrightarrow a \leq cb$ with $c$ independent of $m$ and $\bs{n}$ below. We write $a \eqsim b$ iff $ a \lesssim b$ and $b \lesssim a$. 
The computational cost $\mathcal{C}(Q_m(\omega))$ for a single sample is assumed to satisfy
\begin{equation}
\mathcal{C}(Q_m(\omega)) \lesssim m^\kappa \label{eq:sample_cost}
\end{equation}
for some constant $\kappa$. Both $\rho$ and $\kappa$ depend on the algorithm employed to solve the PDE.

%

\newcommand{\YMC}{\ensuremath{\hat{Y}^\textup{MC}}}
\newcommand{\QMLMC}{\ensuremath{\hat{Q}^\textup{MLMC}}}
\newcommand{\nvec}{\{n_\ell\}}
\old{\com{mention recursive application of control variate variance reduction technique,  Applying this idea recursively yields ...}}
In the MLMC method \cite{cliffe2011, giles2015} one considers multiple approximations $Q_{m_0}, \ldots, Q_{m_L}$ for $Q$. In our setting, $m_\ell = m_0\cdot2^\ell$ corresponds to the grid size of the PDE discretization. The coarsest grid size is $m_0$, the finest is $m_L$. The method recursively estimates an expected value on a finer grid as an expected value on a coarser grid (acting as a control variate) combined with a corrective term. This leads to a telescopic sum decomposition
\begin{equation*}
\mean{Q_{m_L}} = \mean{Q_{m_0}} + \sum\limits_{\ell=1}^{L}\mean{Q_{m_\ell} - Q_{m_{\ell-1}}} = \sum\limits_{\ell=0}^{L}\mean{Y_\ell}
\end{equation*}
where $Y_l \triangleq Q_{m_\ell} - Q_{m_{\ell-1}}$ and $Q_{m_{-1}} \triangleq 0$. On level $\ell$, $\mean{Y_\ell}$ is estimated using the ordinary Monte Carlo (MC) method with $n_\ell$ samples, yielding
\begin{equation}
\YMC_{\ell, n_\ell} \triangleq \frac{1}{n_\ell}\sum\limits_{i=1}^{n_\ell} Y_\ell(\omega_i) = \frac{1}{n_\ell}\sum\limits_{i=1}^{n_\ell} \Big(Q_{m_\ell}(\omega_i) - Q_{m_{\ell-1}}(\omega_i)\Big).\label{eq:MC_def1Y}
\end{equation}
It is important to use the same stochastic realization $\omega_i$ on both levels for each sample of $Y_\ell$ to ensure a high correlation. The MLMC estimator is then defined as
\begin{equation}
\QMLMC_{\bs{m},\bs{n}} \triangleq \sum\limits_{\ell=0}^{L}\YMC_{\ell,n_\ell} \label{eq:MLMC_def2Q}
\end{equation}
with the vector $\bs{m} = \{m_\ell\}_{\ell = 0}^L$ and $\bs{n} = \nvec_{\ell = 0}^L$.
The linearity of the expected value operator and the fact that all the expectations are estimated independently, lead to
\begin{align}
\mean{\QMLMC_{\bs{m},\bs{n}}} = \mean{Q_{m_L}} && \var{\QMLMC_{\bs{m},\bs{\bs{n}}}} = \sum\limits_{\ell=0}^{L}n_\ell^{-1}\var{Y_\ell}. \label{eq:MLMC_meanvar}
\end{align}
Moreover, the mean square error (MSE) of $\QMLMC_{\bs{m},\bs{n}}$ as an estimator for $\mean{Q}$ can be characterized, see \cite{cliffe2011}, as follows
\begin{align}
\meanBig{\big(\QMLMC_{\bs{m},\bs{n}} - \mean{Q}\big)^2} 
&= \var{\QMLMC_{\bs{m},\bs{n}}} + \big(\mean{\QMLMC_{\bs{m},\bs{n}}} - \mean{Q}\big)^2 \nonumber\\
&= \sum\limits_{\ell=0}^{L}n_\ell^{-1}\var{Y_\ell}
\;\;+\;
\mean{Q_{m_L} - Q}^2.
\label{eq:MLMC_error}
\end{align}
The first term is due to the stochastic error, which can be decreased by taking more samples. The second term is due to the discretization error, equal to the bias squared. It can be decreased by solving the PDE on a finer grid, i.e., by increasing $L$. 
In order to have a RMSE of at most $\epsilon$ it is sufficient if both\footnote{In practice, often a larger part of the RMSE is allocated to the stochastic error.} terms are smaller than $\epsilon^2/2$. 

Many possibilities exist for $\bs{n}$ to achieve a stochastic error smaller than $\epsilon^2/2$. This freedom can be used to minimize the cost of the MLMC estimator. 
Denote the cost of taking a sample of $Y_\ell$ as $\mathcal{C}_l \triangleq \mathcal{C}(Y_\ell(\omega))$. The cost of the MLMC estimator is then
$\mathcal{C}(\QMLMC_{\bs{m},\bs{n}}) = \sum_{\ell = 0}^{L} n_\ell \mathcal{C}_\ell.
$
Minimizing this cost subject to the constraint $\sum_{\ell=0}^{L}n_\ell^{-1}\var{Y_\ell} = {\epsilon^2}/{2}$
yields an optimization problem which is easily solved using Lagrange multipliers. The solution, rounded upward, yields the optimal number of samples
\begin{equation}
n_\ell = \Bigg\lceil \frac{2}{\epsilon^2} \sqrt{\var{Y_\ell}\mathcal{C}_\ell^{-1}} \sum\limits_{i=0}^{L} \sqrt{\var{Y_i}\mathcal{C}_i }\Bigg\rceil. \label{eq:MLMC_n}
\end{equation}
Substituting (\ref{eq:MLMC_n}), before rounding upward, into the expression for the cost, yields
\begin{align*}
\mathcal{C}(\QMLMC_{\bs{m},\bs{n}}) = \frac{2}{\epsilon^2} \Bigg(\sum\limits_{\ell = 0}^{L} \sqrt{\var{Y_\ell}\mathcal{C}_\ell}\Bigg)^2.
\end{align*}
\nes{ 
	The values $\var{Y_\ell}$ depend on the difference in accuracy between levels $\ell$ and $\ell - 1$. This difference is usually larger if $m_\ell$ and $m_{\ell -1}$ are further apart. Optimizing for $\bs{n}$ and $\bs{m}$ simultaneously would be very difficult and problem dependent. In the remainder of this thesis we use $m_\ell = 2^\ell m_0$, because this choice is most convenient in the PDE context. Nevertheless, we can not rule out the existence of a set of grid sizes $\bs{m}$ that performs better.
} 
If $\var{Y_\ell}$ decreases faster than $\mathcal{C}_\ell$ increases with increasing $\ell$, the dominant cost is on the coarsest level $\ell = 0$ and is proportional to $\var{Y_0}\mathcal{C}_0$.  The cost savings compared to the standard MC method are then proportional to $\mathcal{C}_0/\mathcal{C}_L \eqsim (m_0/m_L)^\kappa \eqsim \epsilon^{\kappa/\rho}$.
If the converse is true, then the dominant cost is on the finest level $L$ and proportional to $\var{Y_L}\mathcal{C}_L$. The cost savings are then approximately $\var{Y_L}/\var{Y_0}$. 
\begin{remark}
	Note that $m_0$ cannot be made arbitrarily small. If the discretization is too coarse, the relevant features of the PDE solution can no longer be resolved. The resemblance between the coarse and fine level solution will then be lost, 
	i.e., $\var{Y_1}$ will no longer be smaller than $\var{Q_{m_1}}$
	It is then cheaper to estimate $\mean{Q_{m_1}}$ directly, which is equivalent to increasing $m_0$.
\end{remark}

Collecting all of the assumptions and quantifying the decay of $\var{Y_\ell}$ yields the MLMC cost theorem as presented and proven in $\cite{cliffe2011}$:
\begin{theorem}[Multilevel Monte Carlo cost]
	Suppose that there are positive constants $\rho, \phi, \kappa > 0$ such that $\rho > \frac{1}{2} \min (\phi, \kappa)$ and
	\begin{align*}
	|\mean{Q_{m_\ell} - Q}| \lesssim m_\ell^{-\rho} &&
	\var{Y_\ell} \lesssim m_\ell^{-\phi} &&
	\mathcal{C}_\ell \lesssim m_\ell^\kappa
	\end{align*}
	Then, for any $\epsilon < e^{-1}$, there exist a value $L$ and a sequence $\bs{n} = \{n_\ell\}_{\ell = 0}^L$ such that the MSE
	\begin{equation*}
	\meanBig{\big(\QMLMC_{\bs{m},\bs{n}} - \mean{Q}\big)^2} < \epsilon^2
	\end{equation*}
	and the cost
	\begin{equation}
	\mathcal{C}(\QMLMC_{\bs{m},\bs{n}}) \lesssim 
	\left\{\begin{array}{ll}
	\epsilon^{-2} &\text{if } \phi > \kappa \\
	\epsilon^{-2}(\log \epsilon)^2 &\text{if } \phi = \kappa \\
	\epsilon^{-2-(\kappa-\phi)/\rho} &\text{if } \phi < \kappa \\
	\end{array}\right.
	\label{eq:MLMC_cost}
	\end{equation}
	\label{theorem:MLMC}
\end{theorem}
In practice, the problem dependent parameters $\rho, \phi$ and $\kappa$ are not always known in advance and may have to be estimated. Furthermore, $L$ has to be selected carefully in order to have a sufficiently small bias term.

\subsection{Function valued quantities of interest}
\label{section:MLMC/function}
The main quantities of interest in this paper are the gradient and the Hessian vector product, which in our application are functions instead of scalar values. These functions are discretized differently on the different levels and have to be combined in the course of the estimation algorithm. Hence, it is necessary to define a mapping between those discretizations. 

\subsubsection{Mapping between different levels}
\renewcommand{\Qt}{\ensuremath{\bs{Q}_m}}
\renewcommand{\YMC}{\ensuremath{\hat{\bs{Y}}{}^\text{MC}}}
\renewcommand{\QMLMC}{\ensuremath{\hat{\bs{Q}}{}^\text{MLMC}}}
\newcommand{\chlvl}[2]{\smash{I_{#1}^{#2}}}
Consider a linear transform $I_{\ell_1}^{\ell_2} : \mathbb{R}^{m_{\ell_1}^d} \rightarrow \mathbb{R}^{m_{\ell_2}^d} : \bs{v} \mapsto \chlvl{\ell_1}{\ell_2}\bs{v}$ that maps vectors from level $\ell_1$ to level $\ell_2$. 
The operator is a \textit{prolongation} if $\ell_1 < \ell_2$, a \textit{restriction} if $\ell_1 > \ell_2$ and the identity if $\ell_1 = \ell_2$. For any $\ell_1,\ell_2 \in \mathbb{N}$ with $\ell_1<\ell_2$, we shall require that  
\begin{equation}
\chlvl{\ell_1}{\ell_2} = \chlvl{\ell_{2}-1}{\ell_2}\chlvl{\ell_{2}-2}{\ell_{2}-1} \cdots \chlvl{\ell_{1}}{\ell_{1}+1} 
\quad \text{and} \quad 
\chlvl{\ell_1}{\ell_2} = c^{\ell_2 - \ell_1} (\chlvl{\ell_2}{\ell_1})^T
\label{eq:chlvl_condition}
\end{equation}
for some constant $c$. In our case, $c=2^d$. An analogous expression should hold if $\ell_{1}>\ell_{2}$. 
The precise definition of $\chlvl{\ell_1}{\ell_2}$ depends on the discretization method and the selection of the mesh points at the different levels. The ideas in this paper do not depend on a specific method used to solve PDE (\ref{model_SPDE}) for a single realization \old{instance of} $\omega$. Here the \hl{finite volume method} will be used to obtain function values at \hl{control volume} centers. None of the nodes existing at a level $\ell$ are then present at the level $\ell+1$, see Figure \ref{fig:grids}. 
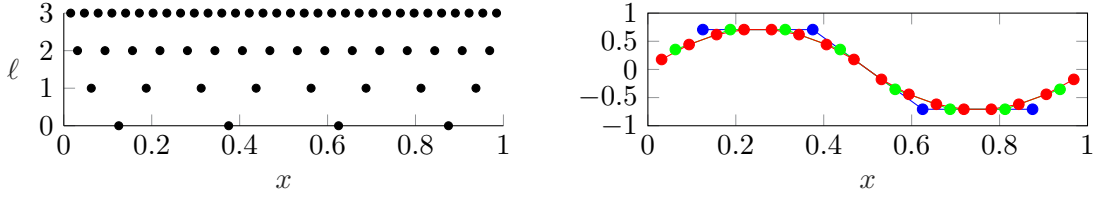
\begin{figure}[h]
	\centering	
	\sethw{0.1}{0.39}
	\input{fig/grid1.tex}
	\sethw{0.1}{0.39}
	\input{fig/mapping.tex}
	\caption{Left: Discretization node pattern corresponding to the control volume centers for a 1D problem using the finite volume method. None of the nodes existing at a certain level $\ell$ are present at another level. Right: Interpolating a function to finer grids smooths the function.}
	\label{fig:grids}
\end{figure}
The prolongation operator is often chosen to interpolate linearly. In our situation, the effect is a smoothing of the function when mapped from one level to the next.
Note that MLMC works best if the results on consecutive levels match as closely as possible. 
Alternative definitions for $\chlvl{\ell_1}{\ell_2}$ are of course possible. However, choosing a poor definition results in a slower decay of $\var{\bs{Y}_\ell}$ since it can cause unnecessary dissimilarity between $\bs{Q}_{m_\ell}$ and $\chlvl{\ell-1}{\ell}\bs{Q}_{m_{\ell-1}}$. 


\subsubsection{Revised algorithm and bias estimation}
\label{section:algo}
\newcommand{\Lmax}{{\bar{L}}}
For (discrete) vector valued quantities, (\ref{eq:MC_def1Y}) and (\ref{eq:MLMC_def2Q}) are amended as follows:
\begin{equation}
\YMC_{\ell, n_\ell} = \frac{1}{n_\ell}\sum\limits_{i=1}^{n_\ell} \Big(\bs{Q}_{m_\ell}(\omega_i) - \chlvl{\ell-1}{\ell}\bs{Q}_{m_{\ell-1}}(\omega_i)\Big) \quad \text{and} \quad
\QMLMC_{\bs{m},\bs{n}} \triangleq \sum\limits_{\ell=0}^{L}\chlvl{\ell}{\Lmax}\YMC_{\ell,n_\ell}. 
\label{eq:MLMCvec_def}
\end{equation}
In this paper, the vector valued estimator is always returned at some fixed level $\Lmax$. If a sufficiently small RMSE is reached for some $L<\Lmax$, no samples are taken at levels $\ell > L$. Returning the result at a fixed predetermined level simplifies the implementation of the optimization algorithm. Optimization software usually requires a gradient and Hessian of a given dimension that is not allowed to change from iteration to iteration. 


Some methods to extend the MLMC theory to vector or function valued QoI can be found in \cite[pp.~274--276]{giles2015}. Let $\chlvl{}{\ell}$ be the discretization operator, which samples a continuous function in the discretization nodes of level $\ell$. In this paper, we demand the MSE (\ref{eq:MLMC_error}) to be smaller than $\epsilon^2$ for each point on the final level $\Lmax$, i.e., we demand 
\begin{equation}
\mean{(\QMLMC_{\bs{m},\bs{n}} - \chlvl{}{\Lmax}\mean{Q})^2} \leq \epsilon^2. \label{eq:RMSE_def}
\end{equation} 
Much of the MLMC theory for scalar valued quantities can then be reused without much modification.

Evaluating the variance of $\QMLMC_{\bs{m},\bs{n}}$ from the definition (\ref{eq:MLMCvec_def}) shows that the relevant series of variances necessary to evaluate (\ref{eq:MLMC_n}) for all the points is given by $\{\var{\chlvl{\ell}{\Lmax}\bs{Y}_\ell}\}_{\ell = 0}^L$. These can be approximated as $\{\chlvl{\ell}{\Lmax}\var{\bs{Y}_\ell}\}_{\ell = 0}^L$, where $\var{\bs{Y}_\ell}$ are estimated using some warm up samples. The optimal number of samples is determined for each domain point separately. Then, the maximum over all the domain points is taken as $\bs{n}$. This ensures that $\sum_{\ell=0}^{L}n_\ell^{-1}\chlvl{\ell}{\Lmax}\var{\bs{Y}_\ell} \leq \epsilon^2 /2 $ in all points of the domain. 

\newcommand{\norminf}[1]{\ensuremath{\|{#1}\|_\infty}}
For the bias we start with an estimation of $\rho$ in (\ref{eq:MC_disc_error}). Assume there exist $c,c' \in \mathbb{R}$ such that from a certain level $\ell$ onward
$$\norminf{\mean{\bs{Q}_{m_\ell} - \chlvl{}{\ell}Q}} \approx cm_\ell^{-\rho} = c2^{-\rho \ell} \; \textrm{ and } \; \norminf{\mean{\bs{Q}_{m_\ell} - \chlvl{\ell-1}{\ell}\bs{Q}_{m_{\ell-1}}}} \approx  c'2^{-\rho \ell}.$$
These assumptions were found experimentally to hold best, especially for low $\ell$, when using the inf-norm.
Fitting a line to $\log_2\norminf{\mean{\bs{Q}_{m_\ell} - \chlvl{\ell-1}{\ell}\bs{Q}_{m_{\ell-1}}}} \allowbreak \approx \log_2c' -\rho \ell$ and getting the first degree coefficient then provides an estimation of $\rho$.
The reverse triangle inequality yields
\begin{align*}
\norminf{\mean{\bs{Q}_{m_\ell} - \chlvl{\ell-1}{\ell}\bs{Q}_{m_{\ell-1}}}} &\geq   \norminf{\mean{\chlvl{\ell-1}{\ell}\bs{Q}_{m_{\ell-1}} - \chlvl{}{\ell}Q}} 
- \norminf{\mean{\bs{Q}_{m_\ell} - \chlvl{}{\ell}Q}} \\
& \approx (2^\rho -1) \norminf{\mean{\bs{Q}_{m_\ell} - \chlvl{}{\ell}Q}},
\end{align*}
leading to the following bound for the largest bias over the domain:
\begin{equation}
\norminf{\mean{\bs{Q}_{m_\ell} - \chlvl{}{\ell}Q}} < (2^\rho - 1)^{-1}\norminf{\mean{\bs{Q}_{m_\ell} - \chlvl{\ell-1}{\ell}\bs{Q}_{m_{\ell-1}}}}. \label{eq:bias}
\end{equation}



Since the necessary number of levels $L$ is not a~priori known, the algorithm starts out with only a few levels and checks 
\begin{equation}
\norminf{\sum\limits_{\ell=0}^{L}n_\ell^{-1}\chlvl{\ell}{\Lmax}\var{\bs{Y}_\ell}}
+ 
\norminf{\mean{\bs{Q}_{m_L} - \chlvl{}{L}Q}}^2 \vphantom{\sum\limits_{\ell=0}^{L}} \leq \epsilon^2,
\label{eq:MLMC_error_vec}
\end{equation}
with the second term estimated through (\ref{eq:bias}). This is a somewhat overly conservative test for (\ref{eq:RMSE_def}). The above equation holds for all domain points if it holds for the worst case point, hence the inf-norm over the first term. 
It is also sufficient to simply replace the first term by $\epsilon^2/2$. 
If the resulting requirement for the bias is not satisfied, an additional level is added. An overestimation of either term would cause the algorithm to consider an additional unnecessary level. Note that this is not too bad if the dominant cost is on the coarsest grid, as is the case for all experiments in this paper.
This provides another justification for the use of the conservative inf-norm in the bias estimation.
The full MLMC algorithm is given in Algorithm \ref{alg:MLMC}. 
\alglanguage{pseudocode}
\begin{algorithm}
	\caption{Multilevel Monte Carlo estimation of function valued quantities}
	\label{alg:grad}
	\begin{algorithmic}[1]
		\State $L \leftarrow 0$, $converged \leftarrow \text{false}$
		\While {not $converged$ and $L \leq \Lmax$}
		\State take an amount $n_\text{init}$ of initial samples at level $L$
		\State estimate $\var{\bs{Y}_L}$ from these samples
		\State calculate the optimal number of samples $\bs{n} = \{n_\ell\}_{\ell=0}^L$  following \S\ref{section:algo}.
		\State take more samples on levels $0,\ldots,L$ until the total number taken is at least $\bs{n}$
		\If {$L \geq 1$}
		\State estimate $\rho$ and the bias following \S\ref{section:algo}
		\State $converged \leftarrow $ check (\ref{eq:MLMC_error_vec}) at level $L$.
		\EndIf
		\State $L \leftarrow L+1$
		\EndWhile
		\State return $\QMLMC_{\bs{m},\bs{n}}$, following (\ref{eq:MLMCvec_def})
	\end{algorithmic}
	\label{alg:MLMC}
\end{algorithm}



\section{Estimator for the gradient}
\label{section:samples}
We turn now to the specific problem of finding an estimate for $\mean{\bs{p}}$ and thus for the gradient in the optimality conditions (\ref{eq:optconds_lagr_d}). Estimating $\mean{\bs{p}}$ seems to require an estimation of $\mean{\bs{y}}$ first. This leads to two problems. First, it is assumed that the available computer memory is too small to save all the samples used to estimate $\mean{\bs{y}}$. Any such sample is thus lost unless it is recalculated later, thereby increasing calculation cost. Secondly, it is unclear which MSE would have to be requested for $\mean{\bs{y}}$. So, we want to get rid of the need to estimate $\mean{\bs{y}}$ in advance. Moreover, we want to retain the property that the calculated gradient is \hl{exact}, meaning that it is the exact gradient of some cost function.  

\subsection{Generating samples of $\bs{p}$ directly}
\label{sec:est/1}
\newcommand{\Vest}[1]{\ensuremath{\smash{\hat{V}_{#1}}}}
\newcommand{\ipO}[2]{\ensuremath{({#1})_{D,{#2}}}}
The $\mean{\bs{y}}$ term in (\ref{eq:optconds_lagr_d}) stems from
$\nabla_{\bs{y}} \gamma\|\std{\bs{y}}\|^2 \\ = 2\gamma(\bs{y} - \mean{\bs{y}})$, where the gradient is expressed w.r.t. the inner product
\begin{equation}
\ip{\bs{u}, \bs{v}} = \int_{\Omega} \frac{\bs{u}^T\bs{v}}{m^d} \d{\mu(\omega)} \approx \ip{u,v}. \label{eq:ip_d_omega_discr}
\end{equation}
This holds for any stochastic space, in particular also for a finite subset of samples $\Omega_0 = \{\omega_1, \ldots, \omega_n\} \subset \Omega$, each having equal probability. For any such set $\Omega_0$, (\ref{eq:ip_d_omega_discr}) reduces to
\begin{equation}
(\bs{u}, \bs{v})_{D,\Omega_0} = \frac{1}{n}\sum_{\omega \in \Omega_0} \frac{\bs{u}^T\bs{v}}{m^d}. \label{eq:ip_d_omega_discr_2}
\end{equation}
Writing $\bs{y}(\omega_i) = \bs{y}_i$, the following gradient w.r.t. $(., .)_{D,\Omega_0}$ 
is therefore equal to
\begin{equation}
\nabla_{\bs{y}} \Bigg\| \sqrt{\frac{1}{n}\sum_{j=1}^{n}(\bs{y}_j - \frac{1}{n}\sum_{i=1}^{n} \bs{y}_i)^2} \Bigg\|^2 = 2(\bs{y} - \frac{1}{n}\sum_{i=1}^{n}{\bs{y}_i}).
\label{eq:diff_biased_var}
\end{equation}
Hence, if $\mean{\bs{y}}$ in (\ref{eq:optconds_lagr_d}) is estimated by means of $n$ MC samples, the standard deviation term in $J(\bs{y},\bs{u})$ is to be evaluated as suggested by the l.h.s in (\ref{eq:diff_biased_var}), i.e., by using the standard (biased) sample variance. Any alternative way to estimate $\mean{\bs{y}}$ entails a corresponding change in the estimator for the variance in the cost functional and vice versa. Consider another estimator based on two sets of $n$ samples each:
\begin{equation}
\hat{V}[\bs{y}] \triangleq \frac{1}{2n}\sum_{j=1}^{n}(\bs{y}_j - \bs{y}_j')^2
\end{equation}
with $\bs{y}_j' = \bs{y}(\omega_j')$ and $\Omega'_0 = \{\omega_1', \ldots, \omega_n'\} \subset \Omega$ a second set of samples. All of the $2n$ samples are independent. Due to the independence of $\omega_j$ and $\omega_j'$ in particular, it is an unbiased estimator for the variance since
\begin{align*}
\mean{(\bs{y}_j - \bs{y}_{j}')^2} & = \mean{(\bs{y}_j - \mean{\bs{y}} + \mean{\bs{y}} - \bs{y}_{j}')^2} \\ 
& = \mean{(\bs{y}_j - \mean{\bs{y}})^2} + \mean{(\mean{\bs{y}} - \bs{y}_{j}')^2} \\
& = 2\var{\bs{y}}.
\end{align*}
Note that the MSE $\var{\Vest{}[\bs{y}]}$ is somewhat less favorable than that of the standard biased estimator, i.e., more samples are needed for an accurate estimate of $\var{\bs{y}}$. This can be demonstrated as follows. Assuming that $\bs{y}_j$ and $\bs{y}_j'$ are Gaussian with variance $\sigma^2$, it can be shown that $\var{(\bs{y}_j - \bs{y}_{j}')^2} = 8\sigma^4$ and thus $\var{\Vest{}[\bs{y}]} =\smash{ \frac{2\sigma^4}{n}}$. The standard biased variance estimator uses only $n$ realizations of $\bs{y}$ and has the same variance. However, since it leads to a gradient containing the term $\mean{\bs{y}}$, which has to be estimated in advance with additional samples of $\bs{y}$, its implied computational cost is not lower in practice. In contrast, the gradient corresponding to $\Vest{}[\bs{y}]$ w.r.t. $(., .)_{D,\Omega_0}$ is 
\begin{equation}
\nabla_{\bs{y}} \Big\| \sqrt{\hat{V}[\bs{y}]} \Big\|^2 = \bs{y} - \bs{y}'. \label{eq:grad_V}
\end{equation}
Using this expression in (\ref{eq:optconds_lagr_d}) to replace $2(\bs{y} - \mean{\bs{y}})$ yields the gradient of (\ref{eq:Jyu2_d}) where $\std{\bs{y}}$ is estimated by $(\Vest{}[\bs{y}])^{1/2}$ and $\mean{.}$ by the average over $\Omega_0$. The $j$-th sample of $\bs{p}$ at some given level then requires two solves of the state equation to obtain $\bs{y}_j$ and $\bs{y}_j'$ and a single solve of the adjoint equation. Because each sample $\bs{p}_j$ depends on distinct samples $\omega_j$ and $\omega_j'$, all of $\{\bs{p}_j\}_{j=1}^n$ are clearly independent, as is required in the MLMC method. 

\begin{remark}
	It is in principle possible to improve the resulting gradient estimator by using $\{\bs{y}'_j\}_{j=1}^n$ to generate additional samples $\{\bs{p}'_j\}_{j=1}^n$ of $\bs{p}$ that in effect correspond with $\{\omega'_j\}_{j=1}^n$, at the cost of $n$ additional adjoint equation solves. However, $\bs{p}_j$ and $\bs{p}'_j$ are not independent! If compromising on the independence of samples is allowed, many other methods can be constructed. The next section provides such a method that is very similar but easier to analyze and generalize.
\end{remark}

\subsection{Generating cheaper samples of $\bs{p}$ directly}
\label{sec:est/2}
Consider yet another estimator for the variance based only on a single set of $n$ independent samples $\Omega_0$:
\begin{equation}
\Vest{1}[\bs{y}] \triangleq \frac{1}{2n}\sum_{j=1}^{n}(\bs{y}_j - \bs{y}_{j-1})^2
\end{equation}
where $\bs{y}_{0} = \bs{y}_{n}$. Since the samples $\Omega_0$ are independent, we have again $\mean{(\bs{y}_j - \bs{y}_{j-1})^2} = 2\var{\bs{y}}$, making $\Vest{1}[\bs{y}]$ an unbiased estimator for $\var{\bs{y}}$.
The gradient w.r.t. $\ipO{.,.}{\Omega_0}$ is worked out explicitly in Appendix \ref{app:derivative}, yielding
\begin{equation}
\nabla_{\bs{y}} \Big\| \sqrt{\hat{V}_1[\bs{y}]} \Big\|^2 = 2\bs{y} - \bs{y}_{+1} - \bs{y}_{-1} \label{eq:grad_V1}
\end{equation}
where $\bs{y}_{+i}$ denotes the stochastic variable $\bs{y}$ `shifted' by $i$ samples in the sampled stochastic space: $\bs{y}_{+i}(\omega_j) \triangleq \bs{y}(\omega_{j+i})$ with $\omega_{n+i} = \omega_i$. This definition only makes sense for a given ordered finite subset of $\Omega$. A single sample of (\ref{eq:grad_V1}), e.g., the sample corresponding to $\omega_j$, is then $2\bs{y}_j - \bs{y}_{j+1} - \bs{y}_{j-1}$. Calculating multiple samples in succession then merely requires to save the previous sample, $\bs{y}_{j-1}$, and obtain the next sample, $\bs{y}_{j+1}$, early. 

One of the disadvantages is the somewhat higher variance of the estimator:
\begin{align*}
\var{\Vest{1}[\bs{y}]} &= \frac{1}{4n^2}\sum_{j=1}^{n}\var{(\bs{y}_j - \bs{y}_{j-1})^2} + \frac{1}{2n^2}\sum_{j=1}^{n}\cov{(\bs{y}_j - \bs{y}_{j-1})^2}{(\bs{y}_{j+1} - \bs{y}_{j})^2}.
\end{align*}
Under the assumption that $\bs{y}_j$ are Gaussian variables with variance $\sigma^2$, one has that
\begin{align*}
\var{(\bs{y}_j - \bs{y}_{j-1})^2} = 8\sigma^4 \quad \text{and} \quad \cov{(\bs{y}_j - \bs{y}_{j-1})^2}{(\bs{y}_{j+1} - \bs{y}_{j})^2} = 2\sigma^4,
\end{align*}
and therefore
\begin{equation*}
\var{\hat{V}_1[\bs{y}]} = \frac{2\sigma^4}{n} + \frac{\sigma^4}{n} = \frac{3\sigma^4}{n}.
\end{equation*}
This is to be compared to the Cram\'er-Rao lower bound for unbiased estimators of the variance, which is $\frac{2\sigma^4}{n}$ for Gaussian variables.

Taking dependent samples requires some changes in the classical MC and MLMC theory. For simplicity, the description is, as before, given for a scalar valued QoI.
We consider only dependencies between samples of $Y_\ell$ taken on the same level $\ell$. From (\ref{eq:MLMC_def2Q}) we then have 
\newcommand{\MC}[1]{\ensuremath{\hat{#1}^{\text{MC}}}}
\newcommand{\MLMC}[1]{\ensuremath{\hat{#1}^{\text{MLMC}}}}
\begin{equation*}
\var{\MLMC{Q}_{\bs{m},\bs{\bs{n}}}} = \sum\limits_{\ell=0}^{L}\var{\MC{Y}_{\ell,n_\ell}} = \sum\limits_{\ell=0}^{L}n_\ell^{-2}\sum\limits_{i=1}^{n_\ell}\sum\limits_{j=1}^{n_\ell}\cov{Y_{\ell,i}}{Y_{\ell,j}}
\end{equation*}
where $Y_{\ell, i}$, the $i$-th sample of $Y_\ell$, is interpreted as a random variable. If the covariance matrix is circulant and if $\cov{Y_{\ell,i}}{Y_{\ell,j}} = 0$ for $b < |i-j| < n_\ell -b$, we get
\begin{equation}
\var{\MLMC{Q}_{\bs{m},\bs{\bs{n}}}} = \sum\limits_{\ell=0}^{L}n_\ell^{-1}(\var{Y_\ell} + 
2\sum\limits_{j=2}^{b+1}\cov{Y_{\ell,1}}{Y_{\ell,j}}). \label{eq:stoch_error_correlated}
\end{equation}
For independent samples, $b = 0$ and this equation reduces to (\ref{eq:MLMC_meanvar}). For the sampling method associated with $\Vest{1}$ above, we have $b=2$. It is then necessary to estimate the $b=2$ covariances in addition to the sample variance in Line 5 of Algorithm \ref{alg:MLMC}. In (\ref{eq:MLMC_n}) and in Theorem \ref{theorem:MLMC}, $\var{Y_\ell}$ is then replaced by $\var{Y_\ell} + 2\sum_{j=2}^{b+1}\cov{Y_{\ell,1}}{Y_{\ell,j}}$. Note that since the covariances can also be negative, they can actually reduce the amount of samples required! Because the variances and covariances are estimated by a small number of samples, especially at the finer levels, the risk of underestimating the latter quantity is mitigated by replacing $\var{Y_\ell}$ with $\max\{\frac{1}{2}\var{Y_\ell}, \var{Y_\ell} + 
2\sum_{j=2}^{b+1}\cov{Y_{\ell,1}}{Y_{\ell,j}}\}$ instead. The $\frac{1}{2}$ term was chosen rather arbitrarily and can probably be improved, depending on the precision of the estimators. $\Vest{1}$ is used in the remainder of this text. 

\subsection{Exactness of the MLMC generated gradient}
\newcommand{\rJc}{\ensuremath{\hat{J}}}

Assume one uses the MLMC method to calculate the gradient in a given point $\bs{u}$ defined on the finest level. This requires the method to take samples $\bs{Q}_{m_\ell}(\bs{v}, \omega, \Omega_k)$ which depend on $\bs{v} = \chlvl{\Lmax}{\ell}\bs{u}$. Following the previous subsections, each sample may depend on multiple elements in the given ordered sample set $\Omega_k$. The variable $\omega \in \Omega_k$ is simply used to index the samples.

\begin{theorem}[exactness of the MLMC gradient]
	\label{theorem:MLMCexactness}
	Assume that for any level $\ell$ and any sample set $\Omega_k \subset \Omega$ of size $n_k$, the mapping $\smash{\mathbb{R}^{m_\ell^d} \rightarrow \mathbb{R}^{m_\ell^d}}: \bs{v} \mapsto n_k^{-1}\sum_{\omega \in \Omega_k} \bs{Q}_{m_\ell}(\bs{v}, \omega, \Omega_k)$ forms the exact gradient of some cost function. 
	Then, the MLMC method that uses on each level a combination of those sample sets, describes a mapping $\smash{\mathbb{R}^{m_\Lmax^d} \rightarrow \mathbb{R}^{m_\Lmax^d}: \bs{u} \mapsto \nabla \hat{J}(\bs{u})}$ which is itself the exact gradient of some cost function $\hat{J}$.
\end{theorem}

\begin{proof}
	Let $\rJ_\ell^k: \mathbb{R}^{m_\ell^d} \rightarrow \mathbb{R}$ denote the cost function 
	corresponding to the exact gradient $n_k^{-1}\sum_{\omega \in \Omega_k} \bs{Q}_{m_\ell}(\bs{v},\omega, \Omega_k) = \nabla \rJ_\ell^k (\bs{v})\in \smash{\mathbb{R}^{m_\ell^d}}$. Consider $\bs{u} \in \smash{\mathbb{R}^{m_\Lmax^d}}$ given at the level $\Lmax$. The chain rule and the second equation of (\ref{eq:chlvl_condition}) yield
	\begin{equation}
	\diff{}{\bs{u}} \rJ_\ell^k(\chlvl{\Lmax}{\ell}\bs{u})[\bs{h}] = (\nabla \rJ_\ell^k(\chlvl{\Lmax}{\ell}\bs{u}), \chlvl{\Lmax}{\ell}\bs{h}) = \frac{1}{m^d_\ell}\nabla \rJ_\ell^k(\chlvl{\Lmax}{\ell}\bs{u})^T\chlvl{\Lmax}{\ell}\bs{h} = \frac{c^\ell m^d_\Lmax}{c^\Lmax m^d_\ell}(\chlvl{\ell}{\Lmax}\nabla \rJ_\ell^k(\chlvl{\Lmax}{\ell}\bs{u}), \bs{h}).
	\label{eq:chain_gradient}
	\end{equation}
	Note that the two instances of the inner product (\ref{eq:norm_ip_d}) are different since they use a different number of discretization points.
	%
	Let $\nabla\rJc(\bs{u})$ denote the gradient approximation generated by the MLMC algorithm, assumed to have converged on a level $L$.
	Denote the set of $n_\ell$ samples taken by the algorithm at any level $\ell$ by $\Omega_\ell \subset \Omega$.
	From (\ref{eq:MLMCvec_def}) we have
	\newcommand{\Ql}[2]{\ensuremath{\bs{Q}_{m_{#1}}(\chlvl{\Lmax}{#1}\bs{u}, \omega, \Omega_{#2} ) }}
	\begin{align*}
	\nabla \rJc(\bs{u}) &= 
	\sum\limits_{\ell=0}^{L} \chlvl{\ell}{\Lmax}\frac{1}{n_\ell}\sum\limits_{\omega\in\Omega_\ell} \big(\Ql{\ell}{\ell} - \chlvl{\ell-1}{\ell}\Ql{\ell-1}{\ell}\big) \\
	&= \chlvl{0}{\Lmax}\nabla\rJ_0^0(\chlvl{\Lmax}{0}\bs{u}) + 
	\sum_{\ell = 1}^{L} \chlvl{\ell}{\Lmax} \big(
	\nabla\rJ_\ell^\ell(\chlvl{\Lmax}{\ell} \bs{u}) - \chlvl{\ell-1}{\ell} \nabla\rJ_{\ell-1}^{\ell}(\chlvl{\Lmax}{\ell-1} \bs{u})\big).
	\end{align*}
	This calculated gradient is the exact gradient of the cost functional
	\begin{equation}
	\rJc(\bs{u}) = \frac{c^\Lmax m^d_0}{c^0 m^d_\Lmax} \rJ_0^0(\chlvl{\Lmax}{0}\bs{u}) + 
	\sum_{\ell = 1}^{L} 
	\big(\frac{c^\Lmax m^d_\ell}{c^\ell m^d_\Lmax}\rJ_\ell^\ell(\chlvl{\Lmax}{\ell} \bs{u}) 
	- \frac{c^\Lmax m^d_{\ell-1}}{c^{\ell-1} m^d_\Lmax}\rJ_{\ell-1}^{\ell}(\chlvl{\Lmax}{\ell-1} \bs{u})\big)
	\label{eq:exact_costfun}
	\end{equation}
	as can be checked using (\ref{eq:chain_gradient}). 
\end{proof}

In a similar fashion, it can be proven that the Hessian vector product calculated using MLMC is exact for some cost functional. 
Note that in this paper, the constant $c$ from (\ref{eq:chlvl_condition}) satisfies $c = 2^d$ and $m^d_\ell = 2^{d\ell} m_0^d$, such that $\smash{c^\ell m^d_\Lmax(c^\Lmax m^d_\ell)^{-1}} = 1$ in (\ref{eq:chain_gradient}). This is also true for many other common grid definitions and mapping operators. The cost functional (\ref{eq:exact_costfun}) then simplifies to 
\begin{equation}
\rJc(\bs{u}) = \rJ_0^0(\chlvl{\Lmax}{0}\bs{u}) + 
\sum_{\ell = 1}^{L} 
\big(\rJ_\ell^\ell(\chlvl{\Lmax}{\ell} \bs{u}) 
- \rJ_{\ell-1}^{\ell}(\chlvl{\Lmax}{\ell-1} \bs{u})\big),
\end{equation}
which corresponds to the cost functional calculated using MLMC. The reason to construct the MLMC estimator for the gradient directly is that the optimization requires the gradient with some known precision, as measured by the RMSE. For some given number of samples $\bs{n}$, the RMSE on the cost functional estimator is, in general, completely different from the RMSE on the gradient estimator.

\section{Numerical Optimization}
\label{section:optimization}
We follow the \hl{reduced} optimization approach, in which the state $y$ is eliminated. This results in $u$ being the only unknown in the optimization problem. The alternative is the \emph{simultaneous} approach in which the original constrained optimization problem is solved directly \cite{borzi2009multigrid}. However, in the stochastic case, one has that $y \in H^1_0(D) \otimes L^2(\Omega)$. Assuming a MLMC approach, the full sample set of discretized state functions on all levels would have to be part of the variables that one optimizes for. This approach seems infeasible due to excessive memory demands.

In this section, we elaborate on two methods to solve our optimization problem up to a given gradient tolerance. First, the use of the nonlinear conjugate gradient (NCG) method is investigated. Specific attention is given to how many samples and which samples should be used. Next, the Newton method is studied, which requires Hessian information. Because of its unwieldy size in the problems we consider, the Hessians are never computed explicitly and the linear systems in the Newton iterations are solved using a matrix-vector product based implementation of the conjugate gradient method (CG). For quadratic problems, the NCG method and a single Newton step with CG are known to be equivalent \cite{nazareth2009}. The difference in this stochastic context will lie mainly in the times at which the samples are updated.

\subsection{Gradient based optimization}
The use of MLMC in a gradient based optimization algorithm is tested using the nonlinear conjugate gradient (NCG) method. Variables at the $k$-th iteration are indicated by a $(k)$ superscript. The gradient calculated at iteration $k$ is denoted by $\smash{\bs{g}^{(k)}}$. 
In the NCG method, the search direction $\bs{d}^{(k)}$ is obtained recursively as $\smash{\bs{d}^{(k)} = -\bs{g}^{(k)} + \beta^{(k)}\bs{d}^{(k-1)}}$ with $\smash{\bs{d}^{(0)} = -\bs{g}^{(0)}}$. We use the Dai-Yuan (DY) formula $$\beta^{(k)} = \frac{||\bs{g}^{(k)}||^2 }{(\bs{d}^{(k-1)}, \bs{g}^{(k)} - \bs{g}^{(k-1)})},$$ which offers certain advantages in PDE constrained optimization \cite{Borzi2012, Dai1999}.
The system input is then updated as
$\bs{u}^{(k+1)} = \bs{u}^{(k)} + s^{(k)}\bs{d}^{(k)}.$
The step size $s^{(k)}$ is found by approximating the cost function along the search direction with an interpolating parabola using a second gradient evaluation (we use the point $\bs{u}^{(k)} + s^{(k-1)}\bs{d}^{(k)}$ with $s^{(-1)}$ some well chosen initial value). Note that this approximate linesearch is exact for a quadratic problem.

\subsubsection{Choosing samples}
\newcommand{\Jn}{\rJc_\$}
In theory, the gradient and Hessian vector product are deterministic quantities due to the expected value operators in their equations. Computationally however, the result depends on the specific samples drawn by the MLMC algorithm. These samples can be saved memory efficiently by storing the number of samples $\bs{n}$ taken at each level and the random number generator seeds that have ultimately determined the samples.
Let the subscript $f$ in $\nabla \rJc_f(\bs{u})$ denote that the gradient in some $\bs{u}$ is calculated using some fixed set of samples $f$. 
The function $\nabla \rJc_f$ is then deterministic. 
Because $\bs{n}$ and $L$ are then given, the bias and many other things do not have to be estimated again.
The case where new samples are taken is denoted with the subscript $\$$, e.g., $\nabla \rJc_\$ (\bs{u})$.

To demonstrate the effect of any fixed or new samples, we perform 40 NCG-DY optimization steps using $\nabla \rJc_f$.  Here, $f$ is determined during the first gradient evaluation call with the tolerance for the underlying MLMC algorithm set to $\epsilon = \num{1e-3}$.
The blue line in Figure \ref{fig:new_vs_fixed} shows the decay of the norm of $\nabla \rJc_f$.
The resulting sequence of control inputs $\smash{\{\bs{u}^{(k)}_f\}_{k=0}^{40}}$ are then evaluated using $\nabla \rJc_\$$. Only this new sample gradient $\nabla \rJc_\$ (\bs{u})$ is relevant to assess the quality of a control input $\bs{u}$. After all, the solution must perform well for the original problem, not just for the specific set of fixed samples $f$. The norm (\ref{eq:norm_ip_d}) of $\smash{\nabla \rJc_\$(\bs{u}^{(k)}_f)}$ is shown as the red line in Figure \ref{fig:new_vs_fixed}. Observe that the decay levels off at a certain point because $\nabla \rJc_f$ only resolves the gradient up to some RMSE $\epsilon$, which we have estimated using (\ref{eq:bias}) and (\ref{eq:MLMC_error_vec}). 
It is thus important to either stop at that point or to decrease $\epsilon$, generating a new, larger set of samples.
For comparison NCG-DY is also executed using $\nabla \rJc_\$$, i.e., while always using new random samples in each iteration. The resulting iterates, denoted by $\smash{\{\bs{u}_\$^{(k)}\}_{k=0}^{40}}$, are outperformed by those produced using fixed samples, even when tested using new samples. 
This is the argument for holding onto the fixed samples as long as possible.

\begin{figure}[b!]
	\centering
	\begin{minipage}[t]{.45\textwidth}
		\centering
		\sethw{0.5}{0.8}
		\input{fig/new_vs_fixed.tex}
		\caption{Effect of using fixed or new samples on the evolution of the gradient, see \S7.1.1. $\smash{\|\nabla \rJc_f(\bs{u}^{(k)}_f)\|}$ (\ref{tikz:DYNCGfixed}), $\smash{\|\nabla \rJc_\$(\bs{u}^{(k)}_f)\|}$ (\ref{tikz:reevaluation}), $\smash{\|\nabla \rJc_\$(\bs{u}_\$^{(k)})\|}$ (\ref{tikz:DYNCGnew}). The expected RMSE $\epsilon$ of $\smash{\nabla \rJc_f(\bs{u}^{(k)}_f)}$ (\ref{tikz:RMSE}). 
		}
		\label{fig:new_vs_fixed}
	\end{minipage}
	\quad
	\begin{minipage}[t]{.45\textwidth}
		\centering
		\sethw{0.5}{0.8}
		\input{fig/variance_evolution2.tex}
		\caption{Evolution of the variances. $\|\var{\bs{Y}_\ell}\|_\infty$ (\ref{tikz:vars}) and $\smash{\| \max\{\frac{1}{2}\var{\bs{Y}_\ell},}$ \\ $\smash{\var{\bs{Y}_\ell} + 
				2\sum_{j=2}^{b+1}\cov{\bs{Y}_{\ell,1}}{\bs{Y}_{\ell,j}}\} \|_\infty}$ (\ref{tikz:covars}) for levels $\ell = \{0, \ldots, 5\}$.  A higher line always corresponds to a coarser grid.}
		\label{fig:variance_evolution}
	\end{minipage}
\end{figure}

Consider again the sequence $\{\bs{u}^{(k)}_f\}_{k=0}^{40}$. As the iterates come closer to the minimizer $\opt{\bs{u}}$, the variances $\var{\bs{Y}_\ell}$ tend to converge to some constant and, in general, nonzero level. Indeed, if $k \rightarrow \infty$, $\bs{u}^{(k)} \rightarrow \opt{\bs{u}}$ and the variances in consideration tend to those for $\opt{\bs{u}}$. This is illustrated in Figure \ref{fig:variance_evolution} where $\|\var{\bs{Y}_\ell}\|_\infty$ and $\| \max\{\frac{1}{2}\var{\bs{Y}_\ell}, \var{\bs{Y}_\ell} + 
2\sum_{j=2}^{b+1}\cov{\bs{Y}_{\ell,1}}{\bs{Y}_{\ell,j}}\} \|_\infty$ are shown for each level as a function of $k$. 
The use of correlated samples clearly reduces the variance on all levels. Observe also that during the first few iterations, the variances still change substantially. This causes the expected RMSE $\epsilon$ to fluctuate in Figure \ref{fig:new_vs_fixed}.

\newcommand{\cl}{q}
\subsubsection{Algorithm}
Let $\tau$ denote the tolerance on the gradient norm and consider Algorithm \ref{alg:optim_grad}. Line 2 evaluates the initial gradient and collects the samples into $f$. Lines 4--8 implement the stopping condition. If $\|\bs{g}^{(k)}\| \leq \tau$, the current iterate is checked again using new samples, see line 5, before returning it. Lines 9--10 describe the optimization step. Other optimization algorithms can also be used here. Lines 11--16 govern the generation of the gradient. In each iteration step, either the gradient is calculated using new samples, which are stored in $f$, see line 13, or the gradient is calculated using the existing fixed sample set $f$, see line 15. In that case the expected RMSE $\epsilon$ is also calculated. During the optimization procedure, a gradient is only useful if the gradient estimator $\bs{g}^{(k)}$ has a RMSE $\epsilon \leq \cl\|\bs{g}^{(k)}\|$. The constant $\cl$, which we set to $1$ for all experiments, essentially determines how large the relative RMSE $\epsilon / \|\bs{g}^{(k)}\|$ is allowed to be. An error of a given size has little effect on a large gradient, but may completely distort a small gradient. Algorithm \ref{alg:optim_grad} therefore keeps $\epsilon$ proportional to the currently attained gradient norm. Note that this is much more efficient than keeping $\epsilon$ proportional to the target norm\footnote{This behaviour can still be achieved using Algorithm \ref{alg:optim_grad} by setting $\epsilon^{(0)} = \cl\tau$ and $\eta = 0$.} $\tau$ since, in the former approach, all but the last few iterations are then computed using a larger $\epsilon$.
The samples are reused as long as possible until $\epsilon \leq \cl\|\bs{g}^{(k)}\|$ no longer holds, see line \ref{alg:optim_grad_increase_epsilon}. 
At that point $\epsilon$ is reduced by a well chosen factor $\eta$.
A smaller $\eta$ causes a given set of samples to last for more iterations, but the requested tolerance $\epsilon$ might then be reduced more than necessary, which is inefficient. In this paper we use $\eta = 0.2$. Note that $\epsilon$ should not be reduced below $\cl\tau$ as the iteration will stop approximately when $\|\bs{g}^{(k)}\| \leq \tau$ anyway. This is the reason for the $\max$ in lines 11--12. Note that line 11 also checks if $\epsilon$ is unnecessarily small by testing $\epsilon^{(k)} < \eta^2 \cl\|\bs{g}^{(k)}\|$.

\old{A gradient of norm $\|\nabla \rJ\| \approx \|\nabla \rJc_f\|$ can only be achieved if the gradient estimator $\nabla \rJc_f$ has a RMSE $\epsilon \leq c\|\nabla \rJc_f\|$, with $c$ a well chosen problem dependent parameter. For all problems that we tested, $c = 1$ was sufficient. 
	For the experiment in Figure \ref{fig:new_vs_fixed}, the gradient norm can even go significantly below $\epsilon$, and a larger $c$ would do just fine. For the sake of a fair comparison we take $c=1$ for all problems in the remainder of this paper.
	
	Let $\tau$ denote the requested tolerance on the gradient $\|\nabla \rJ \|$. Consider algorithm \ref{alg:optim_grad}, which makes sure that $\epsilon \leq c\|\nabla \rJc_f\|$ during most iteration steps. The samples are reused as long as possible until the above condition no longer holds. At that point we demand a sharper $\epsilon = \max\{c\tau, \eta c\|\bs{g}^{(k)}\|\}$ (there is no reason to ever have $\epsilon \leq c\tau$, hence the $\max$). A smaller $\eta$ causes a given set of samples to last for more iterations, but during some iterations, one might then have that $\epsilon << c\|\nabla \rJc_f\|$, which is inefficient. In this paper we use $\eta = 0.1$.}

\alglanguage{pseudocode}
\begin{algorithm}
	\caption{Gradient based optimization}
	\label{alg:optim_grad}
	\begin{algorithmic}[1]
		\State input $\tau$, $\cl$, $\epsilon^{(0)}$, $\eta$, $k_\text{max}$, $\bs{u}^{(0)}$, $\nabla \rJc_.(.)$
		\State $(\bs{g}^{(0)}, f) \leftarrow \nabla \rJc_\$(\bs{u}^{(0)})$ using RMSE $\epsilon^{(0)}$. \Comment Save new sample data in $f$
		\For {$k = 0,\ldots, k_\text{max}-1$}
		\If{$\|\bs{g}^{(k)}\| \leq \tau$} \Comment Test convergence using current samples
		\If{$\|\nabla \rJc_\$(\bs{u}^{(k)})\| \leq \tau$} \Comment Test convergence using new random samples
		\State \Return $\bs{u}^{(k)}$
		\EndIf
		\EndIf
		\State Get $\bs{d}^{(k)}$ and $s^{(k)}$, e.g., using NCG-DY and approximate linesearch
		\State $\bs{u}^{(k+1)} \leftarrow \bs{u}^{(k)} + s^{(k)}\bs{d}^{(k)}$
		\If {$\epsilon^{(k)} > \max\{\cl\tau, \cl\|\bs{g}^{(k)}\|\}$ or $\epsilon^{(k)} < \eta^2 \cl\|\bs{g}^{(k)}\| $}
		\label{alg:optim_grad_increase_epsilon}
		\State $\epsilon^{(k+1)} \leftarrow \max\{\cl\tau, \eta \cl\|\bs{g}^{(k)}\|\}$
		\State $(\bs{g}^{(k+1)}, f) \leftarrow \nabla \rJc_\$(\bs{u}^{(k+1)})$ using RMSE $\epsilon^{(k+1)}$. \Comment Save new sample data in $f$
		\Else
		\State $(\bs{g}^{(k+1)}, \epsilon^{(k+1)}) \leftarrow \nabla \rJc_f(\bs{u}^{(k+1)})$ \Comment Calculate expected RMSE $\epsilon$
		\EndIf
		\EndFor
	\end{algorithmic}
\end{algorithm}

\subsubsection{Performance}

For the NCG-DY algorithm, the norm of the gradient is known to converge linearly, i.e., $||\bs{g}^{(k)}|| = \mathcal{O}(e^{-k})$. We assume that this linear convergence is retained even though the sample set that generates the gradient changes at some of the iterations. 
We also assume that a single triple $\rho, \phi, \kappa$ exists such that the three assumptions\footnote{For correlated samples, the second assumption is amended as suggested by (\ref{eq:stoch_error_correlated}), see \S\ref{sec:est/2}.} in Theorem \ref{theorem:MLMC} hold uniformly, i.e., for the same constants implicit in $\lesssim$, for each point on the return level $\Lmax$. This follows from simply considering the constants corresponding to the worst case point. Furthermore, it is reasonable to also assume uniformity from a certain optimization step onward. It was observed already that the variances converge to fixed values, meaning $\phi$ converges to a fixed value. The costs at each level, and therefore also $\kappa$, are constant by design.

\begin{theorem}[MLMC optimization cost]
	\label{theorem:MLMC_optcost}
	Let $\rho, \phi, \kappa$ exist such that the assumptions in Theorem \ref{theorem:MLMC} hold uniformly for each point on the return level $\Lmax$ and from a certain optimization step onward and let $\epsilon^{(0)}$ be independent of $\tau$. If the norm of the gradient converges linearly, the cost $\mathcal{C}_\text{opt}(\tau)$ for Algorithm \ref{alg:optim_grad} to reach a gradient $\|\bs{g}^{(k)}\| \leq \tau$ is
	\begin{equation}
	\mathcal{C}_\text{opt}(\tau) \lesssim 
	\left\{\begin{array}{ll}
	\tau^{-2} &\text{if } \phi > \kappa \\
	\tau^{-2}(\log \tau)^2 &\text{if } \phi = \kappa \\
	\tau^{-2-(\kappa-\phi)/\rho} &\text{if } \phi < \kappa \\
	\end{array}\right., \quad \tau \rightarrow 0.
	\label{eq:tau_cost}
	\end{equation} 
\end{theorem}

\begin{proof}
	Since $||\bs{g}^{(k)}|| = \mathcal{O}(e^{-k})$, the number of iterations $K$ needed to satisfy the tolerance $\tau$ is $K = \mathcal{O}(-\log \tau)$. Consider the three cases in (\ref{eq:MLMC_cost}). 
	From a certain optimization step onward, the case in which one finds oneself remains the same. We now consider each case separately.
	Assume the cost of a gradient evaluation to be $\mathcal{C}(\bs{g}^{(k)}) \lesssim (\epsilon^{(k)})^{-2}$, i.e., assume the first case in Theorem \ref{theorem:MLMC}. 
	If $\epsilon^{(k)} \simeq \|\bs{g}^{(k)}\|$, the total cost is
	\begin{equation*}
	\mathcal{C}_\text{opt}(\tau) 
	= \sum_{k=0}^{K}\mathcal{C}(\bs{g}^{(k)}) 
	\lesssim \sum_{k=0}^{K}\|\bs{g}^{(k)}\|^{-2} 
	\simeq \sum_{k=0}^{K}e^{2k}
	= \frac{e^{2K+2} -1}{e^2 -1} 
	= \mathcal{O}(\tau^{-2}), \quad \tau \rightarrow 0.
	\end{equation*}
	The case $\phi = \kappa$ in Theorem \ref{theorem:MLMC} has $\mathcal{C}(\bs{g}^{(k)}) \lesssim (\epsilon^{(k)})^{-2}(\log \epsilon^{(k)})^2$ and yields
	\begin{equation*}
	\sum_{k=0}^{K}\|\bs{g}^{(k)}\|^{-2} (\log \|\bs{g}^{(k)}\|)^2
	\simeq \sum_{k=0}^{K}k^2e^{2k}
	= \mathcal{O}(K^2e^{2K}) 
	= \mathcal{O}(\tau^{-2}(\log\tau)^2), \quad \tau \rightarrow 0.
	\end{equation*}
	The third case is mathematically analogous to the first. 
\end{proof}

This cost is proportional to the cost of a single gradient evaluation with RMSE $\tau$. Note that if instead the tolerance would be kept fixed, i.e., $\epsilon^{(0)}=\ldots=\epsilon^{(k)} \simeq \tau, \tau \rightarrow 0$, the total cost would amount to
\begin{equation*}
\mathcal{C}_\text{opt}(\tau) = \mathcal{C}(\bs{g}^{(k)}) \mathcal{O}(-\log \tau ) 
\lesssim 
\left\{\begin{array}{ll}
-\tau^{-2} \log \tau &\text{if } \phi > \kappa \\
-\tau^{-2}(\log \tau)^3 &\text{if } \phi = \kappa \\
-\tau^{-2-(\kappa-\phi)/\rho} \log \tau &\text{if } \phi < \kappa \\
\end{array}\right. , \quad \tau \rightarrow 0.
\label{eq:tau_cost_inef}
\end{equation*}

\subsection{Hessian based optimization}
For a general problem, the cost functional can be approximated at a certain iteration as
\begin{equation*}
\rJ(\bs{u}^{(k)} + \Delta\bs{u}) \approx \rJ(\bs{u}^{(k)}) + (\nabla \rJ(\bs{u}^{(k)}))^T\Delta\bs{u}^{(k)} + \frac{1}{2}\Delta\bs{u}^T(\hess{\rJ(\bs{u}^{(k)})})\Delta\bs{u}.
\end{equation*}
For our quadratic model problem, this approximation is exact. Taking the derivative to $u$ and setting it to $\bs{0}$ yields
\begin{equation}
(\hess{\tilde{J}(\bs{u}^{(k)})})\Delta\bs{u} = -\nabla \tilde{J}(\bs{u}^{(k)}). \label{eq:hess_solve}
\end{equation}
Solving for $\Delta \bs{u}$ constitutes a Newton step and produces the next point $\bs{u}^{(k+1)} = \bs{u}^{(k)} + \Delta\bs{u}$.
The model problem converges in one single Newton iteration. Of course, in general, multiple Newton steps are required.

As was already noted, it is infeasible to directly solve (\ref{eq:hess_solve}) because the Hessian is very large and dense. Therefore, a conjugate gradient (CG) method is employed, as described in, e.g., \cite{trefethen1997}. This method requires only Hessian vector products and solves symmetric positive definite systems of equations. 
It can be shown that for a quadratic problem, the NCG method and a single solve of (\ref{eq:hess_solve}) using the CG method are equivalent, assuming the line searches are exact. A concise overview of these relationships can be found in, e.g., \cite{nazareth2009}. 
This allows us to reuse some results from the previous section. Notably, Figure \ref{fig:new_vs_fixed} can be reinterpreted as being about the residual $\bs{r}^{(i)}=\smash{(\hess{\rJc(\bs{u}^{(k)})})\Delta\bs{u}^{(i)} + \nabla \rJc(\bs{u}^{(k)})}$ for iteration $i$ of the CG method. Analogously, we conclude that the CG iterations can only be meaningful as long as $\epsilon \leq \cl\|\bs{r}^{(i)}\|$, with $\cl$ having the same meaning as before, leading to a residual tolerance of $\cl^{-1}\epsilon$.
A possible algorithm is then given below as Algorithm \ref{alg:optim_hess}. It is constructed such that the CG method can be thought of as a black box solver (which allows one to swap it with other iterative methods). The sample set used in a single Newton iteration is determined by the RMSE $\epsilon$ requested for the gradient. The first iteration uses a large RMSE $\epsilon^{(0)}$ to cheaply get somewhat close to the optimizer $\opt{\bs{u}}$. Subsequent Newton steps lower the RMSE by a factor $\eta$ (we take $\eta = 0.2$) until $\cl\tau$ (we take $\cl=1$) is reached. 

\alglanguage{pseudocode}
\begin{algorithm}
	\caption{Hessian based optimization}
	\label{alg:optim_hess}
	\begin{algorithmic}[1]
		\State input $\tau$, $\cl$, $\epsilon^{(0)}$, $\eta$, $k_\text{max}$, $\bs{u}^{(0)}$, $\nabla \rJc_.(.)$, $\hess{\rJc_.(.)}[.]$ \label{alg:optim_hess.eps1}
		\For {$k = 0,\ldots, k_\text{max}-1$}
		\State $(\bs{g}^{(k)}, f) \leftarrow \nabla \rJc_\$(\bs{u}^{(k)})$ using RMSE $\epsilon^{(k)}$. \label{alg:optim_hess.fgen}
		\If{$\|\bs{g}^{(k)}\| \leq \tau$}
		\If{$\|\nabla \rJc_\$(\bs{u}^{(k)})\| \leq \tau$}
		\State \Return $\bs{u}^{(k)}$
		\EndIf
		\EndIf
		\State $\Delta \bs{u} \leftarrow CG(\hess{\rJc_f(\bs{u}^{(k)})}[.], -\bs{g}^{(k)})$ with residual tolerance $\cl^{-1}\epsilon^{(k)}$.
		\State $\bs{u}^{(k+1)} \leftarrow \bs{u}^{(k)} + \Delta\bs{u}$
		\State $\epsilon^{(k+1)} \leftarrow \max\{\cl\tau, \eta \epsilon^{(k)}\}$ \label{alg:optim_hess.eps2}
		\EndFor
	\end{algorithmic}
\end{algorithm}

\subsubsection{Performance}
Close to the solution, the Hessian does not change significantly. Consider the sequence of all CG iterations during all the Newton steps in the algorithm and index it with $i$. The total amount of steps until convergence
is $I = \mathcal{O}(-\log \tau)$.
From the algorithm it is clear that $\epsilon^{(i)} \lesssim \|\bs{r}^{(i)}\|$ (where $\epsilon^{(i)}$ now denotes the value of $\epsilon$ during CG iteration $i$).
The cost of a single CG iteration, denoted $\mathcal{C}(\text{CG}^{(i)})$, is dominated by the Hessian vector product, which uses the samples $f$ in line \ref{alg:optim_hess.fgen}. Hence, this cost is again given by Theorem \ref{theorem:MLMC} as a function of $\epsilon^{(i)}$. Everything thus being analogous to the gradient case, working out $\mathcal{C}_\text{opt}(\tau) 
= \sum_{i=0}^{I}\mathcal{C}(\text{CG}^{(i)}) $ leads again to (\ref{eq:tau_cost}).

\section{Numerical results}
This section contains results of a set of numerical experiments in which the gradient and Hessian based optimization algorithms are applied to the model problem. The algorithms are also tested on a nonlinear problem in Section \ref{sec:num.3}. The target function is set as 
\begin{equation*}
y_D(x) = \begin{cases} 
1 & x \in [0.25, 0.75] \times [0.25, 0.75]\\
0 & \text{otherwise}
\end{cases}
\end{equation*}
The following table contains all parameters that we fix for all experiments.
\begin{equation*}
\small
\begin{array}{c | c c c}
\quad\quad \text{uncertainty} \quad\quad\quad & \multicolumn{3}{c}{\text{solver parameters}}  \\ \hline
\begin{aligned}
\lambda &= 0.3 \\
\nkl &= 500
\end{aligned} &
\begin{aligned}
m_0 &= 8 \\
m_{\Lmax} &= 256
\end{aligned} &
\begin{aligned}
\cl &= 1 \\
\eta &= 0.2
\end{aligned} &
\begin{aligned}
\epsilon^{(0)} &=\num{1e-2} \\ 
\bs{u}^{(0)} &= \bs{0} 
\end{aligned}
\end{array}
\label{par:2}
\end{equation*}
All calculations are performed in Matlab on an Intel(R) Core(TM) i5-5200U CPU @ 2.20GHz. The algebraic system of equations resulting from the finite volume discretization of the PDEs are solved using Matlab's sparse matrix solver. This can be shown experimentally to yield $\kappa = 2.26$ in Theorems \ref{theorem:MLMC} and \ref{theorem:MLMC_optcost} for our 2D problem. In these experiments $\phi > \kappa$, and therefore the dominant cost is on the coarsest level.

\subsection{Problem 1}
Problem 1 is further defined by
\begin{equation}
\alpha = \num{1e-6}, \gamma = 1, \tau = \num{1e-4}, \sigma^2 = 0.1,
\end{equation}
see also Figure \ref{fig:2Dsamples}. The problem is solved using Algorithm \ref{alg:optim_grad} and Algorithm \ref{alg:optim_hess}. The convergence behavior of both methods is visualized in Figure \ref{fig:exp1/conv}. The total time (with all overhead included) was $1542s$ (gradient based) and $1989s$ (Hessian based). \cref{tab:prob1.grad,tab:prob1.hess} below give sampling information each time new samples are generated for both methods. At the finer levels, the number of initial samples sometimes appears, meaning that no additional samples were required beyond those initial samples. The timings are calculated as the average wallclock time for a single NCG or CG iteration that uses the indicated number of samples. For comparison, taking $n_0$ samples on level $4$, as would be approximately the case for the classical MC method, would take an estimated $19$hrs for a single NCG iteration and $12$hrs for a single CG iteration for $\epsilon = \num{1e-4}$.
\begin{table}[H]
	\vspace{-0.5cm}
	\caption{Behaviour of gradient based optimization (Algorithm \ref{alg:optim_grad}).}
	\label{tab:prob1.grad}
	\centering
	\vspace{-0.75cm}
	\begin{equation*}
	\small
	\begin{array}{l|lllllllll}
	k & \epsilon^{(k)} & n_0 & n_1 & n_2 & n_3 & n_4 & n_5 & \text{estimate of } \rho & t^{(k)} [s]  \\
	\hline
	0 & \num{1e-2} & 140 & 76 & 44 & & & & 2.0237& 2.05\\
	4 & \num{2.24e-4} & 17150 &	1512 & 80 &	28 & 20 &  & 1.5824 & 47.49\\
	15 & \num{1e-4} & 98452 & 9156 &	940	& 118 &	20 &  & 1.5825 & 248.84
	\end{array}
	\end{equation*}
	\vspace{-0.75cm}
\end{table}
\begin{table}[H]
	\vspace{-0.5cm}
	\caption{Behaviour of Hessian based optimization (Algorithm \ref{alg:optim_hess}).}
	\label{tab:prob1.hess}
	\centering
	\vspace{-0.75cm}
	\begin{equation*}
	\small
	\begin{array}{l|lllllllll}
	i & \epsilon^{(i)} & n_0 & n_1 & n_2 & n_3 & n_4 & n_5 & \text{estimate of } \rho & t^{(i)} [s]  \\
	\hline
	0 & \num{1e-2} & 140 & 76 & 44 & & & & 1.8355 & 4.12\\
	2 & \num{2e-3} & 140 & 76 & 44 & & & & 1.7030 & 4.36\\
	4 & \num{4e-4} & 5964 &	521	& 44 & 28 &	20 &  & 1.5905 & 14.45\\
	13 & \num{1e-4} & 96159	& 9010 & 821 & 93 &	22 &  & 1.6195 & 153.42
	\end{array}
	\end{equation*}
	\vspace{-0.75cm}
\end{table}
\begin{figure}[h]
	\centering
	\sethw{0.25}{0.4}
	\input{fig/exp1/conv.tex}
	\caption{Behavior of Algorithms \ref{alg:optim_grad} and \ref{alg:optim_hess} for Problem 1. The crosses ($\times$) indicate the results of convergence tests performed using new samples. }
	\label{g:exp1/conv}
\end{figure}
Figure \ref{fig:exp1/g_lvl} shows (a cross section of) the contributions to the gradient on each level. Note that the smoothness of these contributions follows from the discussion about Figure \ref{fig:grids}. Since the gradient must converge to zero, the contributions cancel each other out more effectively in the later iterations. Nevertheless, the variances remain much higher on the coarsest levels (the behavior is similar to the behavior observed in Figure \ref{fig:variance_evolution}). Therefore, removing the coarsest level (i.e. setting $m_0 = 16$ instead) would not improve performance. 
\begin{figure}[h]
	\centering
	\begin{subfigure}[t]{.45\textwidth}
		\centering
		\sethw{0.5}{0.8}
		\input{fig/exp1/g_lvl_mid.tex}
		\caption{Iteration $k = 5$}
		\label{fig:exp1/g_lvl_mid}
	\end{subfigure}
	\quad
	\begin{subfigure}[t]{.45\textwidth}
		\centering
		\sethw{0.5}{0.8}
		\input{fig/exp1/g_lvl_opt.tex}
		\caption{Iteration $k = 18$ (last)}
		\label{fig:exp1/g_lvl_opt}
	\end{subfigure}
	\definecolor{mycolor1}{rgb}{0.00000,0.44700,0.74100}%
	\definecolor{mycolor2}{rgb}{0.85000,0.32500,0.09800}%
	\definecolor{mycolor3}{rgb}{0.92900,0.69400,0.12500}%
	\definecolor{mycolor4}{rgb}{0.49400,0.18400,0.55600}%
	\definecolor{mycolor5}{rgb}{0.46600,0.67400,0.18800}%
	\caption{Cross section of $\bs{g}^{(k)} = \sum_{\ell=0}^{L}\chlvl{\ell}{\Lmax}\YMC_{\ell,n_\ell}$ (\ref{tikz:g_total}) and of contributions $\chlvl{\ell}{\Lmax}\YMC_{\ell,n_\ell}$ at levels $0, \ldots, L$ (\ref{tikz:g_0},\ref{tikz:g_1},\ref{tikz:g_2},\ref{tikz:g_3},\ref{tikz:g_4}) for Problem 1.}
	\label{fig:exp1/g_lvl}
\end{figure}

\subsection{Problem 2}
Problem 2 is further defined by
\begin{equation}
\alpha = \num{1e-5}, \gamma = 0, \tau = \num{1e-4}, \sigma^2 = 0.5.
\end{equation}
Figure \ref{fig:exp2/conv} and \cref{tab:prob2.grad,tab:prob2.hess} describe the behavior of the optimization algorithms in the same way as before. The total time (with all overhead included) was $6973s$ (gradient based) and $5114s$ (Hessian based). For $\epsilon = \num{1e-4}$, taking $n_0$ samples at level $5$ instead, would take approximately $427$hrs for a single NCG iteration and $259$hrs for a single CG iteration.
\begin{table}[H]
	\vspace{-0.5cm}
	\caption{Behaviour of gradient based optimization (Algorithm \ref{alg:optim_grad}).}
	\label{tab:prob2.grad}
	\centering
	\vspace{-0.75cm}
	\begin{equation*}
	\small
	\begin{array}{l|lllllllll}
	k & \epsilon^{(k)} & n_0 & n_1 & n_2 & n_3 & n_4 & n_5 & \text{estimate of } \rho & t^{(k)} [s]  \\
	\hline
	0 & \num{1e-2} & 140 & 76 & 44 & & & & 2.0237& 2.06\\
	4 & \num{3.51e-4} & 35563 &	3220 & 136 & 28 & 20 &  & 1.5824 & 93.44\\
	9 & \num{1e-4} & 375256 & 38259 & 2082	& 135 &	21 & 16 & 1.5825 & 1092.71
	\end{array}
	\end{equation*}
	\vspace{-0.75cm}
\end{table}
\begin{table}[H]
	\vspace{-0.5cm}
	\caption{Behaviour of Hessian based optimization (Algorithm \ref{alg:optim_grad}).}
	\label{tab:prob2.hess}
	\centering
	\vspace{-0.75cm}
	\begin{equation*}
	\small
	\begin{array}{l|lllllllll}
	i & \epsilon^{(i)} & n_0 & n_1 & n_2 & n_3 & n_4 & n_5 & \text{estimate of } \rho & t^{(i)} [s]  \\
	\hline
	0 & \num{1e-2} & 140 & 76 & 44 & & & & 1.9102 & 8.89\\
	2 & \num{2e-3} & 393 & 76 & 44 & & & & 2.0975 & 11.77\\
	5 & \num{4e-4} & 33063 & 6980 &	193 & 28 & 20 &  & 1.7029 & 76.20\\
	10 & \num{1e-4} & 388834 & 37023 &	1747 & 255 & 56 & 16 & 1.7818 & 668.58
	\end{array}
	\end{equation*}
	\vspace{-0.75cm}
\end{table}
\begin{figure}[h]
	\centering
	\sethw{0.25}{0.4}
	\input{fig/exp2/conv.tex}
	\caption{Behavior of Algorithms \ref{alg:optim_grad} and \ref{alg:optim_hess} for Problem 2. The crosses indicate convergence tests performed using new samples. }
	\label{fig:exp2/conv}
\end{figure}

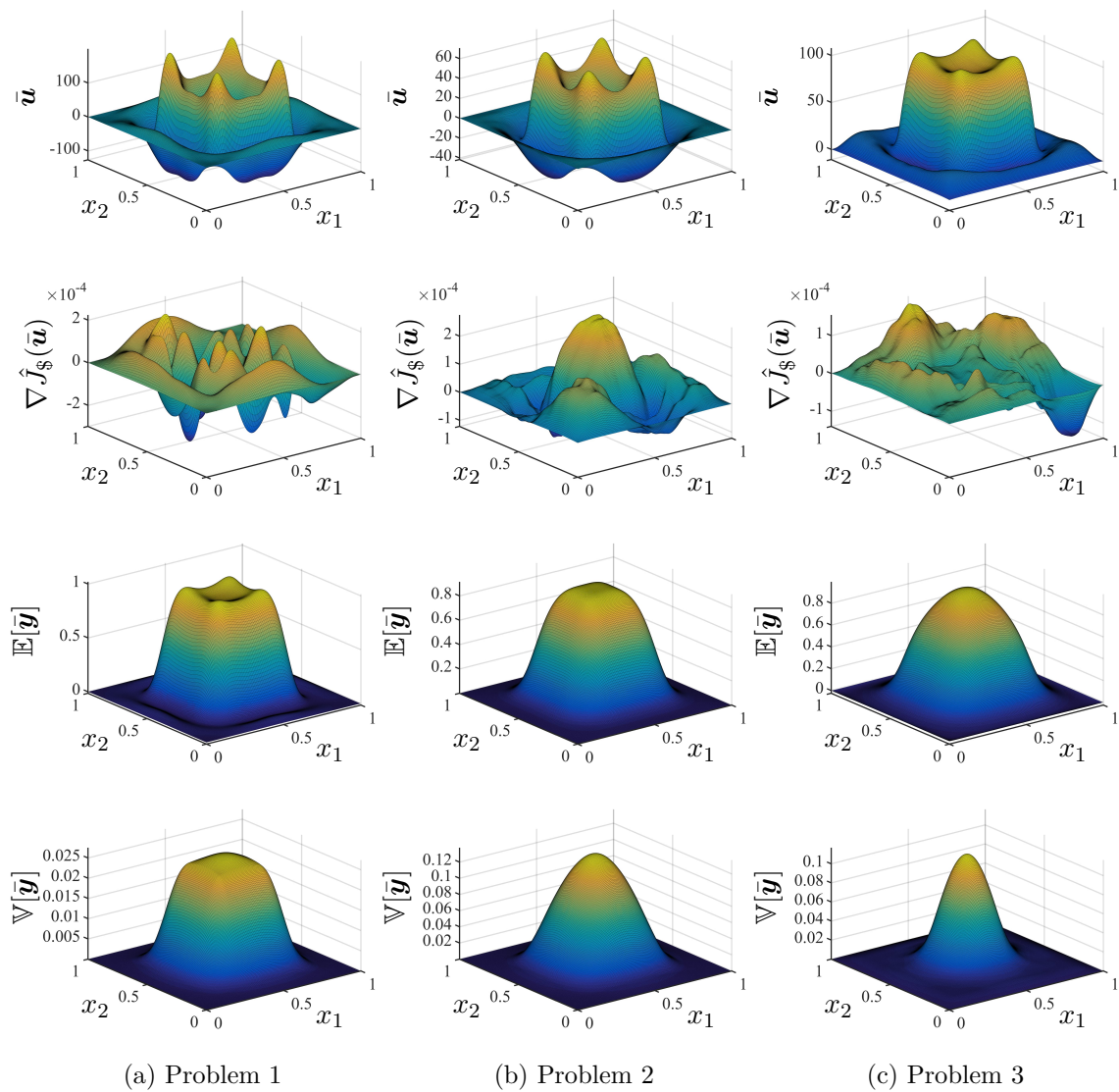
\begin{figure}[t]
	\hspace{-0.04\textwidth}
	\begin{subfigure}[t]{.36\textwidth}
		\centering
		\sethw{0.88}{0.88}
		\input{fig/exp1/u.tex} \\
		\input{fig/exp1/g.tex} \\
		\input{fig/exp1/Ey.tex} \\
		\input{fig/exp1/Vy.tex} \\
		\label{fig:exp1}
		\caption{Problem 1}
	\end{subfigure}
	\hspace{-0.04\textwidth}
	\begin{subfigure}[t]{.36\textwidth}
		\centering
		\sethw{0.88}{0.88}
		\input{fig/exp2/u.tex} \\
		\input{fig/exp2/g.tex} \\
		\input{fig/exp2/Ey.tex} \\
		\input{fig/exp2/Vy.tex} \\
		\label{fig:exp2}
		\caption{Problem 2}
	\end{subfigure}
	\hspace{-0.04\textwidth}
	\begin{subfigure}[t]{.36\textwidth}
		\centering
		\sethw{0.88}{0.88}
		\input{fig/exp3/u.tex} \\
		\input{fig/exp3/g.tex} \\
		\input{fig/exp3/Ey.tex} \\
		\input{fig/exp3/Vy.tex} \\
		\label{fig:exp3}
		\caption{Problem 3}
	\end{subfigure}
	\caption{Optimization results. $256 \times 256$.}
\end{figure}

\subsection{Problem 3}
\label{sec:num.3}
Consider as an example the following nonlinear extension to the model problem:
\begin{equation}
-\nabla \cdot(k\nabla y) + f(y) = \beta u \quad \mbox{on } D \quad \text{with} \quad y = 0 \quad \mbox{on } \boundary{D}.
\label{eq:c_nonlin}
\end{equation}
The function $f$ is some nonlinear reaction term.
From (\ref{eq:optcond_lagr_general}) it is clear that only the term $(\diffp{c}{y})^*\lm$ needs updating, leading to
\begin{equation}
\left\{\begin{array}{rcll}
-\nabla \cdot(k\nabla y)+ f(y)& = &\beta u &\quad \mbox{on } D\\
-\nabla \cdot(k\nabla p)+f'(y)p& = &2(y -y_D) + 2\gamma(y-\mean{y}) &\quad \mbox{on } D\\
\nabla \rJ(u) & = & 2\alpha u + \beta \mean{p}
\end{array}\right.\label{eq:gradient_nonlin}
\end{equation}
The expression $f'$ denotes the derivative of $f$, which is again localized ($f'(y)(x) = g'(y(x))$). The derivation of the Hessian equations is slightly more involved. We only state the result:
\begin{equation}
\left\{\begin{array}{rcll}
-\nabla \cdot(k\nabla \de{y})+ f'(y)\de{y} &=& \beta\de{u}. &\quad \mbox{on } D\\
-\nabla \cdot(k\nabla \de{p})+ f'(y)\de{p} +f''(y)p\de{y} &=& 2\de{y} + 2\gamma(y-\mean{\de{y}}) &\quad \mbox{on } D\\
\hess{\tilde{J}(u)}[\de{u}] & = & 2\alpha \de{u} + \beta \mean{\de{p}}
\end{array}\right.\label{eq:Hessian_nonlin}
\end{equation}
It is important to note that the adjoint equation is always linear in $p$ for the gradient and $\de{p}$ for the Hessian. Therefore, the sampling methods from Section \ref{section:samples} can always be used.

Consider as an example the nonlinear term $f(y) = 20 + e^{5y}$ and the parameters
\begin{equation}
\alpha = \num{1e-5}, \gamma = 1, \tau = \num{5e-5}, \sigma^2 = 0.5.
\end{equation}
Figure \ref{fig:exp3/conv} shows the convergence plot. Figure \ref{fig:exp3/g_lvl} shows a cross section of the contributions to the gradient on each level. The details obtained during the optimization  are described in \cref{tab:prob2.grad,tab:prob2.hess} for both algorithms. The total time (with all overhead included) was $2607s$ (gradient based) and $8307s$ (Hessian based). For $\epsilon = \num{5e-5}$, a single NCG and CG iteration would take approximately $73$hrs and $69$hrs respectively.
\begin{table}[H]
	\vspace{-0.5cm}
	\caption{Behaviour of gradient based optimization (Algorithm \ref{alg:optim_grad}).}
	\label{tab:prob3.grad}
	\centering
	\vspace{-0.75cm}
	\begin{equation*}
	\small
	\begin{array}{l|lllllllll}
	k & \epsilon^{(k)} & n_0 & n_1 & n_2 & n_3 & n_4 & n_5 & \text{estimate of } \rho & t^{(k)} [s]  \\
	\hline
	0 & \num{1e-2} & 923 & 83 & 44 & & & & 1.6309& 10.19\\
	7 & \num{7.32e-5} & 13369 &	3093 &	391 & 43 & 20 &	16 & 1.6069 & 341.36\\
	12 & \num{5e-5} & 33347	& 11435	& 711 &	77 & 20	& 16 & 1.5983 & 508.27
	\end{array}
	\end{equation*}
	\vspace{-0.75cm}
\end{table}
\begin{table}[H]
	\vspace{-0.5cm}
	\caption{Behaviour of Hessian based optimization (Algorithm \ref{alg:optim_grad}).}
	\label{tab:prob3.hess}
	\centering
	\vspace{-0.75cm}
	\begin{equation*}
	\small
	\begin{array}{l|lllllllll}
	i & \epsilon^{(i)} & n_0 & n_1 & n_2 & n_3 & n_4 & n_5 & \text{estimate of } \rho & t^{(i)} [s]  \\
	\hline
	0 & \num{1e-2} & 826   & 76   & 44  & & & & 1.9102 & 9.95\\
	2 & \num{2e-3} & 140   & 76   & 44 & & & & 2.0975 & 7.55\\
	4 & \num{4e-4} & 305   & 108  & 44   & 28 &  &  & 1.5521 & 14.54\\
	9 & \num{8e-5} & 14357 & 3211 & 205  & 28   & 20 &  & 1.5931 & 153.21\\
	16 & \num{5e-5} & 34181 & 7623 & 850  & 345  & 44 & 16 & 1.5964 & 593.68\\
	20 & \num{5e-5} & 24287 & 7394 & 2647 & 122  & 20 & 16 & 1.6211 & 727.11\\
	23 & \num{5e-5} & 26008 & 7058 & 1597 & 1057 & 20 & 16 & 1.5856 & 996.88\\
	\end{array}
	\end{equation*}
	\vspace{-0.75cm}
\end{table}

\begin{figure}[h]
	\centering
	\sethw{0.25}{0.4}
	\input{fig/exp3/conv.tex}
	\caption{Behavior of Algorithms \ref{alg:optim_grad} and \ref{alg:optim_hess} for Problem 3. The crosses indicate convergence tests performed using new samples. }
	\label{fig:exp3/conv}
\end{figure}

\begin{figure}[h]
	\centering
	\begin{subfigure}[t]{.45\textwidth}
		\centering
		\sethw{0.5}{0.8}
		\input{fig/exp3/g_lvl_mid.tex}
		\caption{Iteration $k = 7$}
		\label{fig:exp3/g_lvl_mid}
	\end{subfigure}
	\quad
	\begin{subfigure}[t]{.45\textwidth}
		\centering
		\sethw{0.5}{0.8}
		\input{fig/exp3/g_lvl_opt.tex}
		\caption{Iteration $k = 12$ (last)}
		\label{fig:exp3/g_lvl_opt}
	\end{subfigure}
	\definecolor{mycolor1}{rgb}{0.00000,0.44700,0.74100}%
	\definecolor{mycolor2}{rgb}{0.85000,0.32500,0.09800}%
	\definecolor{mycolor3}{rgb}{0.92900,0.69400,0.12500}%
	\definecolor{mycolor4}{rgb}{0.49400,0.18400,0.55600}%
	\definecolor{mycolor5}{rgb}{0.46600,0.67400,0.18800}%
	\definecolor{mycolor6}{rgb}{0.30100,0.74500,0.93300}%
	\caption{Cross section of $\bs{g}^{(k)} = \sum_{\ell=0}^{L}\chlvl{\ell}{\Lmax}\YMC_{\ell,n_\ell}$ (\ref{tikz:g_total}) and of contributions $\chlvl{\ell}{\Lmax}\YMC_{\ell,n_\ell}$ at levels $0, \ldots, L$ (\ref{tikz:g_0},\ref{tikz:g_1},\ref{tikz:g_2},\ref{tikz:g_3},\ref{tikz:g_4},\ref{tikz:g_5}) for Problem 3.}
	\label{fig:exp3/g_lvl}
\end{figure}

\section{Conclusions and further work}

We presented a MLMC method for solving the robust optimization problem for a tracking type cost functional. Including an additional penalty on the variance of the state allows to also solve the average control problem in the same framework. It has been shown that correlations between a limited number of samples of the gradient are hard to avoid if efficiency and correct error estimation are desired. The classical MLMC theory was extended to be able to deal with these samples. Since the correlations are usually negative, the usage of correlated samples turns out to reduce the number of samples required. 

The MLMC method proves to be orders of magnitude more efficient compared to the regular MC method, which takes all samples on the finest level. In our experiments, The NCG method is usually more performant than the Hessian based method, but not always. The performance depends on the specific problem considered. The Hessian based method may be better suited for a small number of optimization variables. The Newton equation could then be solved directly. 
The method was tested in this paper on simple academic model problems. Its performance for realistic applications has not yet been investigated.

Methods for optimization under uncertainties are often born out of previous research in simulation and uncertainty quantification of stochastic problems. 
In that context, the quasi-Monte Carlo (QMC) method \cite{graham2011, dick2013high} and its multilevel \cite{kuo2015multi, kuo2017multilevel} and multi-index \cite{robbe2016multi} variants have been shown to reduce the cost for a given RMSE $\epsilon$ from $\mathcal{O}(\epsilon^{-2})$ to, in ideal circumstances, $\mathcal{O}(\epsilon^{-1})$. Hence, the use of QMC is expected to reduce the complexity of (\ref{eq:tau_cost}) even further. The algorithmic details and the numerical evidence remain to be investigated. 



\appendix
\section{Gradient of $\Vest{1}$}
\label{app:derivative}
This section we provide the details of the derivation of the gradient of the variance estimator $\Vest{1}$, i.e., we show how to arrive at equation (\ref{eq:grad_V1}).
Note that $\| \sqrt{\bs{v}}\|^2 = (\bs{v},\bs{1})$, such that
\begin{align*}
\diff{}{\bs{y}} \Bigg\| \sqrt{\frac{1}{2n}\sum_{j=1}^{n}(\bs{y}_j - \bs{y}_{j-1})^2} \Bigg\|^2 [\bs{h}] 
&= \frac{1}{2n} \diff{}{\bs{y}} (\sum_{j=1}^{n}(\bs{y}_j - \bs{y}_{j-1})^2, \bs{1}) [\bs{h}].
\end{align*}
For simplicity, let $\bs{y}_i = \bs{y}_{n+i}$ and $\bs{h}_i = \bs{h}_{n+i}$. Using, e.g., the limit definition of the derivative yields
\begin{align*}
\frac{1}{2n}\diff{}{\bs{y}} (\sum_{j=1}^{n}(\bs{y}_j - \bs{y}_{j-1})^2, \bs{1}) [\bs{h}] 
& = \frac{1}{n}(\sum_{j=1}^{n}(\bs{y}_j - \bs{y}_{j-1})(\bs{h}_j - \bs{h}_{j-1}), \bs{1}) \\
& = \frac{1}{n}(\sum_{j=1}^{n}\bs{y}_j\bs{h}_j + \sum_{j=1}^{n}\bs{y}_{j-1}\bs{h}_{j-1} 
- \sum_{j=1}^{n}\bs{y}_j\bs{h}_{j-1} - \sum_{j=1}^{n}\bs{y}_{j-1}\bs{h}_j, \bs{1}) \\
& = \frac{1}{n}(\sum_{j=1}^{n}\bs{y}_j\bs{h}_j + \sum_{j=1}^{n}\bs{y}_{j}\bs{h}_{j} 
- \sum_{j=1}^{n}\bs{y}_{j+1}\bs{h}_{j} - \sum_{j=1}^{n}\bs{y}_{j-1}\bs{h}_j, \bs{1}) \\
& = \frac{1}{n}(\sum_{j=1}^{n}(2\bs{y}_j - \bs{y}_{j-1} - \bs{y}_{j+1})\bs{h}_j, \bs{1}) \\
& = (2\bs{y} - \bs{y}_{-1} - \bs{y}_{+1}, \bs{h})_{D,\Omega_0}
\end{align*}
Therefore, the gradient of $\Vest{1}$ w.r.t. $(.,.)_{D,\Omega_0}$ is equal to $2\bs{y} - \bs{y}_{-1} - \bs{y}_{+1}$.


\bibliographystyle{siamplain}
\bibliography{bib_avb}
\end{document}

%% file: fig/sample2D_var01.tex
%
%
\begin{tikzpicture}

\begin{axis}[%
width=\figurewidth,
height=\figureheight,
at={(0\figurewidth,0\figureheight)},
scale only axis,
point meta min=0.423231728673059,
point meta max=2.05412899839786,
axis on top,
xmin=0,
xmax=1,
xlabel style={font=\color{white!15!black}},
xlabel={$x_1$},
y dir=reverse,
ymin=0,
ymax=1,
ylabel near ticks,
ylabel style={font=\color{white!15!black}, rotate=-90},
ylabel={$x_2$},
axis background/.style={fill=white},
colormap={mymap}{[1pt] rgb(0pt)=(0.2081,0.1663,0.5292); rgb(1pt)=(0.211624,0.189781,0.577676); rgb(2pt)=(0.212252,0.213771,0.626971); rgb(3pt)=(0.2081,0.2386,0.677086); rgb(4pt)=(0.195905,0.264457,0.7279); rgb(5pt)=(0.170729,0.291938,0.779248); rgb(6pt)=(0.125271,0.324243,0.830271); rgb(7pt)=(0.0591333,0.359833,0.868333); rgb(8pt)=(0.0116952,0.38751,0.881957); rgb(9pt)=(0.00595714,0.408614,0.882843); rgb(10pt)=(0.0165143,0.4266,0.878633); rgb(11pt)=(0.0328524,0.443043,0.871957); rgb(12pt)=(0.0498143,0.458571,0.864057); rgb(13pt)=(0.0629333,0.47369,0.855438); rgb(14pt)=(0.0722667,0.488667,0.8467); rgb(15pt)=(0.0779429,0.503986,0.838371); rgb(16pt)=(0.0793476,0.520024,0.831181); rgb(17pt)=(0.0749429,0.537543,0.826271); rgb(18pt)=(0.0640571,0.556986,0.823957); rgb(19pt)=(0.0487714,0.577224,0.822829); rgb(20pt)=(0.0343429,0.596581,0.819852); rgb(21pt)=(0.0265,0.6137,0.8135); rgb(22pt)=(0.0238905,0.628662,0.803762); rgb(23pt)=(0.0230905,0.641786,0.791267); rgb(24pt)=(0.0227714,0.653486,0.776757); rgb(25pt)=(0.0266619,0.664195,0.760719); rgb(26pt)=(0.0383714,0.674271,0.743552); rgb(27pt)=(0.0589714,0.683757,0.725386); rgb(28pt)=(0.0843,0.692833,0.706167); rgb(29pt)=(0.113295,0.7015,0.685857); rgb(30pt)=(0.145271,0.709757,0.664629); rgb(31pt)=(0.180133,0.717657,0.642433); rgb(32pt)=(0.217829,0.725043,0.619262); rgb(33pt)=(0.258643,0.731714,0.595429); rgb(34pt)=(0.302171,0.737605,0.571186); rgb(35pt)=(0.348167,0.742433,0.547267); rgb(36pt)=(0.395257,0.7459,0.524443); rgb(37pt)=(0.44201,0.748081,0.503314); rgb(38pt)=(0.487124,0.749062,0.483976); rgb(39pt)=(0.530029,0.749114,0.466114); rgb(40pt)=(0.570857,0.748519,0.44939); rgb(41pt)=(0.609852,0.747314,0.433686); rgb(42pt)=(0.6473,0.7456,0.4188); rgb(43pt)=(0.683419,0.743476,0.404433); rgb(44pt)=(0.71841,0.741133,0.390476); rgb(45pt)=(0.752486,0.7384,0.376814); rgb(46pt)=(0.785843,0.735567,0.363271); rgb(47pt)=(0.818505,0.732733,0.34979); rgb(48pt)=(0.850657,0.7299,0.336029); rgb(49pt)=(0.882433,0.727433,0.3217); rgb(50pt)=(0.913933,0.725786,0.306276); rgb(51pt)=(0.944957,0.726114,0.288643); rgb(52pt)=(0.973895,0.731395,0.266648); rgb(53pt)=(0.993771,0.745457,0.240348); rgb(54pt)=(0.999043,0.765314,0.216414); rgb(55pt)=(0.995533,0.786057,0.196652); rgb(56pt)=(0.988,0.8066,0.179367); rgb(57pt)=(0.978857,0.827143,0.163314); rgb(58pt)=(0.9697,0.848138,0.147452); rgb(59pt)=(0.962586,0.870514,0.1309); rgb(60pt)=(0.958871,0.8949,0.113243); rgb(61pt)=(0.959824,0.921833,0.0948381); rgb(62pt)=(0.9661,0.951443,0.0755333); rgb(63pt)=(0.9763,0.9831,0.0538)},
colorbar
]
\addplot [forget plot] graphics [xmin=0, xmax=1, ymin=0, ymax=1] {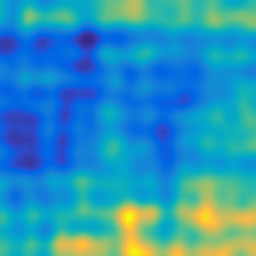};
\end{axis}
\end{tikzpicture}%

%% file: fig/sample2D_var05.tex
%
%
\begin{tikzpicture}

\begin{axis}[%
width=\figurewidth,
height=\figureheight,
at={(0\figurewidth,0\figureheight)},
scale only axis,
point meta min=0.20878826555389,
point meta max=5.05971381552426,
axis on top,
xmin=0,
xmax=1,
xlabel style={font=\color{white!15!black}},
xlabel={$x_1$},
y dir=reverse,
ymin=0,
ymax=1,
ylabel near ticks,
ylabel style={font=\color{white!15!black}, rotate=-90},
ylabel={$x_2$},
axis background/.style={fill=white},
colormap={mymap}{[1pt] rgb(0pt)=(0.2081,0.1663,0.5292); rgb(1pt)=(0.211624,0.189781,0.577676); rgb(2pt)=(0.212252,0.213771,0.626971); rgb(3pt)=(0.2081,0.2386,0.677086); rgb(4pt)=(0.195905,0.264457,0.7279); rgb(5pt)=(0.170729,0.291938,0.779248); rgb(6pt)=(0.125271,0.324243,0.830271); rgb(7pt)=(0.0591333,0.359833,0.868333); rgb(8pt)=(0.0116952,0.38751,0.881957); rgb(9pt)=(0.00595714,0.408614,0.882843); rgb(10pt)=(0.0165143,0.4266,0.878633); rgb(11pt)=(0.0328524,0.443043,0.871957); rgb(12pt)=(0.0498143,0.458571,0.864057); rgb(13pt)=(0.0629333,0.47369,0.855438); rgb(14pt)=(0.0722667,0.488667,0.8467); rgb(15pt)=(0.0779429,0.503986,0.838371); rgb(16pt)=(0.0793476,0.520024,0.831181); rgb(17pt)=(0.0749429,0.537543,0.826271); rgb(18pt)=(0.0640571,0.556986,0.823957); rgb(19pt)=(0.0487714,0.577224,0.822829); rgb(20pt)=(0.0343429,0.596581,0.819852); rgb(21pt)=(0.0265,0.6137,0.8135); rgb(22pt)=(0.0238905,0.628662,0.803762); rgb(23pt)=(0.0230905,0.641786,0.791267); rgb(24pt)=(0.0227714,0.653486,0.776757); rgb(25pt)=(0.0266619,0.664195,0.760719); rgb(26pt)=(0.0383714,0.674271,0.743552); rgb(27pt)=(0.0589714,0.683757,0.725386); rgb(28pt)=(0.0843,0.692833,0.706167); rgb(29pt)=(0.113295,0.7015,0.685857); rgb(30pt)=(0.145271,0.709757,0.664629); rgb(31pt)=(0.180133,0.717657,0.642433); rgb(32pt)=(0.217829,0.725043,0.619262); rgb(33pt)=(0.258643,0.731714,0.595429); rgb(34pt)=(0.302171,0.737605,0.571186); rgb(35pt)=(0.348167,0.742433,0.547267); rgb(36pt)=(0.395257,0.7459,0.524443); rgb(37pt)=(0.44201,0.748081,0.503314); rgb(38pt)=(0.487124,0.749062,0.483976); rgb(39pt)=(0.530029,0.749114,0.466114); rgb(40pt)=(0.570857,0.748519,0.44939); rgb(41pt)=(0.609852,0.747314,0.433686); rgb(42pt)=(0.6473,0.7456,0.4188); rgb(43pt)=(0.683419,0.743476,0.404433); rgb(44pt)=(0.71841,0.741133,0.390476); rgb(45pt)=(0.752486,0.7384,0.376814); rgb(46pt)=(0.785843,0.735567,0.363271); rgb(47pt)=(0.818505,0.732733,0.34979); rgb(48pt)=(0.850657,0.7299,0.336029); rgb(49pt)=(0.882433,0.727433,0.3217); rgb(50pt)=(0.913933,0.725786,0.306276); rgb(51pt)=(0.944957,0.726114,0.288643); rgb(52pt)=(0.973895,0.731395,0.266648); rgb(53pt)=(0.993771,0.745457,0.240348); rgb(54pt)=(0.999043,0.765314,0.216414); rgb(55pt)=(0.995533,0.786057,0.196652); rgb(56pt)=(0.988,0.8066,0.179367); rgb(57pt)=(0.978857,0.827143,0.163314); rgb(58pt)=(0.9697,0.848138,0.147452); rgb(59pt)=(0.962586,0.870514,0.1309); rgb(60pt)=(0.958871,0.8949,0.113243); rgb(61pt)=(0.959824,0.921833,0.0948381); rgb(62pt)=(0.9661,0.951443,0.0755333); rgb(63pt)=(0.9763,0.9831,0.0538)},
colorbar
]
\addplot [forget plot] graphics [xmin=0, xmax=1, ymin=0, ymax=1] {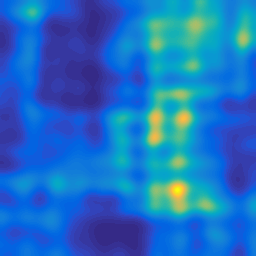};
\end{axis}
\end{tikzpicture}%

%% file: fig/grid1.tex
%
%
\begin{tikzpicture}

\begin{axis}[%
width=\figurewidth,
height=\figureheight,
at={(0\figurewidth,0\figureheight)},
scale only axis,
xmin=0,
xmax=1,
xlabel style={font=\color{white!15!black}},
xlabel={$x$},
ymin=0,
ymax=3,
ylabel near ticks,
ylabel style={font=\color{white!15!black}, rotate=-90},
ylabel={$\ell$},
axis background/.style={fill=white},
axis x line*=bottom,
axis y line*=left
]
\addplot[only marks, mark=*, mark options={}, mark size=1.5000pt, color=black, fill=black] table[row sep=crcr]{%
x	y\\
0.125	0\\
0.375	0\\
0.625	0\\
0.875	0\\
};
\addplot[only marks, mark=*, mark options={}, mark size=1.5000pt, color=black, fill=black] table[row sep=crcr]{%
x	y\\
0.0625	1\\
0.1875	1\\
0.3125	1\\
0.4375	1\\
0.5625	1\\
0.6875	1\\
0.8125	1\\
0.9375	1\\
};
\addplot[only marks, mark=*, mark options={}, mark size=1.5000pt, color=black, fill=black] table[row sep=crcr]{%
x	y\\
0.03125	2\\
0.09375	2\\
0.15625	2\\
0.21875	2\\
0.28125	2\\
0.34375	2\\
0.40625	2\\
0.46875	2\\
0.53125	2\\
0.59375	2\\
0.65625	2\\
0.71875	2\\
0.78125	2\\
0.84375	2\\
0.90625	2\\
0.96875	2\\
};
\addplot[only marks, mark=*, mark options={}, mark size=1.5000pt, color=black, fill=black] table[row sep=crcr]{%
x	y\\
0.015625	3\\
0.046875	3\\
0.078125	3\\
0.109375	3\\
0.140625	3\\
0.171875	3\\
0.203125	3\\
0.234375	3\\
0.265625	3\\
0.296875	3\\
0.328125	3\\
0.359375	3\\
0.390625	3\\
0.421875	3\\
0.453125	3\\
0.484375	3\\
0.515625	3\\
0.546875	3\\
0.578125	3\\
0.609375	3\\
0.640625	3\\
0.671875	3\\
0.703125	3\\
0.734375	3\\
0.765625	3\\
0.796875	3\\
0.828125	3\\
0.859375	3\\
0.890625	3\\
0.921875	3\\
0.953125	3\\
0.984375	3\\
};
\end{axis}
\end{tikzpicture}%

%% file: fig/mapping.tex
%
%
\begin{tikzpicture}

\begin{axis}[%
width=\figurewidth,
height=\figureheight,
at={(0\figurewidth,0\figureheight)},
scale only axis,
xmin=0,
xmax=1,
ymin=-1,
ymax=1,
xlabel style={font=\color{white!15!black}},
xlabel={$x$},
ylabel near ticks,
axis background/.style={fill=white},
]
\addplot [color=blue, mark=*, mark options={solid, fill=blue, blue}]
  table[row sep=crcr]{%
0.125	0.707106781186547\\
0.375	0.707106781186548\\
0.625	-0.707106781186547\\
0.875	-0.707106781186548\\
};

\addplot [color=green, mark=*, mark options={solid, fill=green, green}]
  table[row sep=crcr]{%
0.0625	0.353553390593274\\
0.1875	0.707106781186547\\
0.3125	0.707106781186548\\
0.4375	0.353553390593274\\
0.5625	-0.353553390593274\\
0.6875	-0.707106781186547\\
0.8125	-0.707106781186548\\
0.9375	-0.353553390593274\\
};

\addplot [color=red, mark=*, mark options={solid, fill=red, red}]
  table[row sep=crcr]{%
0.03125	0.176776695296637\\
0.09375	0.441941738241592\\
0.15625	0.618718433538229\\
0.21875	0.707106781186547\\
0.28125	0.707106781186548\\
0.34375	0.618718433538229\\
0.40625	0.441941738241592\\
0.46875	0.176776695296637\\
0.53125	-0.176776695296637\\
0.59375	-0.441941738241592\\
0.65625	-0.618718433538229\\
0.71875	-0.707106781186547\\
0.78125	-0.707106781186548\\
0.84375	-0.618718433538229\\
0.90625	-0.441941738241592\\
0.96875	-0.176776695296637\\
};

\end{axis}
\end{tikzpicture}%

%% file: fig/new_vs_fixed.tex
\definecolor{mycolor1}{rgb}{0.00000,0.44700,0.74100}%
\definecolor{mycolor2}{rgb}{0.85000,0.32500,0.09800}%
\definecolor{mycolor3}{rgb}{0.92900,0.69400,0.12500}%
\begin{tikzpicture}[baseline]

\begin{axis}[%
width=0.951\figurewidth,
height=\figureheight,
at={(0\figurewidth,0\figureheight)},
scale only axis,
xmin=0,
xmax=40,
xlabel style={font=\color{white!15!black}},
xlabel={iteration $k$},
ymode=log,
ymin=1e-05,
ymax=0.1,
yminorticks=true,
axis background/.style={fill=white}
]
\addplot [color=mycolor1, forget plot]
  table[row sep=crcr]{%
0	0.0435041281023918\\
1	0.00457589386362532\\
2	0.0012303193961584\\
3	0.00198354944443222\\
4	0.000681600569425512\\
5	0.00118999484980201\\
6	0.000620583208460862\\
7	0.000491259704438938\\
8	0.00057562290520128\\
9	0.000290570710672541\\
10	0.000427006221818035\\
11	0.000266842133892441\\
12	0.000230422214478471\\
13	0.000192069312542158\\
14	0.000297397159907632\\
15	0.000346977326308606\\
16	0.000121734557173113\\
17	0.000208471602708981\\
18	0.000251987751032512\\
19	0.000190423591160134\\
20	0.000183456746959438\\
21	0.00015028782238552\\
22	0.000287448526001508\\
23	8.01919690704077e-05\\
24	6.48061364038852e-05\\
25	0.000120714075658666\\
26	7.8004610985003e-05\\
27	0.000177997438154904\\
28	7.96262019540004e-05\\
29	0.000112802575036855\\
30	4.31652958448758e-05\\
31	0.000227438649799213\\
32	4.15752126083873e-05\\
33	6.31278798125758e-05\\
34	9.49747670175925e-05\\
35	4.47113031310873e-05\\
36	4.70144343474047e-05\\
37	2.61783604623309e-05\\
38	4.58391519182654e-05\\
39	3.62190072960656e-05\\
40	5.53639174645407e-05\\
};
\label{tikz:DYNCGfixed}

\addplot [color=mycolor2, forget plot]
  table[row sep=crcr]{%
0	0.0427088516670683\\
1	0.00446220543135468\\
2	0.00130751557097006\\
3	0.00209249512531425\\
4	0.000871732866367063\\
5	0.0015335475216593\\
6	0.000970928166907421\\
7	0.00067881812407504\\
8	0.000845693845567252\\
9	0.000383281394311167\\
10	0.00056820056549348\\
11	0.000378618973939186\\
12	0.000678371100940031\\
13	0.000580538935396865\\
14	0.000647100302323117\\
15	0.000860386649518173\\
16	0.000681043918743286\\
17	0.000465202539855689\\
18	0.000331665107777663\\
19	0.000472698250587955\\
20	0.000481084634637536\\
21	0.00186358969673393\\
22	0.000671745080392903\\
23	0.000462086077996553\\
24	0.00052182177923259\\
25	0.000752021523291336\\
26	0.000519196091384443\\
27	0.00052077445457585\\
28	0.000566282592148064\\
29	0.00103209280385629\\
30	0.000348274601780394\\
31	0.000343721794140679\\
32	0.000640130070645451\\
33	0.00123671930657265\\
34	0.000991909682619796\\
35	0.000389296120083134\\
36	0.000504445861985194\\
37	0.000628529090313135\\
38	0.000324127038133121\\
39	0.000355330676360389\\
40	0.000331665450084401\\ 
};
\label{tikz:reevaluation}

\addplot [color=mycolor3, forget plot]
  table[row sep=crcr]{%
0	0.0369385826290358\\
1	0.00991019654801363\\
2	0.00464793806703529\\
3	0.00491038193614514\\
4	0.00137009973314949\\
5	0.00249035256120037\\
6	0.00321613266029848\\
7	0.00242239096158975\\
8	0.00202205395997306\\
9	0.00105482810306755\\
10	0.00114803603307065\\
11	0.000779384369658431\\
12	0.000962758420921421\\
13	0.00100494975264006\\
14	0.00137249571104885\\
15	0.00119955067399861\\
16	0.0017364016961309\\
17	0.000818325100360348\\
18	0.000675959267745969\\
19	0.00223792949744706\\
20	0.000580203849795835\\
21	0.00043661390170538\\
22	0.000590963000547353\\
23	0.000818093448346747\\
24	0.000821631169078368\\
25	0.000750009809733129\\
26	0.000538374010117889\\
27	0.00102401309213627\\
28	0.000732703482138496\\
29	0.000930202417052265\\
30	0.00142683682540805\\
31	0.00057418678759361\\
32	0.00231449359834385\\
33	0.000964779264092382\\
34	0.000940515893093142\\
35	0.00120047889021187\\
36	0.00143150575471436\\
37	0.00117511714658783\\
38	0.000667451716796626\\
39	0.000932106551509511\\
40	0.000762598943384605\\ 
};
\label{tikz:DYNCGnew}

\addplot [color=black, dashed, forget plot]
table[row sep=crcr]{%
	0	0.000827599351237417\\
	1	0.00205614439733514\\
	2	0.00242624499748682\\
	3	0.00238620452798319\\
	4	0.00219862513299143\\
	5	0.00210741199390773\\
	6	0.00218604211291381\\
	7	0.00211115335931766\\
	8	0.00212670074103475\\
	9	0.00202060513518012\\
	10	0.00212996221508633\\
	11	0.00214609029645053\\
	12	0.00230236028426964\\
	13	0.00228373197802035\\
	14	0.00239534025447321\\
	15	0.00243151518121181\\
	16	0.00258367498680898\\
	17	0.0026000372500899\\
	18	0.00267962878730141\\
	19	0.00272276605187926\\
	20	0.00282407998118152\\
	21	0.00281636896114632\\
	22	0.00294054168510232\\
	23	0.0029907890281915\\
	24	0.00305067201383186\\
	25	0.00305944924541278\\
	26	0.00307023306377016\\
	27	0.00309448611763064\\
	28	0.00313038366810226\\
	29	0.00313650233699888\\
	30	0.00315153989541404\\
	31	0.00315953930948542\\
	32	0.00314717738160748\\
	33	0.00314067018548639\\
	34	0.00313817355813328\\
	35	0.0031438297927609\\
	36	0.00313591345609659\\
	37	0.00313589472749841\\
	38	0.00313702594449578\\
	39	0.00314153268101188\\
	40	0.00313877820637354\\
};
\label{tikz:RMSE}

\end{axis}
\end{tikzpicture}%

%% file: fig/variance_evolution2.tex
%
%
\definecolor{mycolor1}{rgb}{0.00000,0.44700,0.74100}%
\definecolor{mycolor2}{rgb}{0.85000,0.32500,0.09800}%
\definecolor{mycolor3}{rgb}{0.92900,0.69400,0.12500}%
\definecolor{mycolor4}{rgb}{0.49400,0.18400,0.55600}%
\definecolor{mycolor5}{rgb}{0.46600,0.67400,0.18800}%
\definecolor{mycolor6}{rgb}{0.30100,0.74500,0.93300}%
\definecolor{black}{rgb}{0,0,0}%
\begin{tikzpicture}[baseline]

\begin{axis}[%
width=0.951\figurewidth,
height=\figureheight,
at={(0\figurewidth,0\figureheight)},
scale only axis,
xmin=0,
xmax=40,
xlabel style={font=\color{white!15!black}},
xlabel={iteration $k$},
ymode=log,
ymin=1e-10,
ymax=0.01,
yminorticks=true,
axis background/.style={fill=white},
]

\addplot [color=black, dashed]
table[row sep=crcr]{%
	0	0.1\\
};
\label{tikz:vars}

\addplot [color=black]
table[row sep=crcr]{%
	0	0.1\\
};
\label{tikz:covars}

\addplot [color=mycolor1]
  table[row sep=crcr]{%
0	0.000958277861407599\\
1	0.00175812942754534\\
2	0.0038369925984216\\
3	0.00356751738421806\\
4	0.00301618835229238\\
5	0.00273037292117247\\
6	0.00281072791683933\\
7	0.00256161411898562\\
8	0.00259555708189916\\
9	0.0024376248334732\\
10	0.00268896032646762\\
11	0.0028490969926219\\
12	0.00303554160687639\\
13	0.00289864127787448\\
14	0.00303071713807626\\
15	0.00295483978834699\\
16	0.0031524939321494\\
17	0.00297435778591768\\
18	0.0030632649939121\\
19	0.00304496663954438\\
20	0.00308715289065896\\
21	0.00296537971928819\\
22	0.00303740652793668\\
23	0.00312077577294265\\
24	0.00310970283918452\\
25	0.00304158299199419\\
26	0.00307686733696448\\
27	0.00309809757289964\\
28	0.00304044734314245\\
29	0.00304415827532788\\
30	0.0030715069706845\\
31	0.00310884765654905\\
32	0.00306520154531318\\
33	0.00302869956563054\\
34	0.00305686706670735\\
35	0.00308302814763958\\
36	0.00305312715961504\\
37	0.00307884670490801\\
38	0.00309290096932663\\
39	0.00308461436253936\\
40	0.00307098929499944\\
};

\addplot [color=mycolor2]
  table[row sep=crcr]{%
0	3.45823551688919e-06\\
1	8.35726730464657e-05\\
2	0.000199726222009501\\
3	0.000183644227249833\\
4	0.000180994632268205\\
5	0.000170735371230882\\
6	0.000175541480942851\\
7	0.000152664211113802\\
8	0.000160939154078771\\
9	0.000159474885928151\\
10	0.000174329676756716\\
11	0.000186670813166647\\
12	0.000186590192523123\\
13	0.000180143330513431\\
14	0.000183111067120038\\
15	0.000173724426934177\\
16	0.000173891162206876\\
17	0.000163080181473296\\
18	0.000169071158630356\\
19	0.000164561781350861\\
20	0.000165591917156524\\
21	0.000162898357804695\\
22	0.000170063552682748\\
23	0.000173248635066339\\
24	0.000173015615071125\\
25	0.000169922711595599\\
26	0.000170817809577368\\
27	0.000171796312892155\\
28	0.000170582866338062\\
29	0.00016989297798091\\
30	0.000173209240271384\\
31	0.000174498306980415\\
32	0.000173524532072095\\
33	0.000171045357246566\\
34	0.000170756732819984\\
35	0.000171533058146521\\
36	0.000170919569881244\\
37	0.000173357700427874\\
38	0.000172683062035541\\
39	0.000171500945059545\\
40	0.000171232244595816\\
};

\addplot [color=mycolor3, forget plot]
  table[row sep=crcr]{%
0	2.51430013163141e-07\\
1	8.3627353649492e-07\\
2	2.60872866273919e-07\\
3	2.19263970931162e-07\\
4	2.9311881170684e-07\\
5	2.9016371672144e-07\\
6	3.27292423645874e-07\\
7	3.62961987251846e-07\\
8	2.93083263794655e-07\\
9	2.94115280865964e-07\\
10	3.17006462366817e-07\\
11	3.10984886351615e-07\\
12	3.15254665059773e-07\\
13	2.99055027129969e-07\\
14	2.982954746214e-07\\
15	3.12271904852895e-07\\
16	3.26448608164372e-07\\
17	3.03473079811568e-07\\
18	4.01335620993502e-07\\
19	4.53178841223101e-07\\
20	5.32571795929546e-07\\
21	5.57362550565138e-07\\
22	6.62570890427229e-07\\
23	6.69974066295062e-07\\
24	7.00429601391108e-07\\
25	7.18618663496247e-07\\
26	7.18133111528651e-07\\
27	7.04059692917604e-07\\
28	6.27946764558704e-07\\
29	5.99501172155929e-07\\
30	5.65977488632812e-07\\
31	5.60897732916598e-07\\
32	5.55026808143573e-07\\
33	5.55766080562477e-07\\
34	6.16792149526708e-07\\
35	6.15412444482388e-07\\
36	6.14368715879359e-07\\
37	5.99600273390039e-07\\
38	5.75582480046676e-07\\
39	5.58833194117518e-07\\
40	5.57919973485283e-07\\
};
\addplot [color=mycolor4, forget plot]
  table[row sep=crcr]{%
0	3.23583085026765e-08\\
1	3.35930863498675e-08\\
2	3.19978134811557e-08\\
3	3.20225179034179e-08\\
4	3.03304750868858e-08\\
5	2.72954452124146e-08\\
6	3.0302744826175e-08\\
7	3.39385868276111e-08\\
8	2.95830417448796e-08\\
9	2.69977946594745e-08\\
10	3.04466054439713e-08\\
11	3.39395110458761e-08\\
12	3.7606091081721e-08\\
13	3.56173116361072e-08\\
14	3.57992512514956e-08\\
15	3.6829853013933e-08\\
16	4.06207530876533e-08\\
17	3.92768705877939e-08\\
18	3.74867918283219e-08\\
19	3.63097338319008e-08\\
20	3.63372220486515e-08\\
21	3.55014153661589e-08\\
22	3.70061123166516e-08\\
23	3.73197493279377e-08\\
24	3.80718969024352e-08\\
25	3.75419549236493e-08\\
26	3.788168913232e-08\\
27	3.96382377154172e-08\\
28	4.02993168204645e-08\\
29	4.05340374693093e-08\\
30	4.16681451818033e-08\\
31	4.20227302996983e-08\\
32	4.18191119642248e-08\\
33	4.20347392477055e-08\\
34	4.21638145115879e-08\\
35	4.24894227562185e-08\\
36	4.23168539087807e-08\\
37	4.30382229460334e-08\\
38	4.37209079160355e-08\\
39	4.46559745276772e-08\\
40	4.45718227822777e-08\\
};
\addplot [color=mycolor5, forget plot]
  table[row sep=crcr]{%
0	2.3096562006927e-09\\
1	7.47116295163583e-10\\
2	1.23805337089986e-09\\
3	1.33730068026998e-09\\
4	1.2672323461574e-09\\
5	1.20718648774824e-09\\
6	1.35697175508269e-09\\
7	1.44932539928675e-09\\
8	2.00130263665517e-09\\
9	2.0482460953005e-09\\
10	2.14030137938719e-09\\
11	2.33317324295676e-09\\
12	2.65129319529313e-09\\
13	2.68310654910057e-09\\
14	3.34575771893089e-09\\
15	4.18362486912051e-09\\
16	5.24254629682065e-09\\
17	5.95723643492449e-09\\
18	6.86324328348652e-09\\
19	7.4658225805113e-09\\
20	8.03125005880282e-09\\
21	8.18593957213453e-09\\
22	8.8975581156568e-09\\
23	9.32470214053304e-09\\
24	9.87021006396751e-09\\
25	1.01267110426117e-08\\
26	1.01442839058162e-08\\
27	1.02758577457939e-08\\
28	1.05967781958764e-08\\
29	1.06700600118932e-08\\
30	1.07423399521035e-08\\
31	1.07432028501692e-08\\
32	1.07465034442309e-08\\
33	1.07524681043175e-08\\
34	1.06771249090131e-08\\
35	1.06338803868479e-08\\
36	1.06257595877833e-08\\
37	1.0594250348414e-08\\
38	1.05370290595114e-08\\
39	1.04873803042711e-08\\
40	1.04909429453222e-08\\
};
\addplot [color=mycolor6, forget plot]
  table[row sep=crcr]{%
0	2.10094164076092e-10\\
1	6.76632993594284e-10\\
2	4.35125760861461e-10\\
3	4.30852962166774e-10\\
4	5.32005291537762e-10\\
5	5.38790941749474e-10\\
6	7.29472981871686e-10\\
7	7.81153459362023e-10\\
8	9.52465067784072e-10\\
9	9.90311371938344e-10\\
10	1.07822876277881e-09\\
11	1.2041570611943e-09\\
12	1.40371167284519e-09\\
13	1.41064417400123e-09\\
14	1.49527081241312e-09\\
15	1.5295317672898e-09\\
16	1.7024267342618e-09\\
17	1.82467707602157e-09\\
18	2.03626146254477e-09\\
19	2.13257474215011e-09\\
20	2.24289639871218e-09\\
21	2.25389752042684e-09\\
22	2.41609179503607e-09\\
23	2.48115425221406e-09\\
24	2.51780517327536e-09\\
25	2.5538106378724e-09\\
26	2.56207230436227e-09\\
27	2.59353732264726e-09\\
28	2.60927547564858e-09\\
29	2.60552802963085e-09\\
30	2.60756225465951e-09\\
31	2.61120804461223e-09\\
32	2.60307538871203e-09\\
33	2.59652877471726e-09\\
34	2.55452838859903e-09\\
35	2.53422982365404e-09\\
36	2.5244138324586e-09\\
37	2.49365962154504e-09\\
38	2.47364993150005e-09\\
39	2.46157856274159e-09\\
40	2.45928954895566e-09\\
};
\addplot [color=mycolor1, dashed, forget plot]
  table[row sep=crcr]{%
0	0.000918643833021419\\
1	0.00291384694037626\\
2	0.00582227180246206\\
3	0.00546180179230364\\
4	0.004659894795898\\
5	0.0042324511172023\\
6	0.00434760748702154\\
7	0.00398651247978395\\
8	0.00405194332036776\\
9	0.00380990713176207\\
10	0.00417922806344379\\
11	0.00439517307564856\\
12	0.00467295059373514\\
13	0.0044692697279885\\
14	0.00465896792236697\\
15	0.00456147293265808\\
16	0.00484994121227456\\
17	0.00458208976197452\\
18	0.00471252099846698\\
19	0.00468406220036562\\
20	0.00474986177925081\\
21	0.00456853092925158\\
22	0.00466573678315229\\
23	0.00478149163122295\\
24	0.00477544573211853\\
25	0.00468124346969863\\
26	0.00473372117637891\\
27	0.00476116487269606\\
28	0.00468226564449425\\
29	0.00468552958125718\\
30	0.00472708111015292\\
31	0.00478199716128643\\
32	0.00471665497532155\\
33	0.00466482309481536\\
34	0.00470577778159167\\
35	0.00474549352489283\\
36	0.00470103642727542\\
37	0.00473734673511623\\
38	0.00475540044318385\\
39	0.00474569293719671\\
40	0.00472541154661264\\
};
\addplot [color=mycolor2, dashed, forget plot]
  table[row sep=crcr]{%
0	3.4316109781774e-06\\
1	9.17690463434844e-05\\
2	0.000216872515823233\\
3	0.000199973040377779\\
4	0.000197191504737332\\
5	0.000186084234789393\\
6	0.000191373482297741\\
7	0.000167559506366707\\
8	0.000176527165923421\\
9	0.000174766640222582\\
10	0.000190746555667074\\
11	0.000203916873116722\\
12	0.000203861144084128\\
13	0.000196860062352523\\
14	0.000199840485281773\\
15	0.000189861269220793\\
16	0.000190158807172295\\
17	0.000181242911609994\\
18	0.000187434950499735\\
19	0.000182604912212439\\
20	0.000183922107562361\\
21	0.000180927873849351\\
22	0.000188666739356173\\
23	0.000191976205028003\\
24	0.000191658396307173\\
25	0.000188242721033856\\
26	0.00018924667355721\\
27	0.000190480998427011\\
28	0.000189350033464623\\
29	0.000188706956614723\\
30	0.000192521479933057\\
31	0.000193972110185514\\
32	0.000192931415181744\\
33	0.000190361512603469\\
34	0.000190125936624504\\
35	0.000191005740638277\\
36	0.000190310886249682\\
37	0.000192905279237127\\
38	0.000192211908693989\\
39	0.000191008107027723\\
40	0.000190706019988005\\
};
\addplot [color=mycolor3, dashed, forget plot]
  table[row sep=crcr]{%
0	1.87179538193628e-07\\
1	9.11519624367552e-07\\
2	2.82815754469345e-07\\
3	2.45551809146355e-07\\
4	4.12863923001879e-07\\
5	3.9911957211596e-07\\
6	3.77801231860556e-07\\
7	3.87803455327813e-07\\
8	3.50354771168187e-07\\
9	3.51498729617556e-07\\
10	3.80831889884649e-07\\
11	3.70759045526226e-07\\
12	3.28608155136962e-07\\
13	3.11295873585917e-07\\
14	3.11735947378484e-07\\
15	3.36907069891515e-07\\
16	3.88785587566538e-07\\
17	4.24085584584985e-07\\
18	5.47655941595017e-07\\
19	6.14506314708454e-07\\
20	7.14154231833139e-07\\
21	7.41404558079673e-07\\
22	8.66200200440654e-07\\
23	8.76904446467233e-07\\
24	9.06742666268412e-07\\
25	9.28501297216859e-07\\
26	9.28759515670911e-07\\
27	9.09912066535642e-07\\
28	8.16619496449387e-07\\
29	7.83875557286882e-07\\
30	7.4101842101429e-07\\
31	7.35865542731862e-07\\
32	7.28194665965978e-07\\
33	7.28126794943045e-07\\
34	8.01037899313936e-07\\
35	7.98283558851144e-07\\
36	7.95781209746444e-07\\
37	7.74855731166002e-07\\
38	7.43969235924938e-07\\
39	7.1904192939584e-07\\
40	7.17418698701407e-07\\
};
\addplot [color=mycolor4, dashed, forget plot]
  table[row sep=crcr]{%
0	4.1064066462651e-08\\
1	4.32223278746453e-08\\
2	5.10120298241675e-08\\
3	6.40450358068358e-08\\
4	6.06609501737716e-08\\
5	5.45908904248292e-08\\
6	6.060548965235e-08\\
7	5.84623186282649e-08\\
8	5.71867535365412e-08\\
9	4.93616172429098e-08\\
10	5.21157818187427e-08\\
11	5.92911002108945e-08\\
12	6.49541474123739e-08\\
13	6.35487455098712e-08\\
14	6.87224630031688e-08\\
15	7.26199495038834e-08\\
16	7.91780649107433e-08\\
17	7.85537411755878e-08\\
18	7.49735836566437e-08\\
19	7.26194676638016e-08\\
20	7.26744440973029e-08\\
21	7.10028307323178e-08\\
22	7.40122246333032e-08\\
23	7.46394986558754e-08\\
24	7.61437938048705e-08\\
25	7.50839098472987e-08\\
26	7.576337826464e-08\\
27	7.92764754308345e-08\\
28	8.05986336409289e-08\\
29	8.10680749386187e-08\\
30	8.33362903636066e-08\\
31	8.40454605993966e-08\\
32	8.36382239284495e-08\\
33	8.40694784954111e-08\\
34	8.43276290231759e-08\\
35	8.4978845512437e-08\\
36	8.46337078175614e-08\\
37	8.60764458920668e-08\\
38	8.74418158320711e-08\\
39	8.93119490553545e-08\\
40	8.91436455645553e-08\\
};
\addplot [color=mycolor5, dashed, forget plot]
  table[row sep=crcr]{%
0	3.30833676991051e-09\\
1	1.09072655294221e-09\\
2	1.96189203066433e-09\\
3	2.67460136053996e-09\\
4	2.53446469231481e-09\\
5	2.41437297549648e-09\\
6	2.71394351016539e-09\\
7	2.8986507985735e-09\\
8	4.00260527331034e-09\\
9	4.09649219060099e-09\\
10	4.28060275877437e-09\\
11	4.66634648591352e-09\\
12	5.30258639058626e-09\\
13	5.36621309820115e-09\\
14	6.69151543786177e-09\\
15	8.36724973824102e-09\\
16	1.04850925936413e-08\\
17	1.1914472869849e-08\\
18	1.3726486566973e-08\\
19	1.49316451610226e-08\\
20	1.60625001176056e-08\\
21	1.63718791442691e-08\\
22	1.77951162313136e-08\\
23	1.86494042810661e-08\\
24	1.9740420127935e-08\\
25	2.02534220852234e-08\\
26	2.02885678116324e-08\\
27	2.05517154915879e-08\\
28	2.11935563917529e-08\\
29	2.13401200237863e-08\\
30	2.14846799042071e-08\\
31	2.14864057003383e-08\\
32	2.14930068884618e-08\\
33	2.15049362086349e-08\\
34	2.13542498180262e-08\\
35	2.12677607736959e-08\\
36	2.12515191755666e-08\\
37	2.1188500696828e-08\\
38	2.10740581190227e-08\\
39	2.09747606085422e-08\\
40	2.09818858906445e-08\\
};
\addplot [color=mycolor6, dashed, forget plot]
  table[row sep=crcr]{%
0	4.20188328152184e-10\\
1	1.35326598718857e-09\\
2	8.70251521722922e-10\\
3	8.61705924333547e-10\\
4	1.06401058307552e-09\\
5	1.02299354430285e-09\\
6	1.19488128501518e-09\\
7	1.17958785615967e-09\\
8	1.49452335272288e-09\\
9	1.517876525059e-09\\
10	1.69691785965218e-09\\
11	1.91967353764487e-09\\
12	2.17347766339236e-09\\
13	2.149314014958e-09\\
14	2.59537694853708e-09\\
15	2.90131517591257e-09\\
16	3.4048534685236e-09\\
17	3.64935415204315e-09\\
18	4.07252292508955e-09\\
19	4.26514948430022e-09\\
20	4.48579279742436e-09\\
21	4.50779504085367e-09\\
22	4.83218359007214e-09\\
23	4.96230850442812e-09\\
24	5.03561034655071e-09\\
25	5.10762127574481e-09\\
26	5.12414460872454e-09\\
27	5.18707464529453e-09\\
28	5.21855095129715e-09\\
29	5.21105605926171e-09\\
30	5.21512450931901e-09\\
31	5.22241608922446e-09\\
32	5.20615077742406e-09\\
33	5.19305754943451e-09\\
34	5.10905677719806e-09\\
35	5.06845964730807e-09\\
36	5.0488276649172e-09\\
37	4.98731924309007e-09\\
38	4.9472998630001e-09\\
39	4.92315712548318e-09\\
40	4.91857909791133e-09\\
};
\end{axis}
\end{tikzpicture}%

%% file: fig/exp1/conv.tex
%
%
\definecolor{mycolor1}{rgb}{0.00000,0.44700,0.74100}%
\definecolor{mycolor2}{rgb}{0.85000,0.32500,0.09800}%
\definecolor{mycolor3}{rgb}{0.92900,0.69400,0.12500}%
\definecolor{mycolor4}{rgb}{0.49400,0.18400,0.55600}%
\definecolor{mycolor5}{rgb}{0.46600,0.67400,0.18800}%
\begin{tikzpicture}

\begin{axis}[%
width=\figurewidth,
height=\figureheight,
at={(0\figurewidth,0\figureheight)},
scale only axis,
xmin=0,
xmax=25,
xlabel style={font=\color{white!15!black}},
xlabel={iteration $k$ (NCG) or $i$ (CG)},
ymode=log,
ymin=6e-05,
ymax=0.1,
yminorticks=true,
axis background/.style={fill=white},
legend style={legend cell align=left, align=left, draw=white!15!black}
]
\addplot [color=mycolor1]
  table[row sep=crcr]{%
0	0.0411935332566013\\
1	0.00391544339653357\\
2	0.00723960300757659\\
3	0.00112045403731143\\
4	0.00149441057005802\\
5	0.00151773158007462\\
6	0.00120501652572579\\
7	0.000912049632018788\\
8	0.000409179044355227\\
9	0.000882801696203746\\
10	0.00146716215161169\\
11	0.000528535962586928\\
12	0.000369145266086674\\
13	0.00125316237936591\\
14	0.000125412514034196\\
15	0.00034324980360404\\
16	0.000354347983324125\\
17	0.000202595295635027\\
18	7.96149268914166e-05\\
};
\addlegendentry{$\|\bs{g}^{(k)}\|$}

\addplot [color=mycolor1, dashed]
  table[row sep=crcr]{%
0	0.01\\
1	0.0016340851859365\\
2	0.00164729220507308\\
3	0.00215289997049437\\
4	0.000224090807462286\\
5	0.00021497361304527\\
6	0.000205183005921987\\
7	0.000199790914216185\\
8	0.00022318789364201\\
9	0.000221930804341641\\
10	0.000229439600956448\\
11	0.00022371041880346\\
12	0.000223775997821223\\
13	0.000233079719333343\\
14	0.000228507322372466\\
15	0.0001\\
16	9.75530948086056e-05\\
17	9.82517167814688e-05\\
18	9.76442688802658e-05\\
};
\addlegendentry{$\epsilon^{(k)}$}

\addplot[only marks, mark=x, mark options={}, mark size=2.5000pt, draw=mycolor1, forget plot] table[row sep=crcr]{%
	x	y\\
	18	9.913e-05\\
};

\addplot [color=mycolor2]
  table[row sep=crcr]{%
0	0.0419553177976711\\
1	0.00390008826032591\\
2	0.00389731697381657\\
3	0.00186809102822704\\
4	0.00206007295578209\\
5	0.00146144307324256\\
6	0.0014298356672655\\
7	0.00084619768314174\\
8	0.00065593761901815\\
9	0.000649435043963648\\
10	0.000771595295128642\\
11	0.000713022016569242\\
12	0.00037566750499135\\
13	0.000365452468520957\\
14	0.00117129904377967\\
15	0.000525530459510173\\
16	0.000368445044537775\\
17	0.000188055841053145\\
18	0.000529529525737019\\
19	0.000354400377044961\\
20	0.000167088312309857\\
21	0.000142231633026159\\
22	0.000131018164885256\\
23	0.000158133219523619\\
24	8.16162635457512e-05\\
};
\addlegendentry{$\|\bs{r}^{(i)}\|$}

\addplot [color=mycolor2, dashed]
  table[row sep=crcr]{%
0	0.01\\
1	0.01\\
2	0.002\\
3	0.002\\
4	0.0004\\
5	0.0004\\
6	0.0004\\
7	0.0004\\
8	0.0004\\
9	0.0004\\
10	0.0004\\
11	0.0004\\
12	0.0004\\
13	0.0001\\
14	0.0001\\
15	0.0001\\
16	0.0001\\
17	0.0001\\
18	0.0001\\
19	0.0001\\
20	0.0001\\
21	0.0001\\
22	0.0001\\
23	0.0001\\
24	0.0001\\
};
\addlegendentry{$\epsilon^{(i)}$}

\addplot[only marks, mark=x, mark options={}, mark size=2.5000pt, draw=mycolor2, forget plot] table[row sep=crcr]{%
	x	y\\
	24	9.117e-05\\
};
\end{axis}
\end{tikzpicture}%

%% file: fig/exp2/conv.tex
%
%
\definecolor{mycolor1}{rgb}{0.00000,0.44700,0.74100}%
\definecolor{mycolor2}{rgb}{0.85000,0.32500,0.09800}%
\definecolor{mycolor3}{rgb}{0.92900,0.69400,0.12500}%
\definecolor{mycolor4}{rgb}{0.49400,0.18400,0.55600}%
\definecolor{mycolor5}{rgb}{0.46600,0.67400,0.18800}%
\definecolor{mycolor6}{rgb}{0.30100,0.74500,0.93300}%
\begin{tikzpicture}

\begin{axis}[%
width=\figurewidth,
height=\figureheight,
at={(0\figurewidth,0\figureheight)},
scale only axis,
xmin=0,
xmax=16,
xlabel style={font=\color{white!15!black}},
xlabel={iteration $k$ (NCG) or $i$ (CG)},
ymode=log,
ymin=6e-05,
ymax=0.1,
yminorticks=true,
axis background/.style={fill=white},
legend style={legend cell align=left, align=left, draw=white!15!black}
]
\addplot [color=mycolor1]
  table[row sep=crcr]{%
0	0.0409571692293628\\
1	0.00408813719351592\\
2	0.00604874404273697\\
3	0.00175292845153977\\
4	0.00147831667153059\\
5	0.00132256246945104\\
6	0.0013053682835006\\
7	0.000858035404812844\\
8	0.000252735211722797\\
9	0.000181642816888737\\
10	0.000181252031606748\\
11	0.000288641299175944\\
12	0.000177566155148581\\
13	8.81952803747017e-05\\
14	6.59479991348626e-05\\
};
\addlegendentry{$\|\bs{g}^{(k)}\|$}

\addplot [color=mycolor1, dashed]
  table[row sep=crcr]{%
0	0.01\\
1	0.00212063820836367\\
2	0.00230556947169498\\
3	0.00315321795553341\\
4	0.000350585690307954\\
5	0.000303312348170532\\
6	0.000287413767534894\\
7	0.000285889959105501\\
8	0.000304930312739127\\
9	0.0001\\
10	9.47703423147316e-05\\
11	9.4100177634015e-05\\
12	9.49845362048961e-05\\
13	9.49000276134469e-05\\
14	9.51608985258692e-05\\
};
\addlegendentry{$\epsilon^{(k)}$}

\addplot[only marks, mark=x, mark options={}, mark size=2.5000pt, draw=mycolor1, forget plot] table[row sep=crcr]{%
x	y\\
13	0.000101329532160805\\
14	7.37948500924026e-05\\
};

\addplot [color=mycolor2]
  table[row sep=crcr]{%
0	0.0433259574428281\\
1	0.00417933751093853\\
2	0.00446889148849073\\
3	0.00731372212883611\\
4	0.0011213982529147\\
5	0.00165442931884132\\
6	0.00169042114460802\\
7	0.00147497615798193\\
8	0.000823605016331796\\
9	0.000354938180547665\\
10	0.000411780161910956\\
11	0.00061366119021541\\
12	0.000343849117239164\\
13	0.000423702513352949\\
14	0.000176317344134711\\
15	0.000121260117034605\\
16	5.52269858676175e-05\\
};
\addlegendentry{$\|\bs{r}^{(i)}\|$}

\addplot [color=mycolor2, dashed]
  table[row sep=crcr]{%
0	0.01\\
1	0.01\\
2	0.002\\
3	0.002\\
4	0.002\\
5	0.0004\\
6	0.0004\\
7	0.0004\\
8	0.0004\\
9	0.0004\\
10	0.0001\\
11	0.0001\\
12	0.0001\\
13	0.0001\\
14	0.0001\\
15	0.0001\\
16	0.0001\\
};
\addlegendentry{$\epsilon^{(i)}$}

\addplot[only marks, mark=x, mark options={}, mark size=2.5000pt, draw=mycolor2, forget plot] table[row sep=crcr]{%
x	y\\
16	8.426e-05\\
};
\end{axis}
\end{tikzpicture}%

%% file: fig/exp1/u.tex
%
%
\begin{tikzpicture}[baseline=(img)]
\node (img) at (0,0) {\includegraphics[width=\figurewidth]{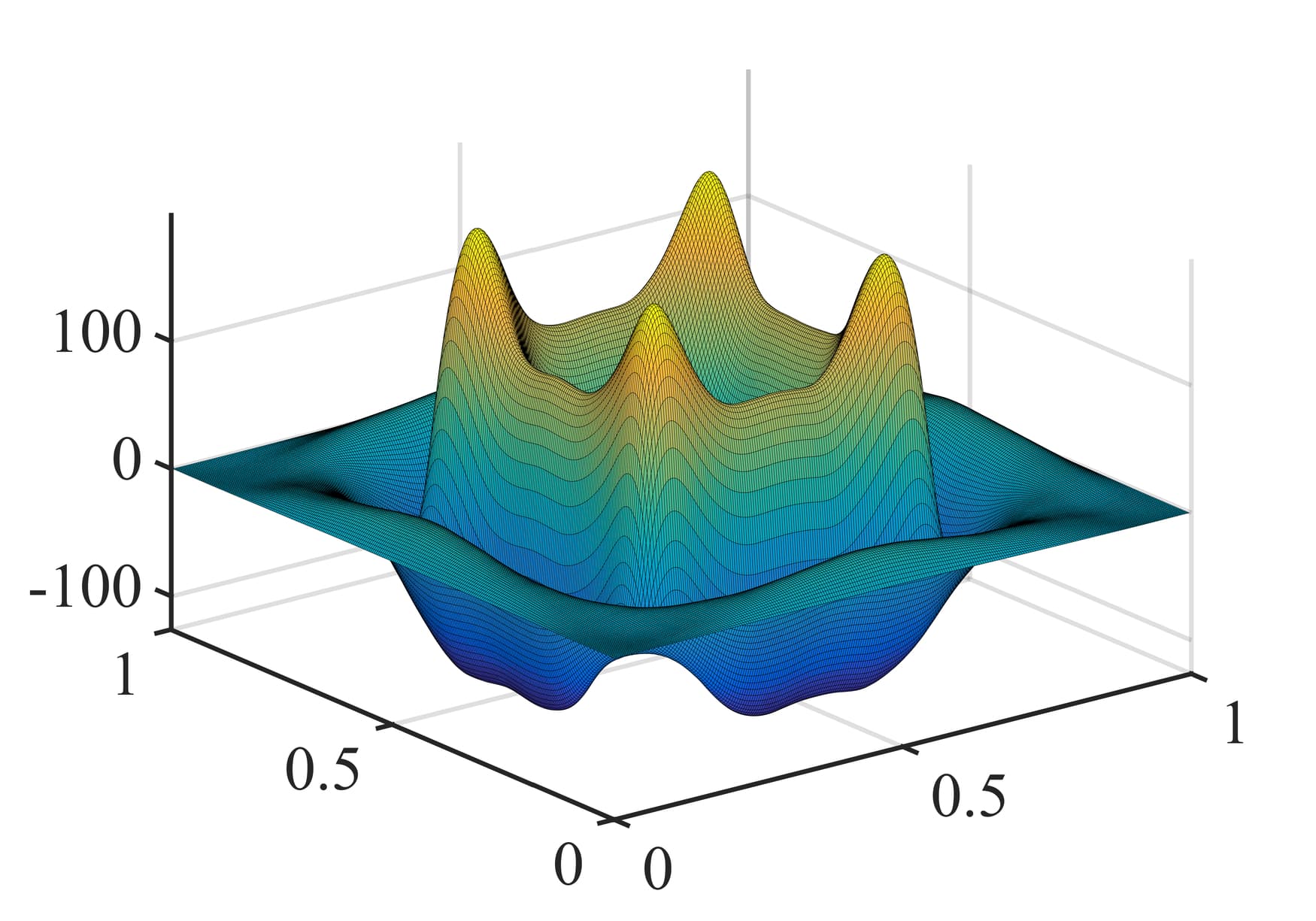}};
\useasboundingbox ([xshift=-0.08\figurewidth]current bounding box.south west) rectangle ([xshift=-0.08\figurewidth]current bounding box.north east);
\node [align=center] at (0.32\figurewidth,-0.3\figureheight){$x_1$};
\node [align=center] at (-0.35\figurewidth,-0.27\figureheight){$x_2$};
\node [align=center,rotate=90] at (-0.55\figurewidth,0.05\figureheight){$\opt{\bs{u}}$};
\end{tikzpicture}%

%% file: fig/exp1/g.tex
%
%
\begin{tikzpicture}[baseline=(img)]
\node (img) at (0,0) {\includegraphics[width=\figurewidth]{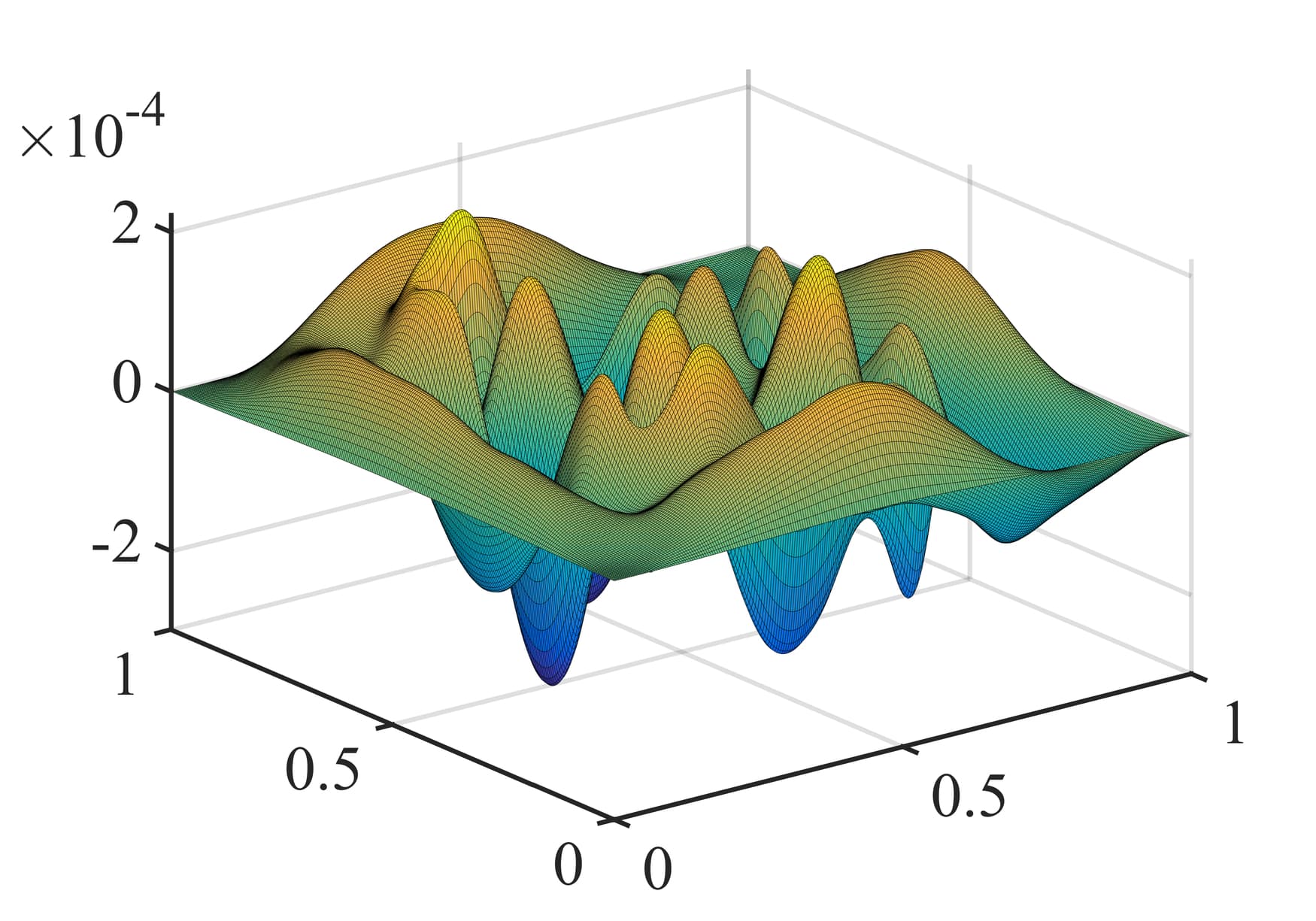}};
\useasboundingbox ([xshift=-0.08\figurewidth]current bounding box.south west) rectangle ([xshift=-0.08\figurewidth]current bounding box.north east);
\node [align=center] at (0.32\figurewidth,-0.3\figureheight){$x_1$};
\node [align=center] at (-0.35\figurewidth,-0.27\figureheight){$x_2$};
\node [align=center,rotate=90] at (-0.52\figurewidth,0.05\figureheight){$\nabla \rJc_\$(\opt{\bs{u}})$};
\end{tikzpicture}%

%% file: fig/exp1/Ey.tex
%
%
\begin{tikzpicture}[baseline=(img)]
\node (img) at (0,0) {\includegraphics[width=\figurewidth]{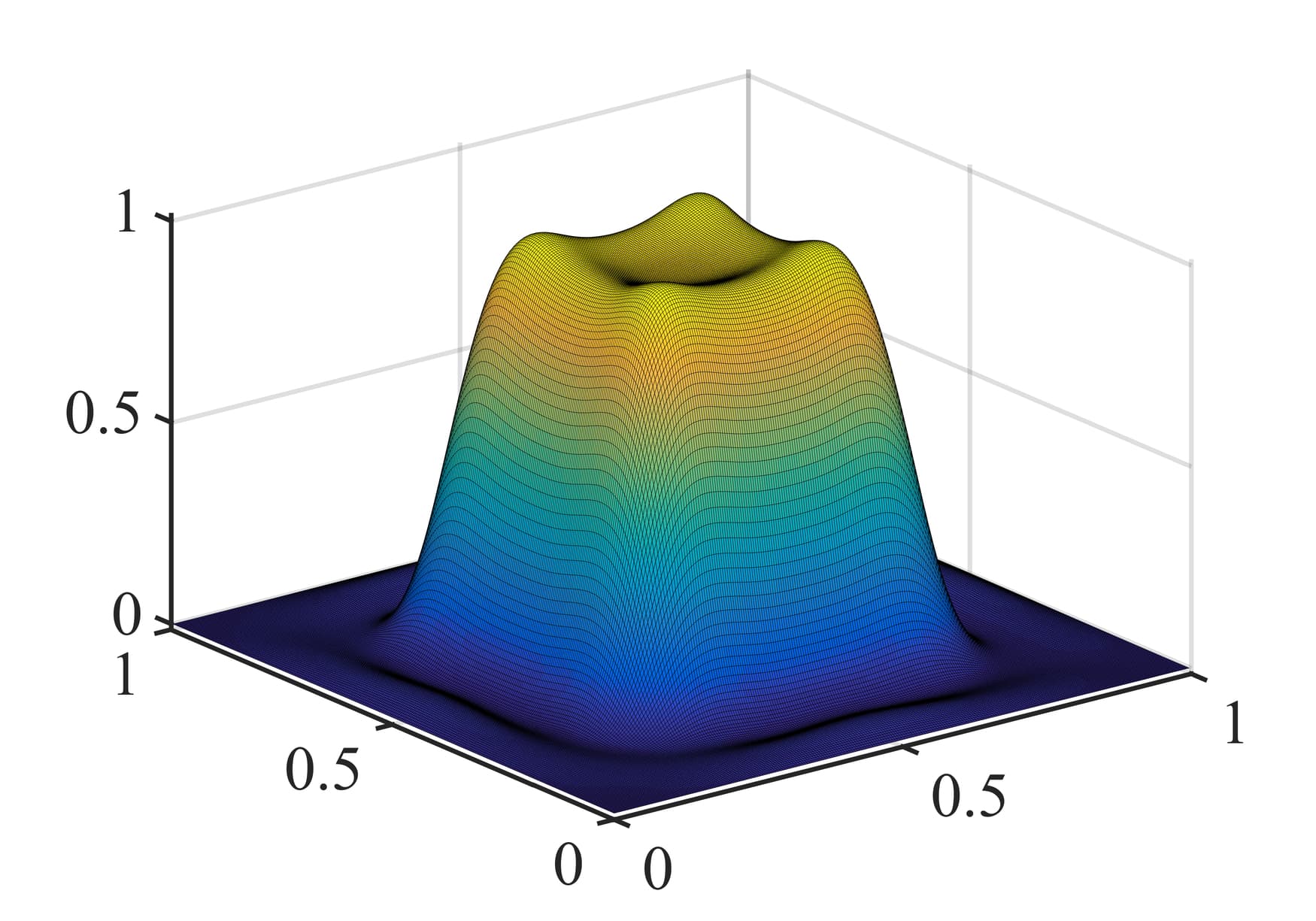}};
\useasboundingbox ([xshift=-0.08\figurewidth]current bounding box.south west) rectangle ([xshift=-0.08\figurewidth]current bounding box.north east);
\node [align=center] at (0.32\figurewidth,-0.3\figureheight){$x_1$};
\node [align=center] at (-0.35\figurewidth,-0.27\figureheight){$x_2$};
\node [align=center,rotate=90] at (-0.55\figurewidth,0.05\figureheight){$\mean{\opt{\bs{y}}}$};
\end{tikzpicture}%

%% file: fig/exp1/Vy.tex
%
%
\begin{tikzpicture}[baseline=(img)]
\node (img) at (0,0) {\includegraphics[width=\figurewidth]{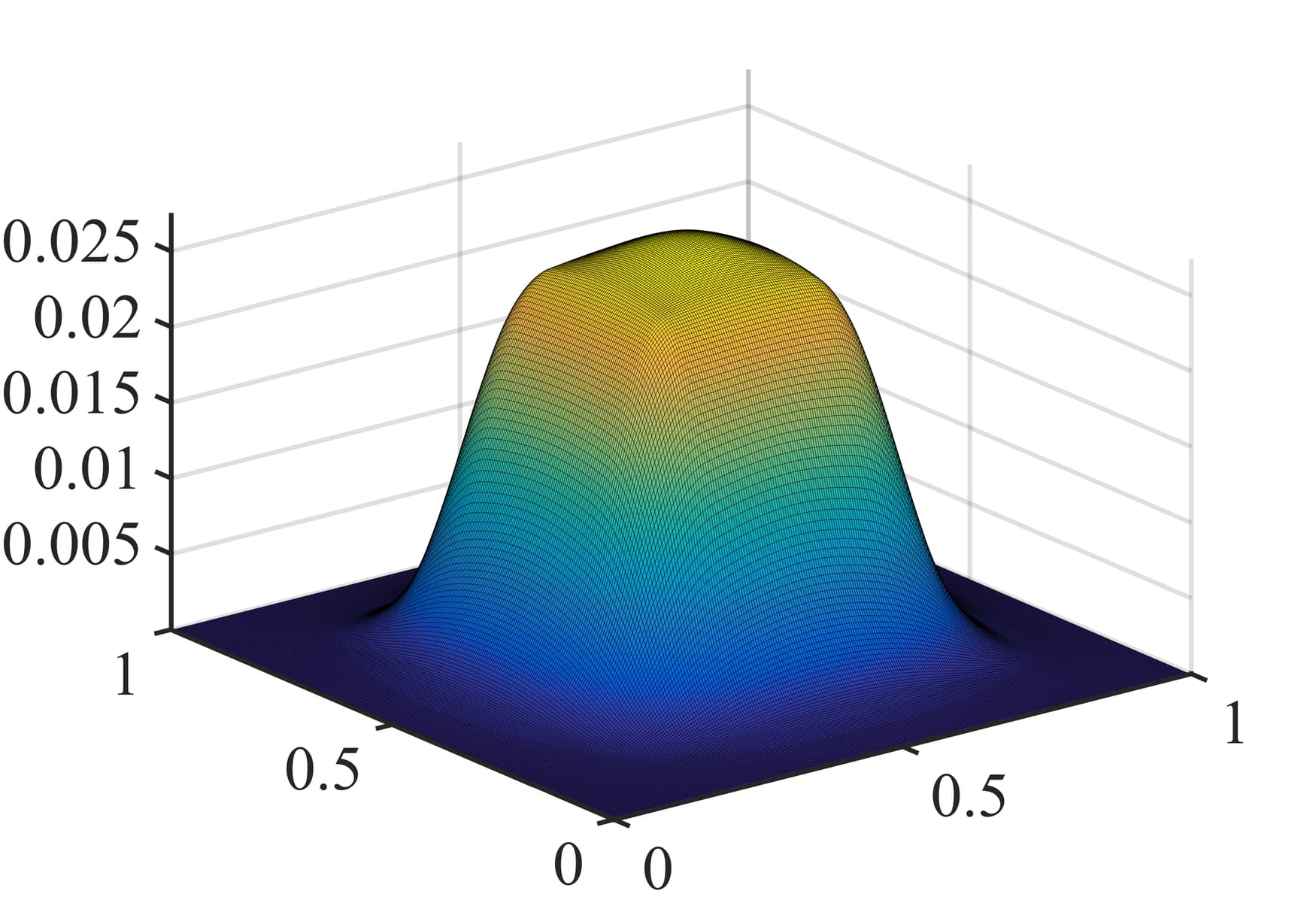}};
\useasboundingbox ([xshift=-0.08\figurewidth]current bounding box.south west) rectangle ([xshift=-0.08\figurewidth]current bounding box.north east);
\node [align=center] at (0.32\figurewidth,-0.3\figureheight){$x_1$};
\node [align=center] at (-0.35\figurewidth,-0.27\figureheight){$x_2$};
\node [align=center,rotate=90] at (-0.55\figurewidth,0.05\figureheight){$\var{\opt{\bs{y}}}$};
\end{tikzpicture}%

%% file: fig/exp2/u.tex
%
%
\begin{tikzpicture}[baseline=(img)]
\node (img) at (0,0) {\includegraphics[width=\figurewidth]{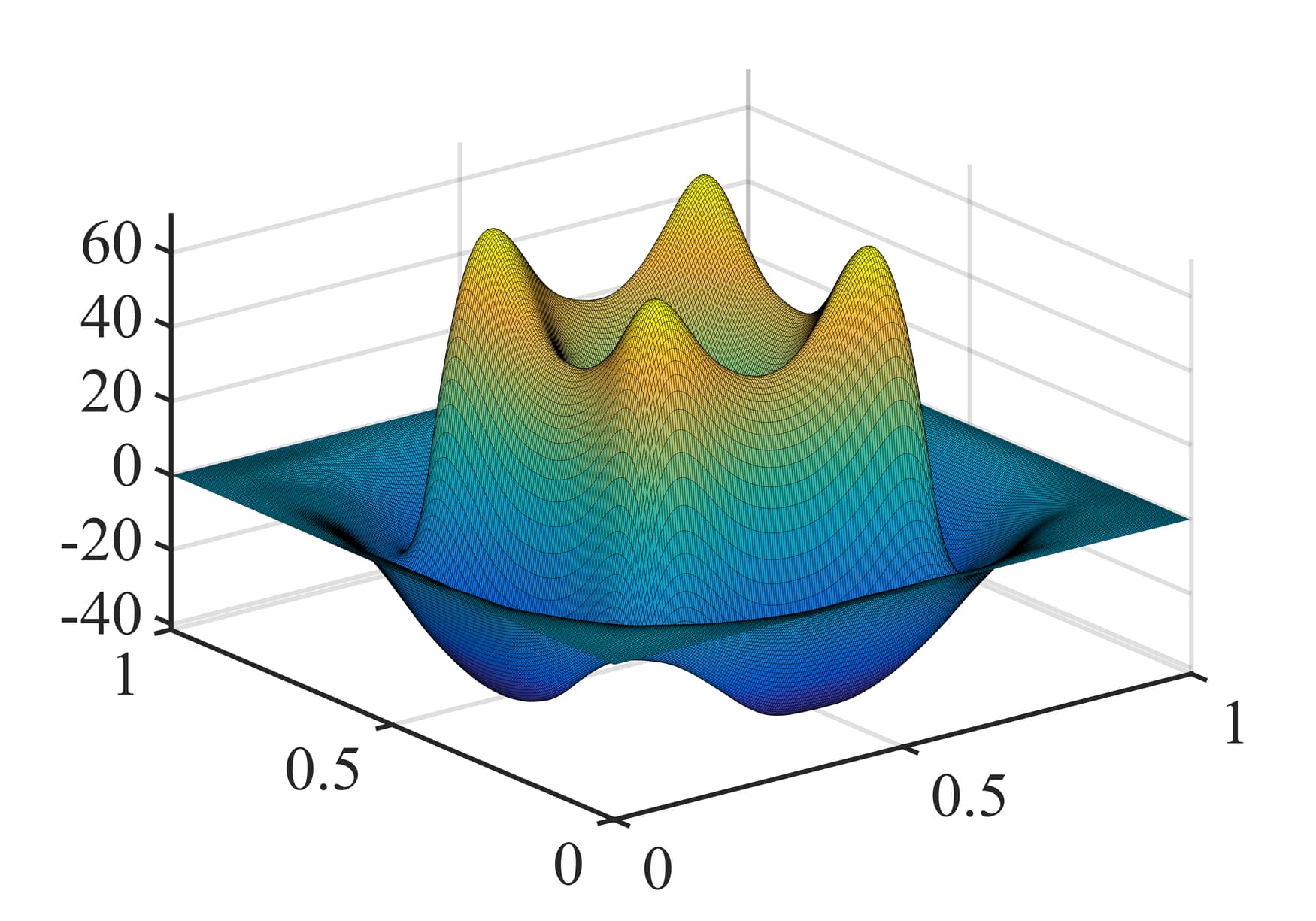}};
\useasboundingbox ([xshift=-0.08\figurewidth]current bounding box.south west) rectangle ([xshift=-0.08\figurewidth]current bounding box.north east);
\node [align=center] at (0.32\figurewidth,-0.3\figureheight){$x_1$};
\node [align=center] at (-0.35\figurewidth,-0.27\figureheight){$x_2$};
\node [align=center,rotate=90] at (-0.55\figurewidth,0.05\figureheight){$\opt{\bs{u}}$};
\end{tikzpicture}%

%% file: fig/exp2/g.tex
%
%
\begin{tikzpicture}[baseline=(img)]
\node (img) at (0,0) {\includegraphics[width=\figurewidth]{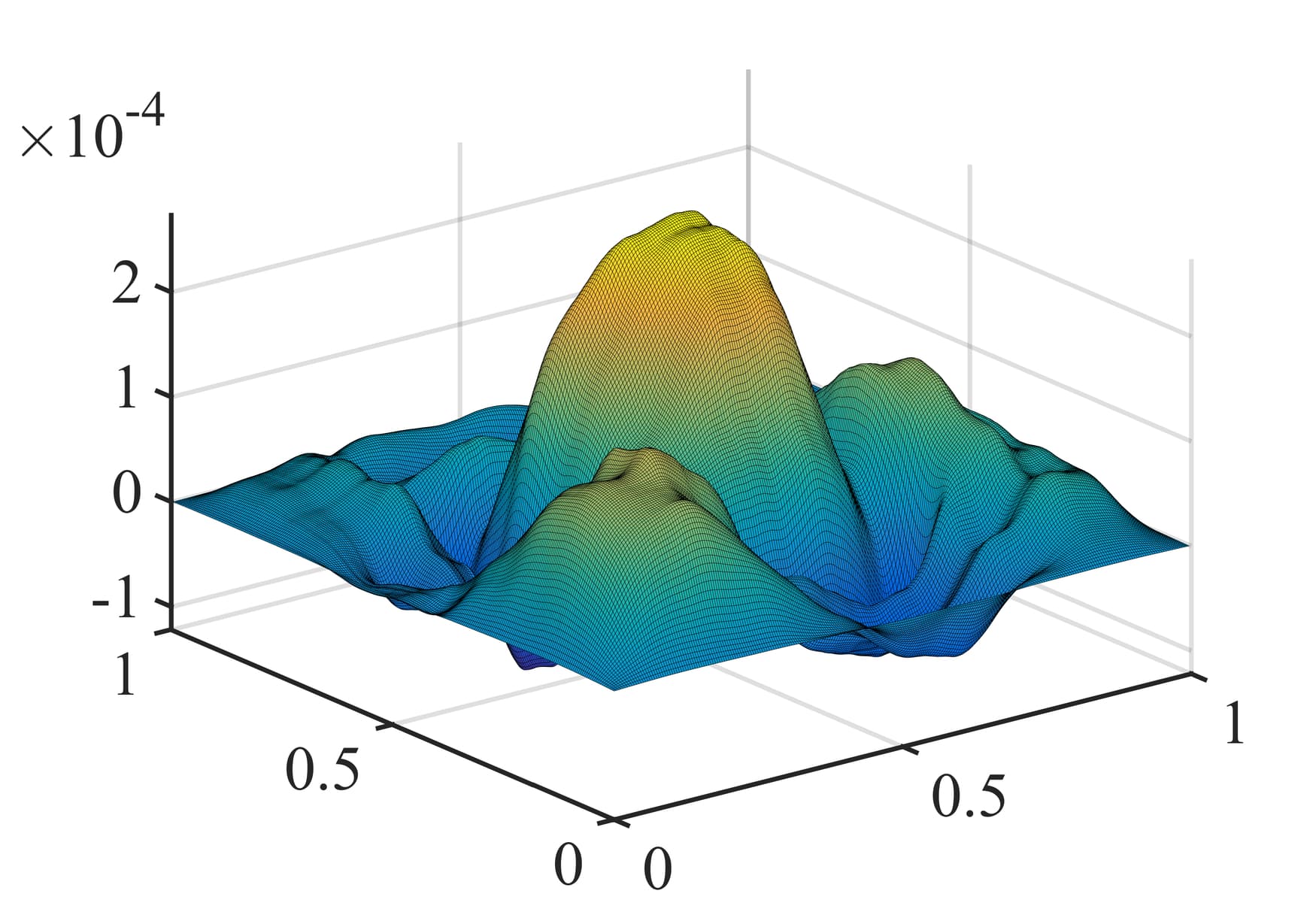}};
\useasboundingbox ([xshift=-0.08\figurewidth]current bounding box.south west) rectangle ([xshift=-0.08\figurewidth]current bounding box.north east);
\node [align=center] at (0.32\figurewidth,-0.3\figureheight){$x_1$};
\node [align=center] at (-0.35\figurewidth,-0.27\figureheight){$x_2$};
\node [align=center,rotate=90] at (-0.52\figurewidth,0.05\figureheight){$\nabla \rJc_\$(\opt{\bs{u}})$};
\end{tikzpicture}%

%% file: fig/exp2/Ey.tex
%
%
\begin{tikzpicture}[baseline=(img)]
\node (img) at (0,0) {\includegraphics[width=\figurewidth]{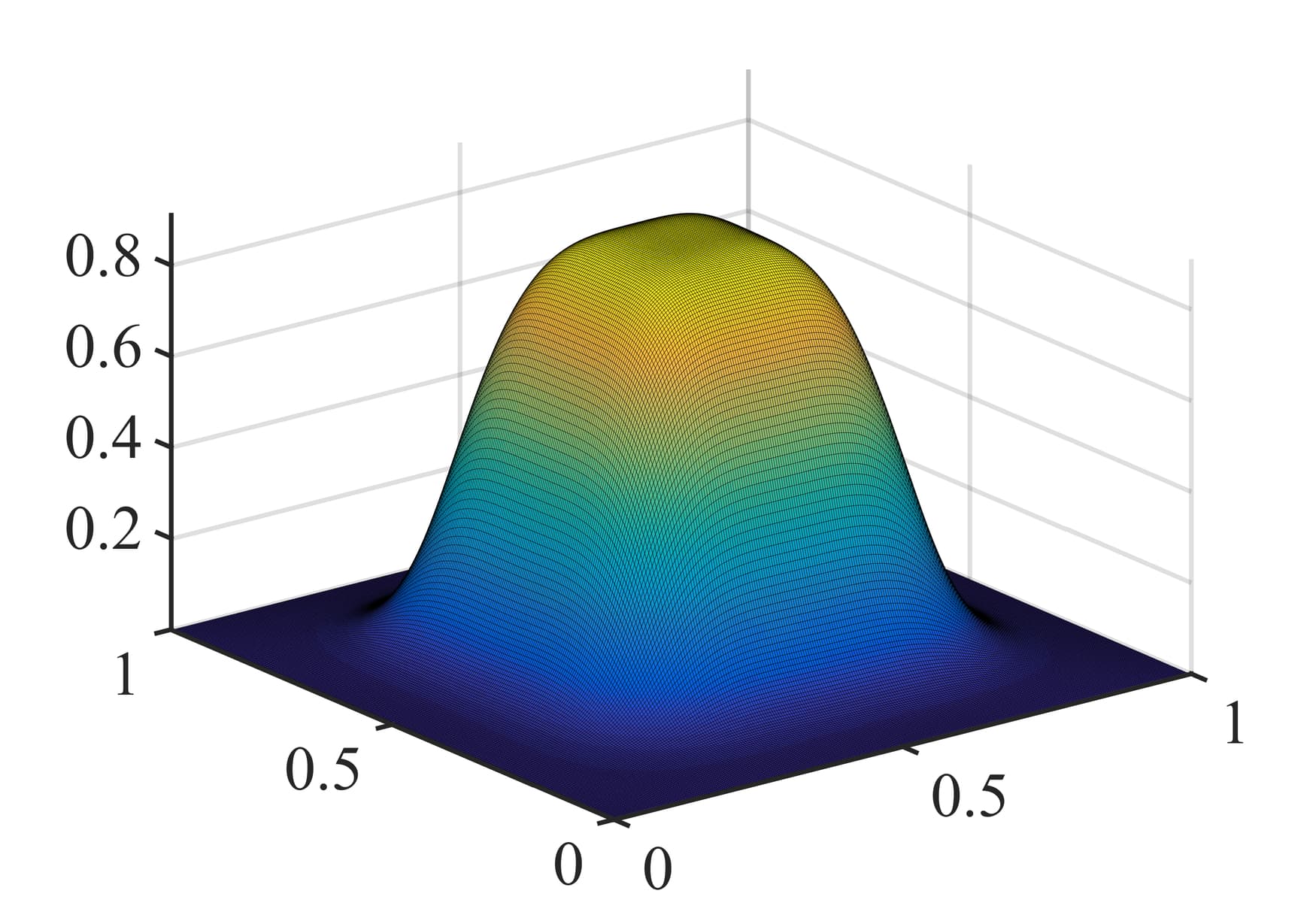}};
\useasboundingbox ([xshift=-0.08\figurewidth]current bounding box.south west) rectangle ([xshift=-0.08\figurewidth]current bounding box.north east);
\node [align=center] at (0.32\figurewidth,-0.3\figureheight){$x_1$};
\node [align=center] at (-0.35\figurewidth,-0.27\figureheight){$x_2$};
\node [align=center,rotate=90] at (-0.55\figurewidth,0.05\figureheight){$\mean{\opt{\bs{y}}}$};
\end{tikzpicture}%

%% file: fig/exp2/Vy.tex
%
%
\begin{tikzpicture}[baseline=(img)]
\node (img) at (0,0) {\includegraphics[width=\figurewidth]{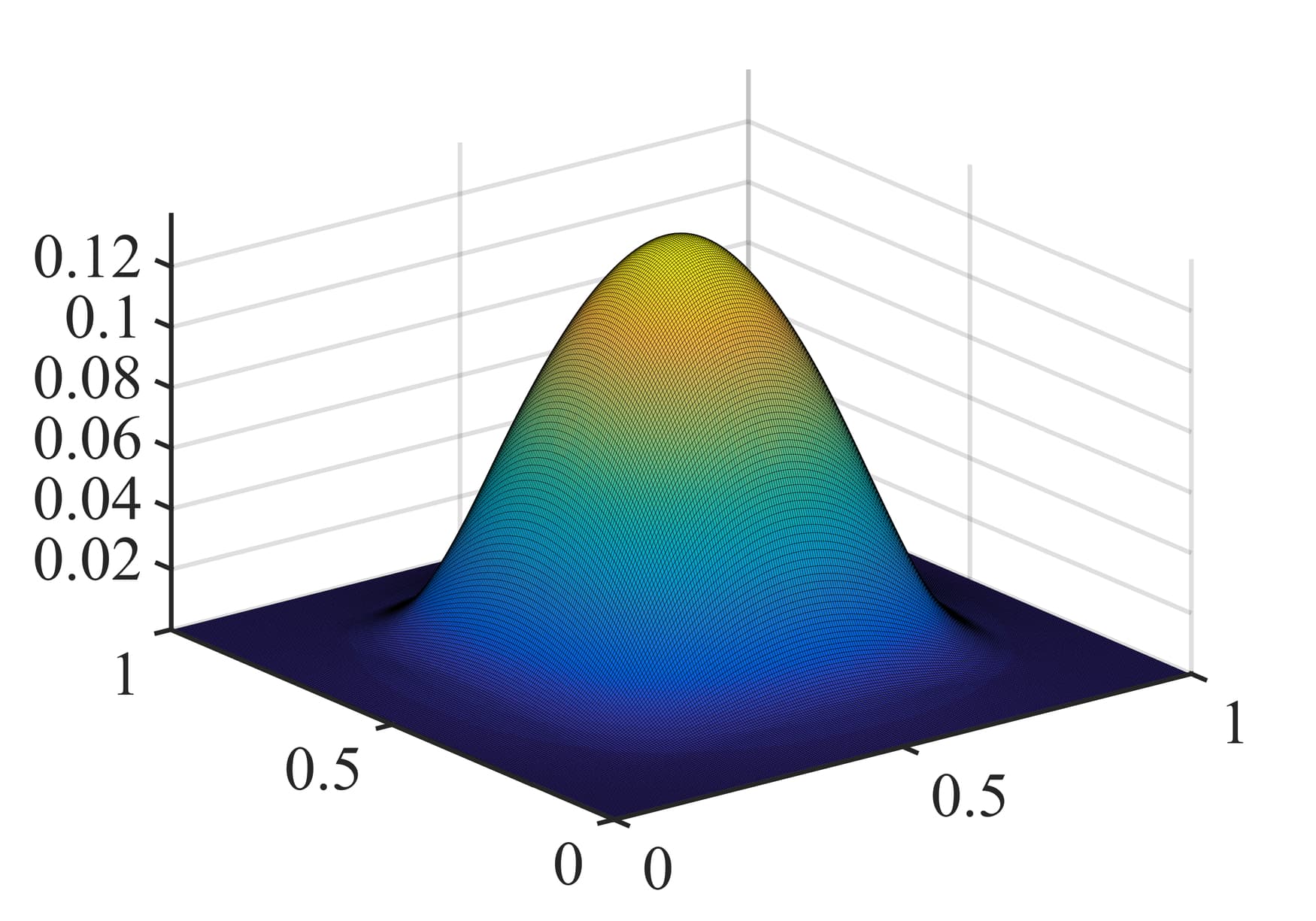}};
\useasboundingbox ([xshift=-0.08\figurewidth]current bounding box.south west) rectangle ([xshift=-0.08\figurewidth]current bounding box.north east);
\node [align=center] at (0.32\figurewidth,-0.3\figureheight){$x_1$};
\node [align=center] at (-0.35\figurewidth,-0.27\figureheight){$x_2$};
\node [align=center,rotate=90] at (-0.55\figurewidth,0.05\figureheight){$\var{\opt{\bs{y}}}$};
\end{tikzpicture}%

%% file: fig/exp3/u.tex
%
%
\begin{tikzpicture}[baseline=(img)]
\node (img) at (0,0) {\includegraphics[width=\figurewidth]{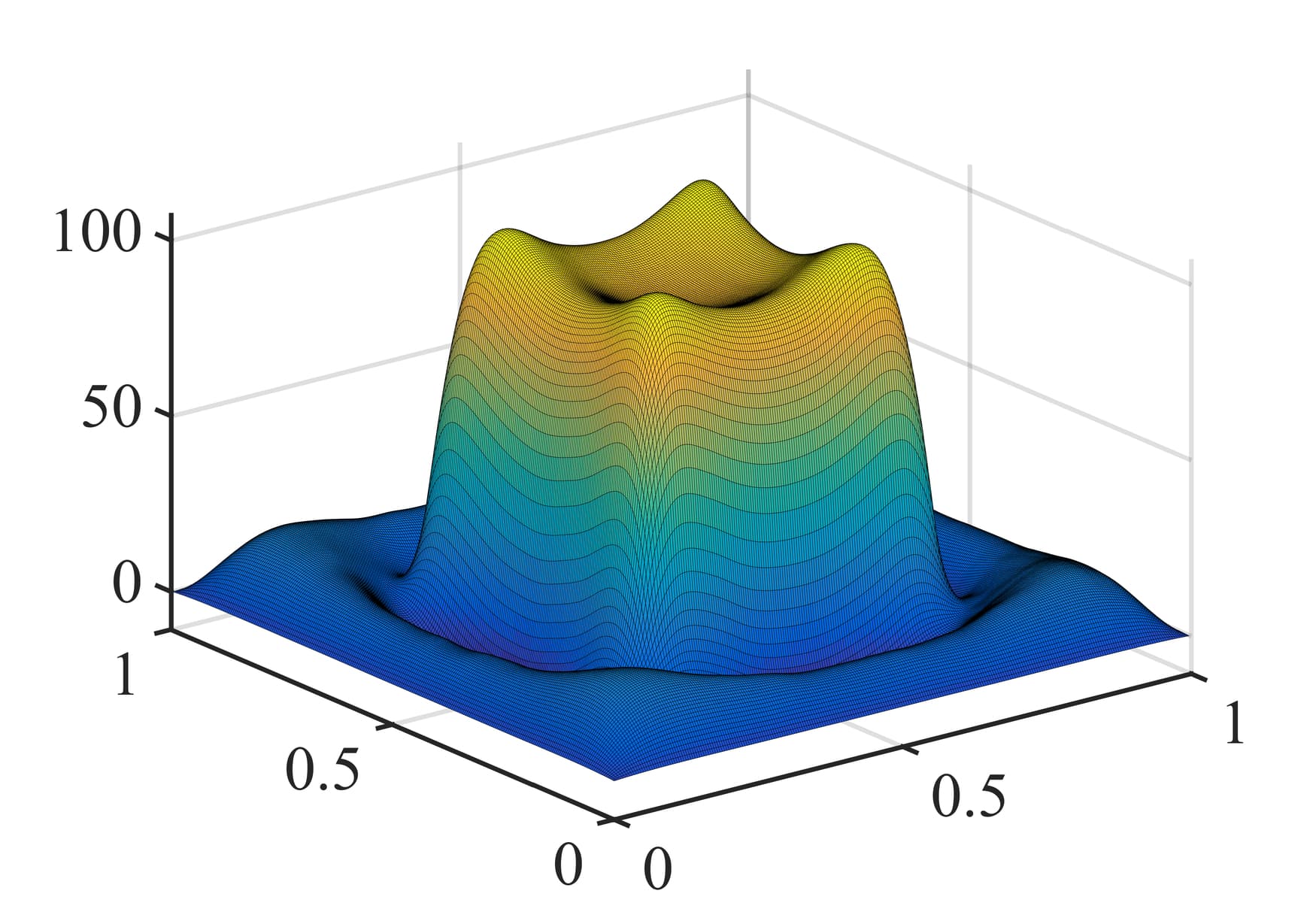}};
\useasboundingbox ([xshift=-0.08\figurewidth]current bounding box.south west) rectangle ([xshift=-0.08\figurewidth]current bounding box.north east);
\node [align=center] at (0.32\figurewidth,-0.3\figureheight){$x_1$};
\node [align=center] at (-0.35\figurewidth,-0.27\figureheight){$x_2$};
\node [align=center,rotate=90] at (-0.55\figurewidth,0.05\figureheight){$\opt{\bs{u}}$};
\end{tikzpicture}%

%% file: fig/exp3/g.tex
%
%
\begin{tikzpicture}[baseline=(img)]
\node (img) at (0,0) {\includegraphics[width=\figurewidth]{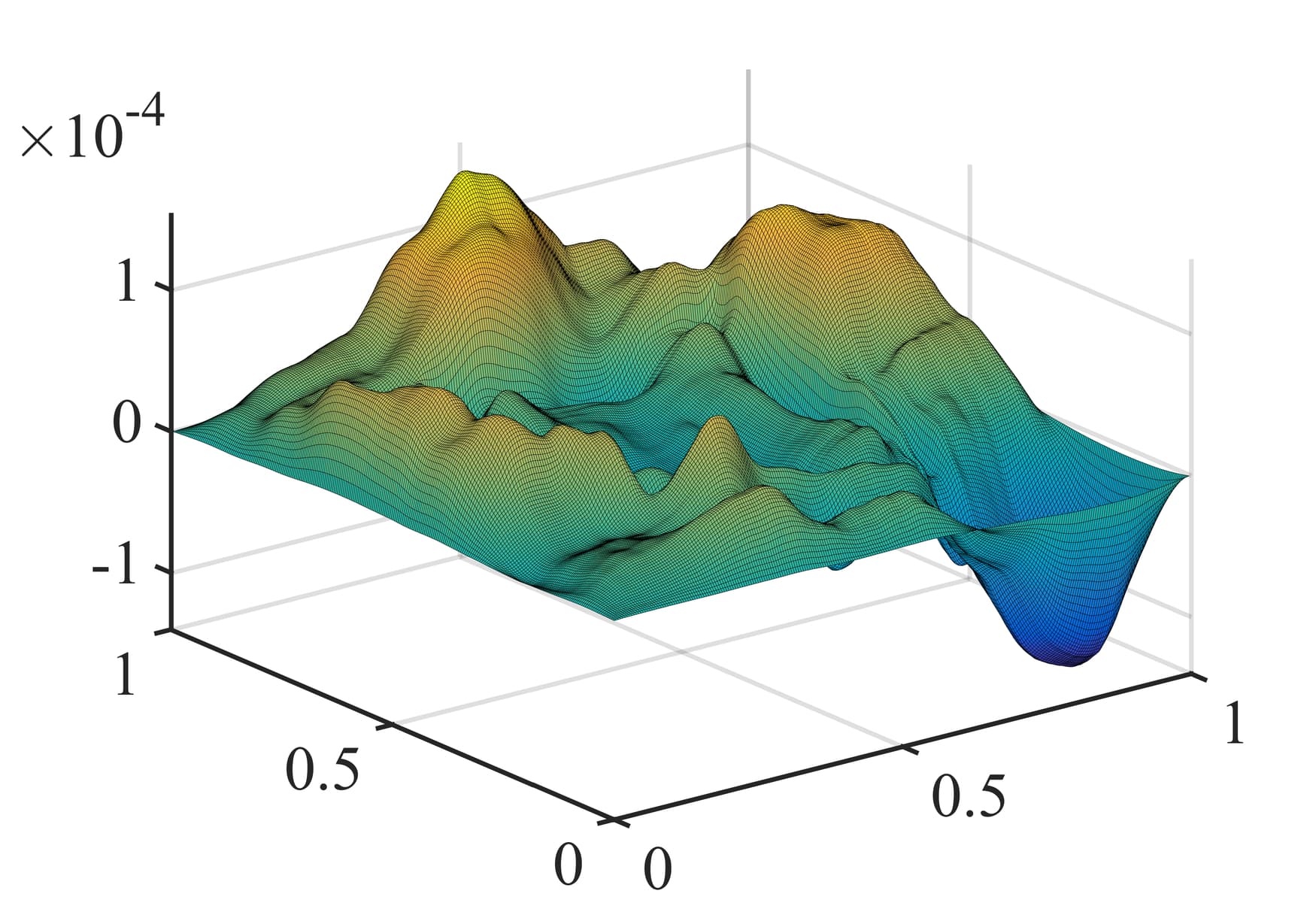}};
\useasboundingbox ([xshift=-0.08\figurewidth]current bounding box.south west) rectangle ([xshift=-0.08\figurewidth]current bounding box.north east);
\node [align=center] at (0.32\figurewidth,-0.3\figureheight){$x_1$};
\node [align=center] at (-0.35\figurewidth,-0.27\figureheight){$x_2$};
\node [align=center,rotate=90] at (-0.52\figurewidth,0.05\figureheight){$\nabla \rJc_\$(\opt{\bs{u}})$};
\end{tikzpicture}%

%% file: fig/exp3/Ey.tex
%
%
\begin{tikzpicture}[baseline=(img)]
\node (img) at (0,0) {\includegraphics[width=\figurewidth]{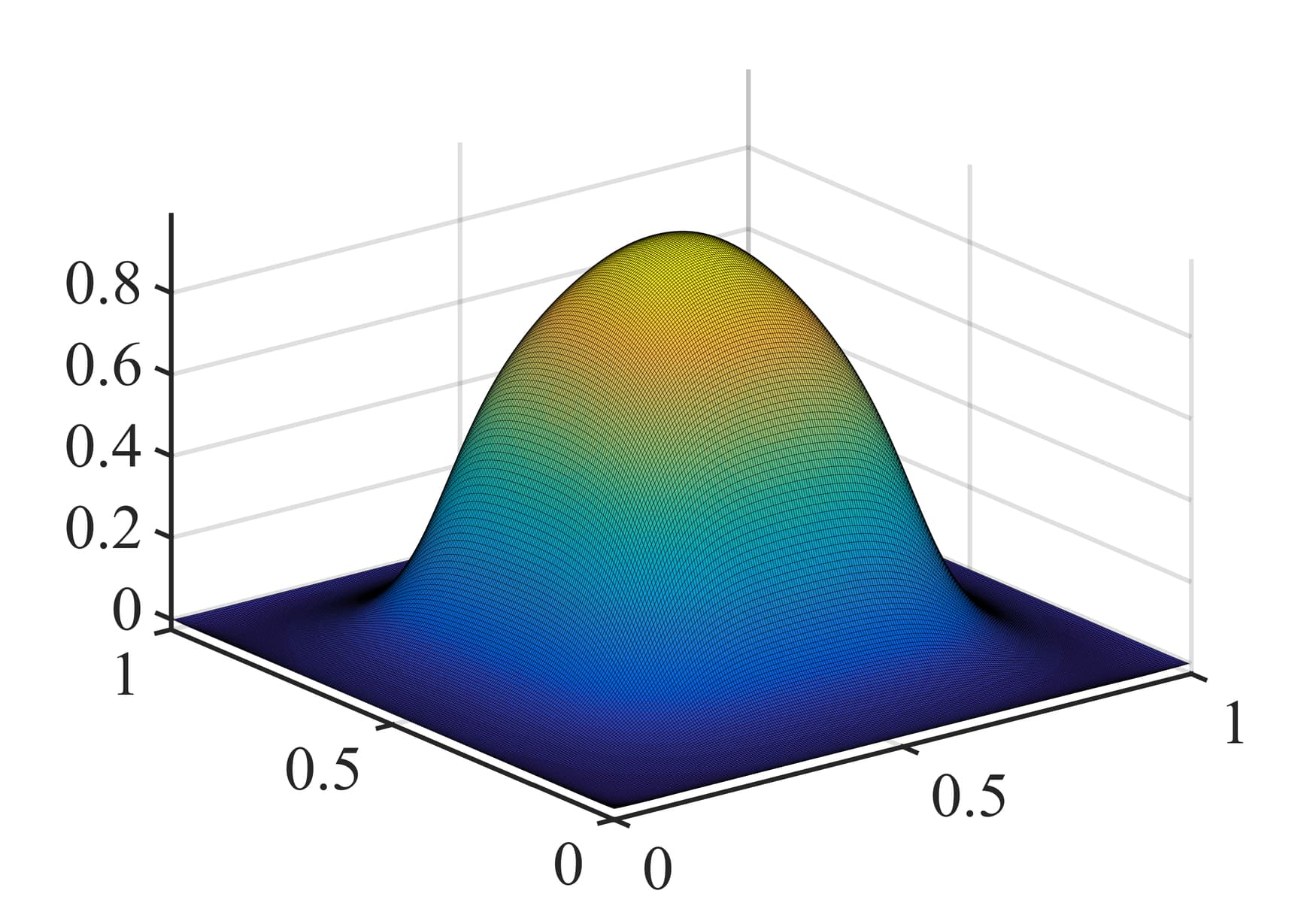}};
\useasboundingbox ([xshift=-0.08\figurewidth]current bounding box.south west) rectangle ([xshift=-0.08\figurewidth]current bounding box.north east);
\node [align=center] at (0.32\figurewidth,-0.3\figureheight){$x_1$};
\node [align=center] at (-0.35\figurewidth,-0.27\figureheight){$x_2$};
\node [align=center,rotate=90] at (-0.55\figurewidth,0.05\figureheight){$\mean{\opt{\bs{y}}}$};
\end{tikzpicture}%

%% file: fig/exp3/Vy.tex
%
%
\begin{tikzpicture}[baseline=(img)]
\node (img) at (0,0) {\includegraphics[width=\figurewidth]{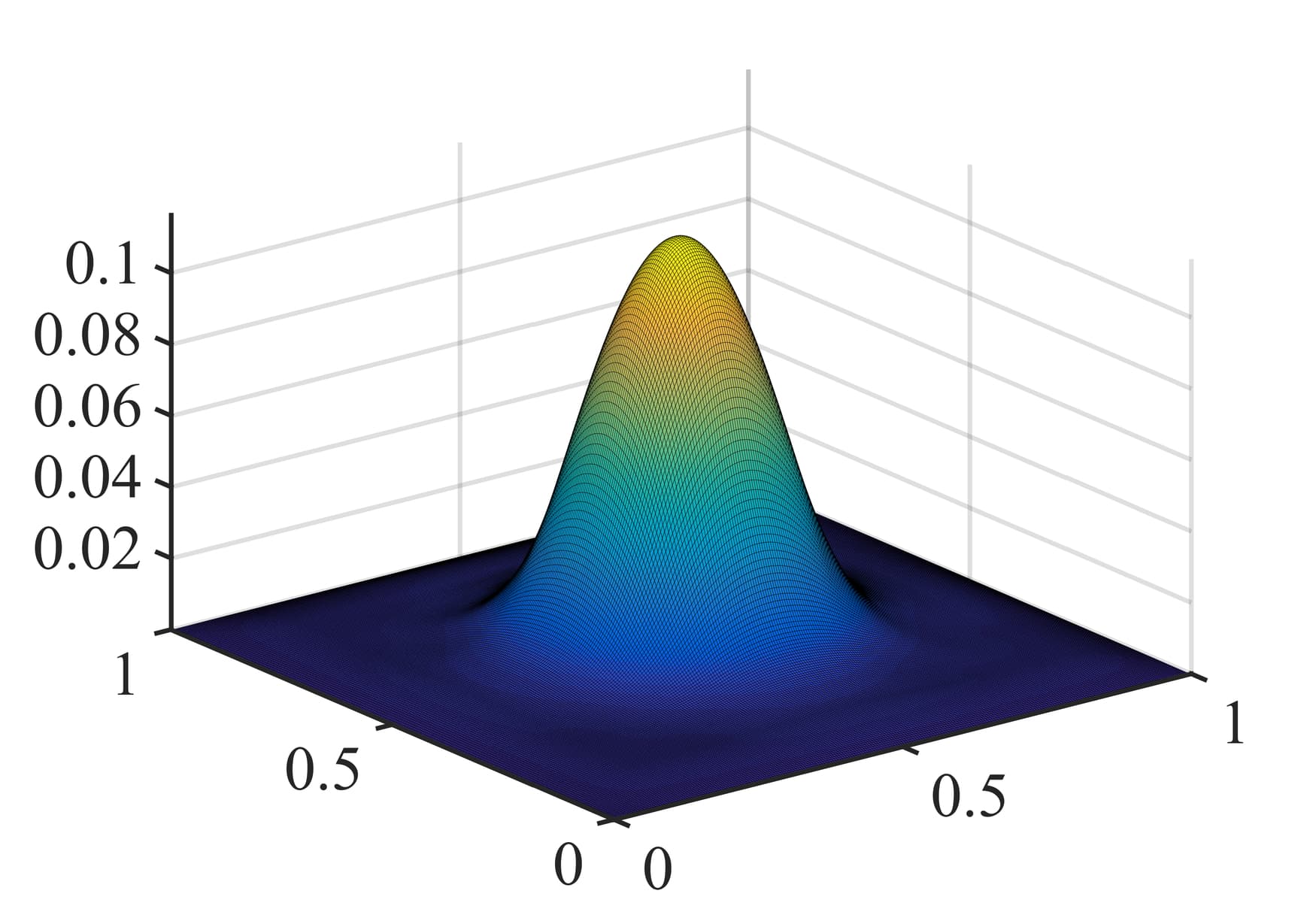}};
\useasboundingbox ([xshift=-0.08\figurewidth]current bounding box.south west) rectangle ([xshift=-0.08\figurewidth]current bounding box.north east);
\node [align=center] at (0.32\figurewidth,-0.3\figureheight){$x_1$};
\node [align=center] at (-0.35\figurewidth,-0.27\figureheight){$x_2$};
\node [align=center,rotate=90] at (-0.55\figurewidth,0.05\figureheight){$\var{\opt{\bs{y}}}$};
\end{tikzpicture}%

%% file: fig/exp3/conv.tex
%
%
\definecolor{mycolor1}{rgb}{0.00000,0.44700,0.74100}%
\definecolor{mycolor2}{rgb}{0.85000,0.32500,0.09800}%
\definecolor{mycolor3}{rgb}{0.92900,0.69400,0.12500}%
\definecolor{mycolor4}{rgb}{0.49400,0.18400,0.55600}%
\definecolor{mycolor5}{rgb}{0.46600,0.67400,0.18800}%
\definecolor{mycolor6}{rgb}{0.30100,0.74500,0.93300}%
\begin{tikzpicture}

\begin{axis}[%
width=\figurewidth,
height=\figureheight,
at={(0\figurewidth,0\figureheight)},
scale only axis,
xmin=0,
xmax=25,
xlabel style={font=\color{white!15!black}},
xlabel={iteration $k$ (NCG) or $i$ (CG)},
ymode=log,
ymin=1e-05,
ymax=1,
yminorticks=true,
axis background/.style={fill=white},
legend style={legend cell align=left, align=left, draw=white!15!black}
]
\addplot [color=mycolor1]
  table[row sep=crcr]{%
0	0.159145295458656\\
1	0.0114854880922617\\
2	0.00314547568761071\\
3	0.00107860485744996\\
4	0.000823768203957275\\
5	0.000756043867508636\\
6	0.000365831892493634\\
7	0.000235463642439068\\
8	0.000193473413044486\\
9	0.000143740098407997\\
10	8.5564537684863e-05\\
11	5.06442294592341e-05\\
12	3.67708132790995e-05\\
13	0\\
14	0\\
15	0\\
16	0\\
17	0\\
};
\addlegendentry{$\|\bs{g}^{(k)}\|$}

\addplot [color=mycolor1, dashed]
  table[row sep=crcr]{%
0	0.01\\
1	0.000527870728842245\\
2	0.000299444874202945\\
3	0.000368578411604147\\
4	0.000378432395075163\\
5	0.000414904821662803\\
6	0.000384220786094473\\
7	7.31663784987268e-05\\
8	6.20080016028859e-05\\
9	6.44770509970521e-05\\
10	6.14760906961756e-05\\
11	6.30533837216114e-05\\
12	5e-05\\
13	0\\
14	0\\
15	0\\
16	0\\
17	0\\
};
\addlegendentry{$\epsilon^{(k)}$}

\addplot[only marks, mark=x, mark options={}, mark size=2.5000pt, draw=mycolor1, forget plot] table[row sep=crcr]{%
x	y\\
0	0\\
1	0\\
2	0\\
3	0\\
4	0\\
5	0\\
6	0\\
7	0\\
8	0\\
9	0\\
10	0\\
11	0\\
12	4.66151665996688e-05\\
13	0\\
14	0\\
15	0\\
16	0\\
17	0\\
};
\addplot [color=mycolor2]
  table[row sep=crcr]{%
0	0.150370171516503\\
1	0.00165839949508266\\
2	0.0112614879437814\\
3	0.00138634093839479\\
4	0.00390013078018082\\
5	0.00180439557742853\\
6	0.000734902680360581\\
7	0.000402684159758143\\
8	0.000377742516300003\\
9	0.00177935344825797\\
10	0.000706759370360477\\
11	0.000192850407186331\\
12	0.000222498340560663\\
13	0.000104555424708768\\
14	8.17983163263399e-05\\
15	3.02067234515968e-05\\
16	0.000418227299924788\\
17	0.000283123886047705\\
18	6.49960230166316e-05\\
19	4.28816266064359e-05\\
20	0.000103774430911108\\
21	9.97791077124742e-05\\
22	3.49428509146829e-05\\
23	6.08319333729144e-05\\
24	3.71706851651653e-05\\
};
\addlegendentry{$\|\bs{r}^{(i)}\|$}

\addplot [color=mycolor2, dashed]
  table[row sep=crcr]{%
0	0.01\\
1	0.01\\
2	0.002\\
3	0.002\\
4	0.0004\\
5	0.0004\\
6	0.0004\\
7	0.0004\\
8	0.0004\\
9	8e-05\\
10	8e-05\\
11	8e-05\\
12	8e-05\\
13	8e-05\\
14	8e-05\\
15	8e-05\\
16	5e-05\\
17	5e-05\\
18	5e-05\\
19	5e-05\\
20	5e-05\\
21	5e-05\\
22	5e-05\\
23	5e-05\\
24	5e-05\\
};
\addlegendentry{$\epsilon^{(i)}$}

\addplot[only marks, mark=x, mark options={}, mark size=2.5000pt, draw=mycolor2] table[row sep=crcr]{%
x	y\\
24	2.972e-05\\
};
\end{axis}
\end{tikzpicture}%